\documentclass[12pt,a4paper]{amsart} 
\usepackage[english]{babel}

\usepackage[utf8]{inputenc}
\usepackage[T1]{fontenc}
\usepackage{lmodern}

\usepackage{amsthm}
\usepackage{amsmath}
\usepackage{amsfonts}
\usepackage{amssymb}

\usepackage{hyperref}

\usepackage[left=2.5cm,right=2.5cm,top=3cm,bottom=2cm]{geometry}

\newtheorem{theo}{Theorem}[subsection]
\newtheorem*{theo*}{Theorem}
\newtheorem{prop}[theo]{Proposition}
\newtheorem{lemm}[theo]{Lemma}
\newtheorem{coro}[theo]{Corollary}

\newtheorem{defi}[theo]{Definition}

\theoremstyle{remark}
\newtheorem{rema}[theo]{Remark}
\newtheorem{ex}[theo]{Example}

\newtheorem{thmx}{Theorem}

\newcommand{\R}{\mathbb{R}}
\newcommand{\Z}{\mathbb{Z}}
\newcommand{\Q}{\mathbb{Q}}
\newcommand{\F}{\mathbb{F}}
\newcommand{\eps}{\varepsilon}

\newcommand{\spec}{\operatorname{Spec}}
\newcommand{\tr}{\operatorname{tr}}
\newcommand{\Ext}{\operatorname{Ext}}
\newcommand{\GL}{\operatorname{GL}}

\newcommand{\m}{\mathfrak{m}}
\newcommand{\len}{\operatorname{len}}
\newcommand{\Gal}{\operatorname{Gal}}
\newcommand{\A}{\mathcal{A}}
\newcommand{\id}{\operatorname{id}}
\newcommand{\Fil}{\operatorname{Fil}}
\newcommand{\unr}{\operatorname{unr}}

\newcommand{\E}{\mathcal{E}}
\newcommand{\X}{\mathcal{X}}
\newcommand{\Y}{\mathcal{Y}}
\newcommand{\z}{\mathcal{Z}}
\newcommand{\D}{\mathcal{D}}
\newcommand{\V}{\mathcal{V}}
\newcommand{\f}{\mathcal{F}}
\newcommand{\M}{\mathcal{M}}

\newcommand{\cyc}{\chi_{cycl}}

\newcommand{\e}{\overline{e}}
\newcommand{\OO}{\mathcal{O}}
\newcommand{\PP}{\mathbb{P}}

\newcommand{\p}{\mathcal{P}}
\newcommand{\fp}{\mathfrak{p}}

\newcommand{\ind}{\operatorname{ind}}

\newcommand{\den}{\operatorname{denom}}

\newcommand{\mua}{\mu_{\text{aut}}}

\newcommand{\Dcr}[1]{\operatorname{D}^{#1}_{\text{crys}}}
\newcommand{\Dst}[1]{\operatorname{D}^{#1}_{\text{st}}}
\newcommand{\Ddr}[1]{\operatorname{D}^{#1}_{\text{dR}}}

\newcommand{\DcrO}{\operatorname{D}_{\text{crys},0}}
\newcommand{\DdRO}{\operatorname{D}_{\text{dR},0}}

\newcommand{\GF}{\Gal(F/\Q_p)}

\newcommand{\WD}{\operatorname{WD}}
\newcommand{\LL}{\mathcal{L}}
\newcommand{\ltau}{\mathcal{L}_{\tau}}
\newcommand{\Max}{\operatorname{Max}}

\newcommand{\matr}[4]{\left(\begin{smallmatrix}#1 & #2 \\ #3 & #4\end{smallmatrix}\right)}

\title[On the locus of $2$-dimensional crystalline representations]{%
On the locus of $2$-dimensional crystalline representations with a
given reduction modulo $p$}

\author{Sandra Rozensztajn}
\address{UMPA, \'ENS de Lyon\\
UMR 5669 du CNRS\\
46, all\'ee d'Italie\\
69364 Lyon Cedex 07\\
France}
\email{sandra.rozensztajn@ens-lyon.fr}


\begin{document}

\begin{abstract}
We consider the family of irreducible crystalline representations of
dimension $2$ of $\Gal(\bar\Q_p/\Q_p)$ given by the $V_{k,a_p}$ for a
fixed weight $k\geq 2$.
We study the locus of the parameter $a_p$ where these representations
have a given reduction modulo $p$. We give qualitative results on this
locus and show that for a fixed $p$ and $k$ it can be computed by
determining the reduction modulo $p$ of $V_{k,a_p}$ for a finite number
of values of the parameter $a_p$.
We also generalize these results to other Galois types.
\end{abstract}.

\maketitle

\setcounter{tocdepth}{1}
\tableofcontents

\section*{Introduction}

Let $p$ be a prime number. Fix a continuous representation $\bar\rho$ of
$G_{\Q_p} = \Gal(\bar\Q_p/\Q_p)$
with values in
$\GL_2(\bar\F_p)$. In \cite{Kis08}, Kisin has defined local rings
$R^{\psi}(k,\bar\rho)$ that parametrize the deformations of
$\bar\rho$ to characteristic $0$ representations that are crystalline
with Hodge-Tate weights $(0,k-1)$ and determinant $\psi$. These rings are
very hard to compute, even for relatively small values of $k$.
We are interested in this paper in the
rings $R^{\psi}(k,\bar\rho)[1/p]$. These rings lose some information from 
$R^{\psi}(k,\bar\rho)$, but still retain all the information about the
parametrization of deformations of $\bar\rho$ in characteristic $0$.

We can relate the study of the rings $R^{\psi}(k,\bar\rho)[1/p]$ to
another problem: When we fix an integer $k\geq 2$ and set the character
$\psi$ to be $\cyc^{k-1}$, the set of
isomorphism classes of irreducible crystalline representations of
dimension $2$, determinant $\psi$ and Hodge-Tate weights $(0,k-1)$ is in bijection
with the set $D = \{x\in\bar\Q_p, v_p(x)>0\}$ via a parameter $a_p$, and
we call $V_{k,a_p}$ the representation corresponding to $a_p$.
So given a residual representation $\bar\rho$ we can consider the set
$X(k,\bar\rho)$ of $a_p\in D$ such that the semi-simplified reduction
modulo $p$ of $V_{k,a_p}$ is equal to $\bar\rho^{ss}$.

It turns out that $X(k,\bar\rho)$ has a special form. We say that a
subset of $\bar\Q_p$ is a standard subset if it is a finite union of
rational open disks from which we have removed a finite union of
rational closed disks. Then we show that (when $\bar\rho$ has trivial
endomorphisms, so that the rings $R^{\psi}(k,\bar\rho)$ are
well-defined):

\begin{thmx}[Theorem \ref{bound} and Proposition \ref{defring1/p}]
\label{A}
The set $X(k,\bar\rho)$ is a standard subset of $\bar\Q_p$, and
$R^{\psi}(k,\bar\rho)[1/p]$ is the ring of bounded analytic functions on
$X(k,\bar\rho)$.
\end{thmx}

This tells us that we can recover $R^{\psi}(k,\bar\rho)[1/p]$ from
$X(k,\bar\rho)$. But we need to be able to understand $X(k,\bar\rho)$
better. 

We can define a notion of complexity for a standard subset $X$ which
is invariant under the absolute Galois group of $E$ for some finite extension $E$ of
$\Q_p$. This complexity is a positive integer $c_E(X)$, which mostly
counts the number of disks involved in the definition of $X$, but with
some arithmetic multiplicity that measures how hard it is to define the
disk on the field $E$. A consequence of this definition is
that if an upper bound for
$c_E(X)$ is given, then $X$ can be recovered from the sets $X\cap F$ for
some finitely many finite
extensions $F$ of $E$, and even from the intersection of $X$ with some
finite set of points under an additional hypothesis
(Theorems \ref{algo} and \ref{algoeps}).

A key point is that
this complexity, which is defined in a combinatorial way, is
actually related to the Hilbert-Samuel multiplicity of the special
fiber of the rings of analytic functions bounded by $1$ on the set $X$
(Theorem \ref{samecompl}). This is especially interesting in the case
where the set $X$ is $X(k,\bar\rho)$ as in this case this Hilbert-Samuel
multiplicity can be bounded explicity using the Breuil-Mézard conjecture.
So, when $\bar\rho$ has trivial endomorphisms and under some conditions
that ensure that the Breuil-Mézard conjecture is known in this case:

\begin{thmx}[Theorem \ref{bound}]
\label{B}
There is an explicit upper bound for the complexity of $X(k,\bar\rho)$.
\end{thmx}

As a consequence we get (with some additional conditions on $\bar\rho$,
that are satisfied for example when $\bar\rho$ is irreducible):

\begin{thmx}[Corollary \ref{algocrisbis}]
\label{C}
The set $X(k,\bar\rho)$ can be determined by computing the
reduction modulo $p$ of $V_{k,a_p}$ for $a_p$ in some finite set.
\end{thmx}

In particular, it is possible to compute the set $X(k,\bar\rho)$, and
also the ring $R^{\psi}(k,\bar\rho)[1/p]$, by a finite number of numerical
computations. We give some examples of this in Section \ref{examples}.
One interesting outcome of these computations is that when $\bar\rho$ is
irreducible, in every example that we computed we observed that the upper bound
for the complexity given by Theorem \ref{B} is actually an equality. 
It would be interesting to have an interpretation for this fact and to
know if it is true in general.

Finally, we could ask the same questions about more general rings
parametrizing potentially semi-stable deformations of a given Galois
type, instead of only rings parametrizing crystalline deformations. 
Our method relies on the fact that we work with rings that have relative
dimension $1$ over $\Z_p$, so we cannot use it beyond the case of
$2$-dimensional representations of $G_{\Q_p}$. But in this case we can
actually generalize our results to all Galois types.
In order to do this, we need to introduce a parameter classifying the
representations that plays a role similar to the role the function $a_p$
plays for crystalline representations, and to show that it defines an
analytic function on the rigid space attached to the deformation ring.
This is the result of Theorem
\ref{parameter}. Once we have this parameter, we show that an
analogue of Theorem \ref{A} holds, and an analogue of Theorem
\ref{B} (Theorem \ref{bound}). However we get only a weaker analogue of
Theorem \ref{C} (Theorem \ref{algogen}). 
The main ingredient of this theorem that is known in the crystalline
case, but missing in the
case of more general Galois types, is the fact that the reduction of the
representation is locally constant with respect to the parameter $a_p$,
with an explicit radius for local constancy. 

\subsection*{Plan of the article}

The first three sections contain some preliminaries.
In Section \ref{disks} we prove some results on the smallest degree of an
extension generated by a point of a disk in $\bar{\Q}_p$. These results may be
of independent interest.
In Section \ref{commalg} we prove some results on Hilbert-Samuel
multiplicities and how to compute them for some special rings of
dimension $1$.
In Section \ref{geometry} we introduce the notion of a standard subset of
$\PP^1(\bar\Q_p)$ and prove some results about
some special rigid subspaces of the affine line.

Section \ref{complexity} contains the main technical results. This is
where we introduce the complexity of so-called standard subsets of
$\PP^1(\bar\Q_p)$, and show that it can be defined in either a combinatorial or an
algebraic way. 

We apply these results in Section \ref{potcrys} to the locus of points
parametrizing potentially semi-stable representations of a fixed Galois
type with a given reduction. We also explain some particularities of the
case of parameter rings for crystalline representations.

In Section \ref{examples} we report on some numerical computations that were made
using the results of Section \ref{potcrys} in the case of crystalline
representations, and mention some questions inspired by these
computations.

Finally in Section \ref{paramsect} we explain the construction of a
parameter classifying the representations on the potentially semi-stable
deformation rings.

\subsection*{Notation}

If $E$ is a finite extension of $\Q_p$, we denote its ring of integers by
$\OO_E$, with maximal ideal $\m_E$, and its residue field by $k_E$. We
write $\pi_E$ for a uniformizer of $E$, and $v_E$ for the
valuation on $E$ normalized so that $v_E(\pi_E) = 1$ and its extension to
$\bar{\Q}_p$. 
We write also write $v_p$ for $v_{\Q_p}$.
Finally, $G_E$ denotes the absolute Galois group of $E$.  

If $R$ is a ring and $n$ a positive integer, we denote by $R[X]_{<n}$ the
subspace of $R[X]$ of polynomials of degree at most $n-1$.

If $a\in\bar{\Q}_p$ and $r\in \R$, we write $D(a,r)^+$ for the set
$\{x\in\bar{\Q}_p,
|x-a| \leq r\}$ (closed disk) and $D(a,r)^-$ for the set $\{x\in\bar{\Q}_p,
|x-a| < r\}$ (open disk).

We denote by $\cyc$ the $p$-adic cyclotomic character, and $\omega$ its
reduction modulo $p$. We denote by $\unr(x)$ the unramified character
that sends a geometric Frobenius to $x$.

\section{Points in disks in extensions of the base field}
\label{disks}

\begin{defi}
\label{definedover}
Let $X$ be a subset of $\bar\Q_p$, and let $E$ be an algebraic extension of
$\Q_p$. We say that $X$ is defined over $E$ if it is invariant by the
action of $G_E$.
\end{defi}

Let $D \subset \bar\Q_p$ be a disk (open or closed). It can happen that $D$
is defined over a finite extension $E$ of $\Q_p$,
but $E \cap D$ is empty. 
For example, let $\pi$ be a $p$-th root of $p$ and let $D$ be the
disk $\{x, v_p(x-\pi) > 1/p\}$. Then $D$ is defined over $\Q_p$, as it
contains all the conjugates of $\pi$, that is, the $\zeta_p^i\pi$ for a
primitive $p$-th root $\zeta_p$ of $1$. On the other hand, $D$ does not
contain any element of $\Q_p$. The goal of this section is to understand
the relationship between the smallest ramification degree over $E$ of an
extension field $F$ such that $F \cap D \neq\emptyset$, and the smallest degree
over $E$ of such a field $F$.

In this Section a disk will mean either a closed or an open disk.

The results of this Section are used in the proofs of Propositions
\ref{pointexistsin} and \ref{pointexistsout}.

\subsection{Statements}

\begin{theo}
\label{illdeftheo}
Let $E$ be a finite extension of $\Q_p$.
Let $D$ be a disk defined over $E$. 
Let $e$ be the smallest integer such
that there exists a finite extension $F$ of $E$ with $e_{F/E}=e$ and 
$F \cap D \neq \emptyset$.
Then $e=p^s$ for some $s$, and there exists an extension $F$ of $E$ with
$[F:E] \leq \max(1,p^{2s-1})$ such that $F \cap D \neq \emptyset$.
For $s \leq 1$ any such $F/E$ is totally ramified.
\end{theo}

We can in fact do better in the case where $p=2$. Note that this result
proves Conjecture 2 of \cite{Ben} in this case.

\begin{theo}
\label{illdeftheo2}
Let $p=2$.
Let $E$ be a finite extension of $\Q_p$.
Let $D$ be a disk defined over $E$. 
Let $e$ be the smallest integer such
that there exists a finite extension $F$ of $E$ with $e_{F/E}=e$ and 
$F \cap D \neq \emptyset$.
Then $e=p^s$ for some $s$, and there exists a totally ramified
extension $F$ of $E$ with
$[F:E] = p^s$ such that $F \cap D \neq \emptyset$.
\end{theo}

\subsection{Preliminaries}

We recall the following result, which is \cite[Lemma 3.6]{Ben} (it is
stated only for closed disks, but applies also to open disks).

\begin{lemm}
\label{Benedetto}
Let $K$ be an algebraic extension of $\Q_p$.
Let $D$ be a disk defined over $K$. Suppose that $D$ contains an
$a \in \bar\Q_p$ of degree $n$ over $K$. Then $D$ contains an element
$b\in\bar\Q_p$ of degree $\leq p^s$ over $K$ where $s = v_p(n)$.
\end{lemm}

\begin{coro}
\label{mindeg}
Let $K$ be an algebraic extension of $\Q_p$.
Let $D$ be a disk defined over $K$. Suppose that $D$ contains an
element $a$ such that $[K(a):K] = n$. Then the minimal degree over $K$ of
an element of $D$ is of the form $p^t$ for some $t \leq v_p(n)$.
\end{coro}

\begin{proof}
It follows from Lemma \ref{Benedetto} that the minimal degree over $K$ of
an element of $D$ is a power of $p$. On the other hand,
applying Lemma \ref{Benedetto} to $a$, 
we get an element of degree at most $p^s$ for $s = v_p(n)$. 
Hence the minimal degree is of the form $p^t$ for some $t \leq s$.
\end{proof}

\begin{coro}
\label{corBen}
Let $E$ be a finite extension of $\Q_p$.
Let $D$ be a disk defined over $E$. Then the minimal ramification
degree over $E$ of an element of $D$ is a power of $p$, and it can be
reached for an element $a$ such that $[E(a):E]$ is a power of $p$.
\end{coro}

\begin{proof}
We first apply Corollary \ref{mindeg} with $K = E^{unr}$ to see that the
minimal ramification degree is a power of $p$. Let $b\in D$ be such that
$e_{E(b)/E} = p^t$ is the minimal ramification degree.

Let $E(b)_0 = E^{unr}\cap E(b)$, and 
let $F$ be the maximal extension of $E$ contained in $E(b)_0$
such that $[F:E]$ is a power of $p$. 
Note that $v_p([E(b):F]) = t$, as $[E(b)_0:F]$ 
is prime to $p$. We apply Corollary \ref{mindeg} to $K=F$, and
we get an element $a\in D$ of degree at most $p^t$ over $F$. By
minimality of $t$, we get that in fact $[F(a):F] = p^t$, and $F(a)/F$ is
totally ramified. As $[F(a):E]$ is a power of $p$, so is $[E(a):E]$,
and $e_{E(a)/E} = p^t$.
\end{proof}

Let $\pi_E$ be a uniformizer of $E$, and let $F$ be a finite unramified
extension of $E$. For $x\in F$, we define the $E$-part of $x$, which we
denote by $x^0$, as follows:
we write $x$
as $x = \sum_{n\geq N}a_n\pi_E^n$ where the $a_n$ are Teichmueller lifts of
elements of the residue field of $F$.
Let $x^0 = \sum_{n=N}^ma_n\pi_E^n$ with $a_n\in E$ for all $n\leq m$ and
$a_{m+1}\not\in E$ (or $m=\infty$ if $a\in E$) so that $x^0\in E$. We
have that $v_E(x-x^0) = m+1$. This definition depends on the choice of
$\pi_E$.

\begin{prop}
\label{illdefunram}
Let $D$ be a disk defined over $E$, and suppose that $F \cap D \neq \emptyset$ for some
unramified extension $F$ of $E$. Then $E \cap D \neq \emptyset$.
\end{prop}

\begin{proof}
Let $a \in F \cap D$. We fix $\pi_E$ a uniformizer of $E$, and 
let $a^0$ be the $E$-part of $a$.
Let $\sigma$ be the Frobenius of
$\Gal(F/E)$. Then $v_E(a-\sigma(a)) = v_E(a-a^0)$.
So any disk containing $a$ and $\sigma(a)$ also contains $a^0$.
\end{proof}

We also recall the well-known result:

\begin{lemm}
\label{imdisk}
Let $f \in \bar\Q_p(X)$ be a rational fraction with indeterminate $X$.
Then for any disk $D$, if $f$ does not have a pole in $D$, then $f(D)$ is
also a disk. Moreover, if $D$ is defined over $E$ and $f\in E(X)$, then
$f(D)$ is defined over $E$.
\end{lemm}

\subsection{Proofs}

The part that states that $e$ is a power of $p$ in Theorems
\ref{illdeftheo} and \ref{illdeftheo2} is a consequence of Corollary
\ref{corBen}.

We start with the rest of the proof of Theorem \ref{illdeftheo2} which is actually
easier.

\begin{proof}[Proof of Theorem \ref{illdeftheo2}]
By applying Corollary \ref{corBen}, we get an element $a\in D$ that
generates a totally ramified extension $F$ of $K$ of degree $e=p^s$, where
$K$ is an unramified extension of $E$ of degree a power of $p$, and we
take $[K:E]$ minimal. If $K\neq
E$, let $K' \subset K$ with $[K:K'] = p$. We will show that we can find $b\in
D$ of degree $e$ over $K'$, which gives a contradiction by minimality of
$K$ so in fact $K=E$.

Let $\mu$ be the minimal polynomial of $a$ over $K$, so $\mu\in K[X]$ is
monic of degree $e$. Now we use that $p=2$: let $(1,u)$ be a basis of $K$ over
$K'$, and write $\mu = \mu_0 + u\mu_1$ with $\mu_0$, $\mu_1$ in $K'[X]$.
If $\mu_0$ has a root in $D$ we are finished, so we can assume that
$\mu_0$ has no zero in $D$, and let $f = \mu_1/\mu_0 \in K'(X)$. Let $D' =
f(D)$. It is a disk defined over $K'$, containing $-u\in K$, so by
Proposition
\ref{illdefunram}, $D'$ contains an element $c\in K'$. This means that
$\mu_0-c\mu_1$ has a root $b$ in $D$.

Then $b$ is of degree at most $e$ over $K'$. By minimality of $e$, it
means that $b$ is of degree exactly $e$ over $K'$, and $K'(b)/K'$ is totally
ramified. So this gives the contradiction we were looking for.
\end{proof}

Now we turn to the proof of Theorem \ref{illdeftheo}. We first prove the
result when we assume an additional condition.

\begin{prop}
\label{illdefeasycase}
Let $D$ be a disk defined over $E$ and  $a\in D$.
Suppose that $v_E(a) = n/e$ where $e =
e_{E(a)/E}$ and $n$ is prime to $e$. Then there exists an extension $F$ of $E$ of
degree at most $e$ such that $F \cap D \neq\emptyset$.
\end{prop}

\begin{proof}
Let $K = E(a) \cap E^{unr}$. Let $\mu$ be the minimal polynomial of $a$ over
$K$, so that $\mu$ has degree $e$. We write $\mu = \sum b_iX^i$, $b_i\in
K$. Define $\mu^0 = \sum b_i^0X^i$ where $b_i^0\in E$ is the $E$-part of
$b_i$.
Let $x_1,\dots,x_e$ be the roots
of $\mu^0$. Then $v_E(\mu^0(a)) = \sum_{i=1}^ev_E(a-x_i)$. On the other hand,
$\mu^0(a) = \mu^0(a) - \mu(a) = \sum_{i=0}^{e-1}(b_i^0-b_i)a^i$. By the
condition on $v_E(a)$, we get that 
$v_E(\mu^0(a)) = \min_{0\leq i <e}(v_E(b_i^0-b_i)+ in/e)$. 
Let $\sigma$ be an element of $G_E$ that
induces the Frobenius on $K$. Let $y_1,\dots,y_e$ be the roots of
$\sigma(\mu) = \sum \sigma(b_i)X^i$. Then as before, 
$v_E(\sigma(\mu)(a)) =\sum v_E(a-y_i)$, and 
$v_E(\sigma(\mu)(a)) = \min_{0\leq i <e}(v_E(\sigma(b_i)-b_i)+ in/e)$. 
As $v_E(b_i^0-b_i) = v_E(\sigma(b_i)-b_i)$ for all $i$, we get that
$v_E(\mu^0(a)) = v_E(\sigma(\mu)(a))$. 

Suppose first that $D$ is closed.
Write $D$ as the set $\{z,v_E(z-a)\geq \lambda\}$ for some $\lambda$, 
then we get that
$v_E(\sigma(\mu)(a)) \geq e\lambda$ as the $y_i$ are among the conjugates of $a$ over
$E$ and hence are in $D$, so $v_E(\mu^0(a)) \geq e\lambda$ and so there exists
an $i$ with $x_i\in D$. Let $F = E(x_i)$ then $F$ is an extension of $E$
of degree at most $e$.
The case of an open disk is similar.
\end{proof}

Note that if we take $e$ to be minimal, then necessarily $F/E$ is totally
ramified and of degree $e$.

\begin{proof}[Proof of Theorem \ref{illdeftheo}]
The case $e=1$ is a consequence of Proposition \ref{illdefunram}.

Assume now that $e > 1$. Let $a\in D$, $F = E(a)$ with $e_{F/E} = e$.
If $a$ is a uniformizer of $F$, the result follows from 
Proposition \ref{illdefeasycase}. Otherwise,
let $f\in E[X]_{<e}$ be a polynomial such that $f(a)$ is a
uniformizer of $F$.

Assume first that such an $f$ exists. Let $D' = f(D)$. Then $D'$ is a
disk defined over $E$ by Lemma \ref{imdisk}, 
containing an element $\varpi = f(a)$ with $e_{E(\varpi)/E} = e$ and
$v_E(\varpi) = 1/e$, so it satisfies the hypotheses of
Proposition \ref{illdefeasycase}. Hence there exists a $c\in D'$ with
$[E(c):E] \leq e$. Let $b\in D$ such that $f(b)=c$, then $[E(b):E]\leq
e(e-1)$ as $b$ is a root of $f(X)-c$, which is a polynomial of degree at
most $e-1$ with coefficients in an extension of degree $e$ of $E$. 
Moreover, by minimality of $e$, we get that $e_{E(b)/E} \geq e$, and so
$[E(b)\cap E^{unr}:E] \leq e-1$. Let $K$ be the maximal extension of $E$
contained in $E(b) \cap E^{unr}$ such that $[K:E]$ is a power of $p$.
Then $[K:E] \leq p^{s-1}$ where $e=p^s$, because $[K:E] \leq e-1$.
Now we apply again Lemma \ref{Benedetto}, to the field $K$:
$D$ contains a point $a'$ with
$[K(a'):K] \leq p^{v_p([E(b):K])}$, that is, $[K(a'):K] \leq p^s$. So
finally $a'\in D$ and $[E(a'):E] \leq p^{2s-1}$.

We prove now the existence of such a polynomial $f$. Fix a uniformizer $\pi_F$ of
$F$, and let $K = E(a) \cap E^{unr}$.
Let $\E$ be the set of pairs of $e$-uples
$(\alpha,P)$ where $\alpha = \alpha_1,\dots,\alpha_e$ are elements of
$K$, $P = P_1,\dots,P_e$ are elements of $E[X]_{<e}$, and
$\sum_i\alpha_iP_i(a) = \pi_F$. Then $\E$ is not empty:  we can write
$\pi_F = Q(a)$ for some $Q\in K[X]_{<e}$; now let $\alpha_1,\dots,\alpha_e$ be a
basis of $K$ over $E$, and write $Q = \sum \alpha_iP_i$ with $P_i\in
E[X]_{<e}$. For each $(\alpha,P)\in \E$ let $m_{(\alpha,P)} =
\inf_iv_E(\alpha_iP_i(a))$, so $m_{(\alpha,P)} \leq 1/e$. It is enough to
show that there is an $(\alpha,P)$ with $m_{(\alpha,P)} = 1/e$. Indeed,
if $v_E(\alpha_iP_i(a)) = 1/e$, let $\beta_i\in E$ with $v_E(\alpha_i) =
v_E(\beta_i)$ then $\beta_iP_i$ is the $f$ we are looking for. 

So choose
a $(\alpha,P)\in \E$ with $m = m_{(\alpha,P)}$ minimal, and with minimal
number of indices $i$ such that $v_E(\alpha_iP_i(a))=m$. 
Suppose that $m<1/e$.
Then there are at least
two indices $i$ with $v_E(\alpha_iP_i(a)) = m$. Say for simplicity that
$v_E(\alpha_1P_1(a)) = v_E(\alpha_2P_2(a)) = m$. By minimality of $e$, $P_1$
and $P_2$
have no root in $D$. Let $f = P_1/P_2$, and $D' = f(D)$. Then $D'$ is
defined over $E$, and contains an element $f(a)$ of valuation $r =
v_E(P_1(a)/P_2(a)) \in \Z$, as $r = v_E(\alpha_2/\alpha_1)$. Consider
$\pi_E^{-r}D'$. It contains an element of valuation $0$ and it does not
contain $0$, so it is contained in a disk $\{z,v_E(z-c)>0\}$ for some
element $c$ that is the Teichmueller lift of an element of
$\bar\F_p^\times$.  So $v_E(\pi_E^{-r}P_1(a)/P_2(a)-c) > 0$.  As
$\pi_E^{-r}D'$ is defined over $E$, we have that $c\in E$.  Let $x =
c\pi_E^r$, then $v_E(P_1(a)-xP_2(a)) > r+v_E(P_2(a))= v_E(P_1(a))$. We
define an element $(\alpha',P')$ of $\E$ by setting $P_1' = P_1-xP_2$ and
$\alpha_2' = \alpha_2+x\alpha_1$, and $\alpha_i' = \alpha_i$ and $P_i' =
P_i$ for all other indices. We observe that $v_E(\alpha_1'P_1'(a)) > m$,
$v_E(\alpha_2'P_2'(a)) \geq m$, and all other valuations are unchanged.
This contradicts the choice we made for $(\alpha,P)$ at the beginning. So
in fact $m=1/e$.  
\end{proof}

\section{Some results on Hilbert-Samuel multiplicities}
\label{commalg}

\subsection{Hilbert-Samuel multiplicity}

Let $A$ be a noetherian local ring with maximal ideal $\m$, and $d$ be
the dimension of $A$. Let $M$ be a finitely generated module over $A$. 
We recall the definition of the Hilbert-Samuel multiplicity $e(A,M)$ (see
\cite[Chapter 13]{Mat}).
For
$n$ large enough, $\len_A(M/\m^nM)$ is a polynomial in $n$ of degree at
most $d$. We can write its term of degree $d$ as $e(A,M)n^d/d!$ for an
integer $e(A,M)$, which is the Hilbert-Samuel multiplicity of $M$
(relative to $(A,\m)$). We also write $e(A)$ for $e(A,A)$.

If $\dim A = 1$, it follows from the definition that $e(A,M) =
\len_A(M/\m^{n+1}M) - \len_A(M/\m^nM) = 
\len_A(\m^nM/\m^{n+1}M)$ for $n$ large enough.

We give some results that will enable us to compute $e(A)$ for some
special cases of rings $A$ of dimension $1$.

\begin{lemm}
\label{computee}
Let $k$ be a field, and $(A,\m)$ be a local noetherian $k$-algebra of dimension
$1$, with $A/\m = k$. 
Suppose that there exists an element $z \in \m$ such that $A$ has no
$z$-torsion and for all
$n$ large enough, $z\m^n = \m^{n+1}$. Then $e(A) = \dim_kA/(z)$.
\end{lemm}

\begin{proof}
For $n$ large enough, we have $\m^{n+1} \subset (z)$. So the surjective map
$A \to A/(z)$ factors through $A/\m^{n+1}$ (and in particular
$\len_A(A/(z))$ is finite). We have an exact sequence:
$$
A \stackrel{z}{\rightarrow} A/\m^{n+1} \to A/(z) \to 0
$$
For $n$ large enough,
the kernel of the first map is $\m^n$ by the assumptions on $z$: it
contains $\m^n$, and as multiplication by $z$ is injective, it is exactly
equal to $\m^n$. So we have an exact sequence:
$$
0 \to A/\m^n \stackrel{z}{\rightarrow} A/\m^{n+1} \to A/(z) \to 0
$$
This gives $\len_A(\m^n/\m^{n+1}) = \len_A(A/(z)) = \dim_kA/(z)$ as
stated.
\end{proof}

\begin{coro}
\label{computeebis}
Let $k$ be a field, and $(A,\m)$ be a local noetherian $k$-algebra of dimension
$1$, with $A/\m = k$. 
Suppose that there exist an element $z \in \m$ such that $A$ has no
$z$-torsion and a nilpotent ideal $I$ such that $\m = (z,I)$.
Then $e(A) = \dim_kA/(z)$.
\end{coro}

\begin{proof}
We need only show that $z\m^n = \m^{n+1}$ for all $n$ large enough, as we
can then apply Lemma \ref{computee}. Let $m$ be an integer such that
$I^m=0$. Then for $n > m$ we have $\m^n = \sum_{i=0}^mI^iz^{n-i}$, 
which gives the result.
\end{proof}

\begin{defi}
\label{nearlysum}
Let $k$ be a field. Let $A_1,\dots,A_s$ be a family of local noetherian
complete
$k$-algebras of dimension $1$ with maximal ideals $M_i$ and residue field
$k$.
Let $A$ be a local noetherian complete $k$-algebra with maximal ideal
$\m$ and residue field $k$. We
say that $A$ is nearly the sum of the family $(A_i)$ if 
there are injective $k$-algebra maps $u_i : A_i \to A$ such that
$A = k \oplus(\oplus_{i=1}^su_i(M_i))$ as a
$k$-vector space and $\m = \oplus_{i=1}^su_i(M_i)$. 

In this case, we write $V_i$ for $u_i(M_i)$, and for all $n>0$, $V_i^n$
is defined as $u_i(M_i^n)$, and $V_i^0$ is defined as $\{1\}$.
For $\alpha = (\alpha_1,\dots,\alpha_s)\in
\Z_{\geq 0}^s$, we denote by $V^\alpha$ the 
vector space generated by elements of the form $x_1\dots x_s$, 
where $x_i$ is an element of $V_i^{\alpha_i}$.
Note that this is not in general an ideal of $A$.
We also denote by $V^\alpha V_i^n$ the set $V^\beta$ where $\beta_j =
\alpha_j$, except for $\beta_i = \alpha_i + n$.
\end{defi}

\begin{ex}
\label{exns}
Let $A_1 = k[[z_1]]$, $A_2 = k[[x_2,z_2]]/(x_2^2)$, and $A =
k[[z_1,x_2,z_2]]/(x_2^2,z_1x_2,z_1z_2-x_2)$. Then $A$ is nearly the sum
of $A_1$ and $A_2$, for the natural maps $A_i \to A$.
\end{ex}

\begin{lemm}
\label{additivitye}
Let $k$ be a field. Let $A_1,\dots,A_s$ be a family of local noetherian
complete
$k$-algebras of dimension $1$ with maximal ideals $M_i$ and residue field
$k$. Suppose that for all
$i$, there is an element $z_i\in A_i$ such that $A_i$ has no
$z_i$-torsion and that for all $n$ large enough, $z_iM_i^n =
M_i^{n+1}$.

Let $A$ be a $k$-algebra with maximal ideal $\m$ 
that is nearly the sum of the family $(A_i)$, and let $V_i = u_i(M_i)$ as
in Definition \ref{nearlysum}.
Moreover, suppose that there exist integers $N_0$ and $t_0$ with $t_0
< N_0$ such that
for all $i$ and $j$, $V_jV_i^n \subset V_i^{n-t_0}$ for all $n
\geq N_0$.

Then $e(A) = \sum_{i=1}^s e(A_i)$.
\end{lemm}

Note that if we had the stronger property that $V_iV_j = 0$ for all $i
\neq j$ the result would be trivial, as then $\m^n =
\oplus_{i=1}^sV_i^n$ and $\m^n/\m^{n+1} = \oplus_{i=1}^sV_i^n/V_i^{n+1}$.

\begin{proof}
Observe first that there exist integers $N$ and $t$ with $N > t$ such that for all
$\alpha$, for all $i$, $V^{\alpha}V_i^n \subset V_i^{n-t}$ for all $n
\geq N$. Indeed, $V^\alpha \subset V_{j_1}\dots V_{j_r}$ where
$\{j_1,\dots,j_r\}\subset \{1,\dots,s\}$ is the set of indices with
$\alpha_j >0$. Then if $n \geq rN_0$, then $V_{j_1}\dots V_{j_r}V_i^n
\subset V_i^{n-rt_0}$. So we can take $N = sN_0$ and $t = st_0$.

If $\alpha_j > N$, then $V^\alpha \subset V_j^{\alpha_j-t} \subset
V_j$. So if there are two different indices $i,j$ with $\alpha_i > N$
and $\alpha_j > N$, then $V^{\alpha} = 0$ as it is contained in $V_i \cap
V_j$. If $|\alpha| > sN$, then there exists at least one $i$ with
$\alpha_i > N$ so $V^\alpha = \sum (V^{\alpha}\cap V_j)$.

Fix some index $i$. 
Let $n > 0$. Then $\m^n = \sum_{|\alpha| = n}V^\alpha$. So if $n > Ns$,
then $(\m^n \cap V_i) = \sum_{\alpha} (V^\alpha \cap V_i)$ and the only
contributing terms are those with $\alpha_j \leq  N$ for all $j \neq i$, and
$\alpha_i > N$. 
For such an $\alpha$, we have $V^\alpha \subset V_i^{n-sN}$ as $\alpha_i \geq n-(s-1)N$. 
Let $r = sN$, so that $V_i^{n-r}\subset V_i$ for all $n>r$. So for all
$n>r$ and all such $\alpha$ we have $V^\alpha \subset V_i$,
so finally for $n>r$ we have: 
\begin{equation}
\label{mn}
(\m^n\cap V_i) = \sum_{|\alpha|=n,\alpha_j \leq N\text{ if }j\neq
i}V^\alpha.
\end{equation}
We see that
$V_i^n \subset (\m^n\cap V_i) \subset V_i^{n-r}$ for all $n > r$.

Note that $(\m^n \cap V_i)$ is an ideal of $A_i$, which we denote by
$W_{i,n}$. We know that $z_iV_i^n = V_i^{n+1}$ for all $n$ large enough, so
by the formula (\ref{mn}) for $W_{i,n}$ we see that $z_iW_{i,n} = W_{i,n+1}$ for all $n$ large
enough. In $A_i$, multiplication by $z_i$ induces an isomorphism from
$V_i^n$ to $V_i^{n+1}$ and from $W_{i,n}$ to $W_{i,n+1}$, so it also induces an
isomorphism from $V_i^{n-r}/W_{i,n}$ to $V_i^{n+1-r}/W_{i,n+1}$ for all $n$
large enough. Note that these vector spaces are finite-dimensional, so
they have the same dimension, as $\dim_k V_i^{n-r}/V_i^n$ is finite for
all $n$. 

We consider the inclusions
$$V_i^n \subset W_{i,n} \subset V_i^{n-r} \subset W_{i,n-r} \subset
V_i^{n-2r}.$$
We know that
for all $n \gg 0$,
$\dim_k V_i^{n-r}/V_i^n = \dim_k V_i^{n-2r}/V_i^{n-r} = re(A_i)$ and
$\dim_k V_i^{n-r}/W_{i,n} = \dim_kV_i^{n-2r}/W_{i,n-r}$, which gives that
$\dim_k W_{i,n-r}/W_{i,n} = re(A_i)$.

We now go back to $A$. For all $n \gg 0$ we have that 
$\dim_k(\m^{n-r}/\m^n) = re(A)$. 
On the other hand, we have seen that for all $n \gg 0$, $\m^n 
= \oplus_i (\m^n\cap V_i)$, so $\m^{n-r}/\m^n$ is isomorphic to
$\oplus_i(\m^{n-r}\cap V_i)/(\m^n\cap V_i) = \oplus_i (W_{i,n-r}/W_{i,n})$. 
So $re(A) = \sum_{i=1}^sre(A_i)$, and so $e(A) = \sum_i e(A_i)$.
\end{proof}

\begin{ex}
\label{exnsbis}
Let us take again $A_1$, $A_2$, $A$ as in Example \ref{exns}. For all $n
>0$, $\m^n$ is the set of terms of the form $\sum_{i \geq n}a_iz_1^i +
>\sum_{i \geq n}b_iz_2^i + \sum_{i \geq n-2}c_ix_2z_2^{i}$. Indeed, we
have that $x_2 \in V_1V_2 \subset \m^2$, although in $A_2$ we have $x_2 \in M_2
\setminus M_2^2$. So we do not have $\m^n \cap V_2 = V_2^n$, but
$\m^n \cap V_2 = V_2^n + (V_1V_2^{n-1} \cap \m^n)$, where $V_1V_2^{n-1}$ is the part
that contains the term $x_2z_2^{n-2}$.
\end{ex}

\subsection{Hilbert-Samuel multiplicity of the special fiber}
\label{HSspecial}

Let $R$ be a discrete valuation ring with uniformizer $\pi$ and residue
field $k$.

Let $A$ be a local $R$-algebra with maximal ideal $\m$, and let $M$ be an
$A$-module of finite type. We denote by $\e_R(A,M)$ the Hilbert-Samuel multiplicity of
$M\otimes_R k$ as an $A\otimes_Rk$-module, with respect to the ideal
$\m\otimes_Rk$. When $M = A$ we just write $\e_R(A)$ instead of
$\e_R(A,A)$, and we omit the subscript $R$ when the choice of the ring is
clear from the context.

\begin{lemm}
\label{extring}
Let
$(T,\m_T) \to (S,\m_S)$ be a local morphism of local noetherian rings of
the same dimension,
with residue fields $k_T$ and $k_S$ respectively,
then $e(T,S) \geq [k_S:k_T]e(S)$.
\end{lemm} 

\begin{proof}
Let $n \geq 0$ be an integer. Then $S/\m_S^n$ is a quotient of
$S/(\m_TS)^n$, so 
$\len_T(S/\m_S^n) \leq \len_T(S/(\m_TS)^n)$.
Morevoer, 
$$
\len_T(S/\m_S^n) = 
\sum_{i=0}^{n-1}\dim_{k_T}\m_S^i/\m_S^{i+1}
=
[k_S:k_T]
\sum_{i=0}^{n-1}\dim_{k_S}\m_S^i/\m_S^{i+1}
= 
[k_S:k_T]
\len_S(S/\m_S^n)
$$
so finally $[k_S:k_T]\len_S(S/\m_S^n) \leq
\len_T(S/(\m_TS)^n)$ which gives the result.
\end{proof}

\begin{prop}
\label{multsmaller}
Let $A$ be a complete noetherian local $R$-algebra which is a
domain. Let $B \subset A[1/\pi]$ be a finite $A$-algebra.
Let $k_A$ and $k_B$ be the residue fields of $A$ and $B$ respectively.
Then $\e(A) \geq [k_B:k_A]\e(B)$.
\end{prop}

\begin{proof}
Note that $B$ is also a complete noetherian local $R$-algebra which
is a domain.  Indeed, $A$ is henselian and $B$ is a finite $A$-algebra,
so $B$ is a finite product of local rings, and so it is a local ring as
it is a domain.

It is enough to prove the result when $\pi B \subset A$, as $B$ is
generated over $A$ by a finite number of elements of the form $x/\pi^n$
for $x \in A$. 

We have an exact sequence of $R$-modules:
$$
0 \to A \to B \to B/A \to 0.
$$
After tensoring by $k$ over $R$ we get the exact sequence:
$$
0 \to B/A \to A\otimes_Rk \to B\otimes_Rk \to B/A \to 0.
$$
Indeed, $(B/A)\otimes_Rk = B/A$, and $(B/A)[\pi] = B/A$ and $B$ is
$\pi$-torsion free so $B[\pi] = 0$.

Hence we get that $\e(A,B) = \e(A,A)$. So we only need to show that
$\e(A,B) \geq [k_B:k_A]\e(B)$, which follows from Lemma \ref{extring}
applied to $T = A\otimes_Rk$ and $S = B\otimes_Rk$.
\end{proof}

\begin{rema}
\label{exdiffe}
We give some examples: Let $R = \Z_p$, $C = R[[X]]$,
$A_n = R[[pX,X^n]] \subset C$ for $n \geq 1$, $B_n =
R[[pX,pX^2,\dots,pX^{n-1},X^n]] \subset C$ for $n \geq 1$.
We check easily that $A_n \subset B_n \subset C$ and that $C$ is finite
over $A_n$, and $A_n$ is not equal to $B_n$ if $n>2$.
We compute that $\e(A_n) = \e(B_n) = n$, and $\e(C) = 1$. So
we see that in Proposition \ref{multsmaller}, both possibilities $\e(B) <
\e(A)$ and $\e(B) = \e(A)$ can happen for $A \neq B$.
See also Paragraph \ref{recover} for more examples.
\end{rema}

\subsection{Change of ring}

We suppose now that $R$ is the ring of integers of a finite extension $K$ of
$\Q_p$. If $K'$ is a finite extension of $K$, we denote by $R'$ its ring
of integers.

\begin{prop}
\label{changeringram}
Let $K'$ be a finite extension of $K$, with ramification degree
$e_{K'/K}$.
Let $A$ be a local noetherian $R'$-algebra.
Then $\e_R(A) = e_{K'/K}\e_{R'}(A)$.
\end{prop}

\begin{proof}
Suppose first that $K'$ is an unramified extension of $K$, and let $k$
and $k'$ be the residue fields of $K$ and $K'$ respectively, and
let $\pi$ be a uniformizer of $R$ and $R'$.
Then $A\otimes_{R'}k' = A\otimes_Rk = A/\pi A$. 
So $\e_R(A) = e(A/\pi A) = \e_{R'}(A)$.

Suppose now that $K'$ is a totally ramified extension of $K$. 
Let $u$ be
an Eisenstein polynomial defining the extension, so that $R' = R[X]/u(X)$,
and $\bar{u}(X) = X^s$ where $s = [K':K]$. Then $A\otimes_Rk =
A\otimes_{R'}(R'\otimes_Rk) = A\otimes_{R'}(k[X]/(X^s)) =
(A\otimes_{R'}k)\otimes_kk[X]/(X^s)$. So $\e_R(A) = s\e_{R'}(A) =
[K':K]\e_{R'}(A)$.

For the general case, let $R_0$ be the ring of integers of the maximal
unramified extension $K_0$ of $K$ in $K'$, then $e_R(A) =
e_{R_0}(A)$ and $e_{R_0}(A) = [K':K_0]e_{R'}(A)$ which gives the
result.
\end{proof}

We recall the following result, which is \cite[Lemme 2.2.2.6]{BM}:

\begin{lemm}
\label{changeringany}
Let $A$ be a local noetherian $R$-algebra, with the same residue field as
$R$ and $A$ is complete and topologically of finite type over $R$.
Let $K'$ be a finite 
extension of $K$, and $A' =
R'\otimes_RA$. Suppose that $A'$ is still a local ring.
Then $\e_R(A) = \e_{R'}(A')$. 
\end{lemm}

\section{Rigid geometry and standard subsets of the affine line}
\label{geometry}

\subsection{Quasi-affinoid algebras and rigid spaces}

\subsubsection{Quasi-affinoid algebras}
\label{qaffalg}

Let $F$ be a finite extension of $\Q_p$, with ring of integers $\OO_F$.
We denote by $R_{n,m}$, or $R_{n,m,F}$, the $F$-algebra 
$\OO_F\langle x_1,\dots,x_n\rangle [[y_1\dots,y_m]]\otimes_{\OO_F}F$.
Following \cite{LR00}, we say that an $F$-algebra is a quasi-affinoid
algebra (or an $F$-quasi-affinoid algebra) if it is a quotient of
$R_{n,m}$ for some $n$, $m$. For example, an $F$-affinoid algebra
is $F$-quasi-affinoid, as it is a quotient of the Tate algebra
$R_{n,0,F}$ for some $n$.  The theory of quasi-affinoid algebras has also
been studied by other authors under the name "semi-affinoid algebras" (see
for example \cite{Kap}).

Let $A$ be an $F$-quasi-affinoid algebra. Following \cite[Definition
2.2]{Kap}, we say that an $\OO_F$-subalgebra $\underline{A}$ of $A$ is an
$\OO_F$-model of $A$ if the canonical morphism
$\underline{A}\otimes_{\OO_F}F \to A$ is an isomorphism. Note that an
$\OO_F$-model is automatically $\OO_F$-flat.
Assume that $A$ is normal. Let $\underline{A}$ be an $\OO_F$-model of $A$,
and let $A^0$ be the integral closure of $\underline{A}$ in $A$. Then
$A^0$ is normal, and is an $\OO_F$-model of $A$.

A quasi-affinoid algebra is said to be of open type if it has an $\OO_F$-model
that is local, or equivalently, if it is a quotient of $R_{0,m}$ for
some $m$. For example, let $\mathcal{R}$ be one of the potentially
semi-stable deformation rings defined by Kisin (as recalled in Section
\ref{kisinrings}), then $\mathcal{R}[1/p]$
is a quasi-affinoid algebra of open type. These will be our main focus of
interest, but we need to use quasi-affinoid algebras that are not
necessarily of open type in order to study them.

Quasi-affinoid algebras have some properties that are similar to affinoid
algebras: for example they are noetherian and they are Jacobson rings,
and the Nullstellensatz holds for them.

\subsubsection{Rigid spaces attached to quasi-affinoid algebras}
\label{dJconstr}

Let $A$ be an $F$-quasi-affinoid algebra.
Using Berthelot's construction, as described in \cite[\S 7]{dJ}, we can
attach canonically to it a rigid space $\X = \X_A$ defined over $F$. 
We say that such a rigid space is the quasi-affinoid space attached to
$A$. We say that a quasi-affinoid space is of open type if it is attached
to a quasi-affinoid algebra of open type.

We give here some properties of this construction. We denote by
$\underline{A}$ an $\OO_F$-model of $A$.

\begin{prop}
\label{propdJ}
\begin{enumerate}
\item 
we have a natural map $A \to \Gamma(\X,\OO_\X)$ which induces a map
$\underline{A} \to \Gamma(\X,\OO_\X^0)$ (where $\OO_\X^0$ is the sheaf of functions
bounded by $1$). 

\item
the map $\underline{A} \to \Gamma(\X,\OO_\X^0)$ is an isomorphism as soon as
$\underline{A}$ is normal. In particular, in this case $A$ is isomorphic to
the subring of $\Gamma(\X,\OO_\X)$ of functions that are bounded.

\item
there is a functorial bijection between $\Max(A)$ and the points of $\X$.

\item
this construction is compatible with base change by a finite extension $F
\to F'$

\end{enumerate}

\end{prop}

\begin{proof}
Property (1) is \cite[7.1.8]{dJ}, (2) is \cite[7.4.1]{dJ}, using the fact
that an $\OO_F$-model is $\OO_F$-flat.

Property (3) is \cite[7.1.9]{dJ} and (4) is \cite[7.2.6]{dJ}.
\end{proof}

If $\X$ is a rigid space over $F$, we write $\A^0_F(\X)$ for
$\Gamma(\X,\OO_\X^0)$ and $\A_F(\X)$ for the subring of $\Gamma(\X,\OO_\X)$ of
functions that are bounded. If $\X$ is the rigid space attached to an
$F$-quasi-affinoid algebra $A$ that is normal, then $A$ has a
normal $\OO_F$-model $\underline{A}$, and we have
$A = \A_F(\X)$ and $\underline{A} = \A_F^0(\X)$ (in particular, there is
actually only one $\OO_F$-model of $A$ that is normal, and it contains all
other $\OO_F$-models).

A map $f : \X \to \Y$ between $F$-quasi-affinoid rigid spaces is quasi-affinoid
if it is induced by an $F$-algebra map $f^\#: \A_F(\Y) \to \A_F(\X)$. By
Proposition \ref{propdJ}, it is easy to see that any rigid analytic map
$f : \X \to \Y$ between $F$-quasi-affinoid spaces is in fact
quasi-affinoid as soon as $\X$ is normal.

\subsubsection{$R$-subdomains}

As in the case of affinoid algebras and rigid spaces, we define some
special subsets of quasi-affinoid spaces.

Let $\X$ be a quasi-affinoid space. Let $h$, $f_1,\dots,f_n$,
$g_1,\dots,g_m$ be
elements of $\A_F(\X)$ that generate the unit ideal of
$\A_F(\X)$. A quasi-rational subdomain of $\X$ is a subset $U$ of the form
$\{x, |f_i(x)| \leq |h(x)|\;\forall i\text{ and }|g_i(x)| < |h(x)|\;\forall i\}$
(see \cite[Definition 5.3.3]{LR00}). 
This generalizes the notion of an affinoid subdomain: let $\X$ be an
affinoid rigid space, then an affinoid subdomain of $\X$ is a subset
defined by equations of the form $\{x, |f_i(x)| \leq |h(x)|\;\forall i\}$
where $f_1,\dots,f_n$ and $h$ generate the unit ideal.

In contrast to the case of affinoid subdomains, it is not necessarily
true that a quasi-rational subdomain of a quasi-rational subdomain of $\X$ is itself a
quasi-rational subdomain of $\X$, see \cite[Example 5.3.7]{LR00}.
We recall the definition of a $R$-subdomain of $\X$ (\cite[Definition
5.3.3]{LR00}):
the set of $R$-subdomains of $\X$ is defined as the smallest set of
subsets of $\X$
that contains $\X$ and is closed by the operation of taking a quasi-rational
subdomain of an element of this set. In particular, any finite
intersection of $R$-subdomains of $\X$ is also an $R$-subdomain.

Any $R$-subdomain of a quasi-affinoid space $\X$ is itself a
quasi-affinoid space in a canonical way, attached to the quasi-affinoid
algebra constructed as in \cite[Definition 5.3.3.]{LR00}.

\subsection{$R$-subdomains of the unit disk}
\label{Rsubdisk}

Our goal now is to understand the subsets of $\bar\Q_p$ that are the set
of points of some quasi-affinoid space. For simplicity, we consider for
the moment only subsets of the unit disk.  Let $\D$ be the rigid closed unit
disk, seen as a quasi-affinoid space defined over $\Q_p$, or over
any finite extension of $\Q_p$, so that $\D(\bar\Q_p) = D(0,1)^+$.  We
say that $X \subset D(0,1)^+$ is an $R$-subdomain if it is of the form
$\X(\bar\Q_p)$ for some $R$-subdomain $\X$ of $\D$, and that it is a
quasi-rational subdomain if is of the form $\X(\bar\Q_p)$ for some
quasi-rational subdomain $\X$ of $\D$.

\begin{defi}
We say that a subset of $\bar\Q_p$ is a rational disk if it is a set of the
form 
$\{x,|x-a|<r\}$ with $a\in\bar\Q_p$, 
$r \in |\bar\Q_p^\times|$
(open disk), or of
the form 
$\{x,|x-a|\leq r\}$ with $a\in\bar\Q_p$, 
$r \in |\bar\Q_p^\times|$
(closed disk).

Let $F$ be a finite extension of $\Q_p$, we say that a disk is
well-defined over $F$ if it can be written as 
$\{x, |x-a|<r\}$ or as $\{x,|x-a|\leq r\}$ for some $a\in F$ and $r\in
|F^\times|$.
\end{defi}

Recall (see Definition \ref{definedover}) that a disk is defined over $F$
if it is fixed by $G_F$, so a disk that is well-defined over $F$ is
defined over $F$, although the converse is not necessarily true.

From now on, when we write "disk" we always mean "rational disk". It is
clear that a rational disk is a quasi-rational subdomain of the affine line.

Following \cite[Definition 4.1]{LR96}, we define:

\begin{defi}
\label{defspecial}
A special subset of $\bar\Q_p$ is a subset of one of the following forms:
\begin{enumerate}
\item 
$\{x, r < |x-a| < r'\}$ for some $a \in \bar\Q_p$ and $r$,$r'$ in
$|\bar\Q_p^\times|$.

\item
$\{x, |x-a| \leq r\text{ and for all }i\in\{1,\dots,N\}, |x-\alpha_i|\geq
r_i\}$ for some $a,(\alpha_i)_{1 \leq i \leq N},\in \bar\Q_p$ and $r$,
$(r_i)_{1 \leq i \leq N}$ in $|\bar\Q_p^\times|$.
\end{enumerate}
\end{defi}

Special subsets are $R$-subdomains, as they are finite intersections of
quasi-rational subdomains. We have the following result:
\begin{lemm}[Theorem 4.5 of \cite{LR96}]
\label{Rspecial}
An $R$-subdomain of $D(0,1)^+$ is a finite union of special sets.
\end{lemm}

\begin{defi}
\label{defRsubset}
We say that a subset $X$ of $D(0,1)^+$ is a connected $R$-subset if it is
of the following form: $D_0 \setminus \cup_{i=1}^n D_i$ where the $D_i$
are rational disks contained in $D(0,1)^+$, $D_0 \neq D_i$ for all $i>0$,
$D_i\subset D_0$, and $D_i$ and $D_j$ are disjoint if $i\neq j$ and
$i,j>0$.

We say that a subset $X$ of $D(0,1)^+$ is an $R$-subset if it is a finite
disjoint union of connected $R$-subsets.
\end{defi}

We say that a connected $R$-subset is of closed type if $D_0$ is
closed and the $D_i$, $i>0$ are open. We say that it is of open type if
$D_0$ is open and the $D_i$, $i>0$ are closed. We say that an
$R$-subset is of closed type (resp. open type) if it is a finite union of connected
$R$-subset of closed type (resp. open type).
We say that a connected $R$-subset is well-defined over some extension
$F$ of $\Q_p$ if each disk involved in its description is well-defined
over $F$.

We check easily the following result:
\begin{lemm}
\label{interstd}
Let $X$ and $Y$ be two connected $R$-subsets of closed (resp. open) 
type. If $X \cap Y \neq \emptyset$ then $X\cap Y$ and $X\cup Y$ are connected
$R$-subsets of closed (resp. open) type. As a consequence, any
finite union of connected $R$-subsets of closed (resp. open) type is
an $R$-subset of closed (resp. open) type.
\end{lemm}

From Lemma \ref{Rspecial}, we get the following property of
$R$-subdomains of the unit disk:

\begin{prop}
\label{descrRdom}
Any $R$-subdomain of the unit disk is an $R$-subset.
\end{prop}

On the other hand, we can ask whether any $R$-subset is an $R$-subdomain.

\begin{prop}
\label{isRdomain}
Let $X$ be a connected $R$-subset. Let $F$ be a finite extension of
$\Q_p$ such that $X$ is well-defined over $F$. Then $X$ is a
quasi-affinoid subdomain of $D(0,1)^+$, and it is the set of points of a
quasi-affinoid space defined over $F$ which is uniquely defined as a
quasi-rational subdomain of $\D$.  
\end{prop}

\begin{proof}
From Definition \ref{defRsubset}, we see that $X$ can be defined
by a finite number of equations of the form $|x-a| < |b|$ or $|x-a|\leq
|b|$ or $|x-a| > |b|$ or $|x-a| \geq |b|$ for $a$, $b$ in $F$ and $b \neq
0$.
\end{proof}

In particular, if $X$ is a connected $R$-subset of closed type, then it
is the set of points of an affinoid subdomain of the unit disk, and 
any affinoid subdomain of the unit disk is of this form by
\cite[Theorem 9.7.2/2]{BGRa}. 

\begin{defi}
\label{defringfunc}
Let $\X$ be an $R$-subdomain of $\D$, defined over $F$ as a quasi-affinoid
space, and let $X = \X(\bar\Q_p)$. Then we write $\A_F(X)$ for $\A_F(\X)$
and $A^0_F(X)$ for $\A_F^0(\X)$. 
\end{defi}

For example, let $X$ be the disk defined by $|x-a|<|b|$ for some $a$, $b$
in $F$, $b \neq 0$. Then $\A^0_F(X) = \OO_F[[(x-a)/b]]$ is isomorphic to the
power series ring $\OO_F[[t]]$.
Let $Y$ be the annulus defined by $|c|<|x-a|<|b|$ for some $a$,$b$,$c$ in
$F$ with $c\neq 0$ and $|c| < |b|$. Then $\A_F^0(Y) =
\OO_F[[(x-a)/b,c/(x-a)]]$, which is isomorphic to $\OO_F[[t,u]]/(tu-c/b)$. 
In general the ring $\A_F(X)$ can be entirely described using
\cite[Def. 5.3.3.]{LR00}. We will not give a formula, but we see
easily that for a connected $R$-subset $X$, $\A_F^0(X)$ is local if and
only if $X$ is of open type. So we have:

\begin{prop}
\label{isopentype}
Let $X$ be a connected $R$-subset. Assume that we know that $X$ is the
set of points of a quasi-affinoid space of open type. Then $X$ is an
$R$-subset of open type.
\end{prop}

\subsection{Rings of functions on standard subsets}

We continue studying subsets of $\bar\Q_p$, or more generally of
$\PP^1(\bar\Q_p)$, coming from quasi-affinoid spaces. From now on, we will
be only interested in $R$-subsets that are of open type, but we will not
necessarily assume that the subsets are contained in the unit disk
anymore. 

\subsubsection{Standard subsets}

We make the following definitions:

\begin{defi}
\label{standard}
We say that a subset $X$ of $\PP^1(\bar\Q_p)$ is a connected standard subset if it
is of one of the following forms:
\begin{enumerate}
\item
$D_0 \setminus \cup_{i=1}^n D_i$ where the $D_i$ are rational disks,
$D_0$ is open and each $D_i$ is closed for $i>0$,
$\infty\not\in D_0$,
$D_0 \neq D_i$ for all $i>0$, $D_i\subset D_0$, and $D_i$ and $D_j$ are
disjoint if $i\neq j$ and $i,j>0$ (bounded connected standard subset).

\item
$\PP^1(\bar\Q_p) \setminus \cup_{i=1}^n D_i$ where the $D_i$ are rational
disks, each $D_i$ is closed,
and $D_i$ and $D_j$ are
disjoint if $i\neq j$ (unbounded connected standard subset).
\end{enumerate}
The disks $(D_i)$ are called the defining disks of $X$.
\end{defi} 

So a bounded standard subset contained in the unit disk is the same thing
as a connected $R$-subset of open type.

\begin{defi}
A standard subset is a finite disjoint union of connected standard
subsets of $\PP^1(\bar\Q_p)$. The connected standard subsets that
appear are called the connected components of the standard subset.
The defining disks of a standard subsets are the defining disks of each
of its connected components.
\end{defi}

It is clear that a standard subset can be written in a unique way as a
finite disjoint union of connected standard subsets so the notion of
connected component is well-defined.

Let $F$ be a finite extension of $\Q_p$. We say that a standard subset is
well-defined over $F$ if each defining disk of $X$ is
well-defined over $F$.

\subsubsection{Definition of the rings of functions of standard subsets}

Let $X \subset \bar\Q_p$ a connected standard subset, which is
well-defined over some finite extension $F$ of $\Q_p$. Although it is not
necessarily contained in the unit disk, it is contained in some closed
disk, and so all the results of Section \ref{Rsubdisk} apply to $X$. In
particular we can define $\A_F(X)$ and $\A_F^0(X)$ as in Definition
\ref{defringfunc}.

Let $X \subset \PP^1(\bar\Q_p)$ be an unbounded connected standard subset, which is
not equal to all of $\PP^1(\bar\Q_p)$. 
Let $f$ be a rational function with $\bar\Q_p$-coefficients which defines
a bijection of $\PP^1(\bar\Q_p)$, and such that 
its pole is outside of $X$, then $Y=f(X)$ is a bounded
connected standard subset of $\bar\Q_p$. Let $F$ be a finite extension of
$\Q_p$ such that $X$ is
well-defined over $F$, and such that the rational function $f$ has coefficients
in $F$.
Then $\A_F(Y)$ and $\A^0_F(Y)$ are
well-defined. We define $\A_F(X)$ and $\A^0_F(X)$ to be the functions of $X$
of the form $u\circ f$ for $u\in\A_F(Y)$ and $\A^0_F(Y)$ respectively. It is
clear that this does not depend on the choice of $f$, as different
choices of $f$ give rise to bounded connected standard subsets coming
from isomorphic quasi-affinoids.

We give now a general formula for these functions rings, which can be
obtained using \cite[Def. 5.3.3.]{LR00}:

\begin{prop}
\label{ringfuncs}
Let $X$ be a connected standard subset that is well-defined over some
finite extension $E$ of $\Q_p$. Write
$X = D(a_0,r_0)^- \setminus \cup_{j=1}^n D(a_j,r_j)^+$ 
or $X = \PP^1(\bar\Q_p) \setminus \cup_{j=1}^n D(a_j,r_j)^+$ 
with $a_j\in E$ for all $j$, and the
sets $D(a_j,r_j)^+$ are pairwise disjoint for $j>0$.
For each $j$, let $t_j\in E$ be such that
$|t_j| = r_j$. 
Then for any finite extension $F/E$, we have:
\begin{eqnarray*}
\A_F(X) = \{f, f(x) = \sum_{i\geq 0}c_{i,0}\left(\frac{x-a_0}{t_0}\right)^i
+ \sum_{j=1}^n \sum_{i> 0}c_{i,j}\left(\frac{t_j}{x-a_j}\right)^i
\\
\text{ with }c_{i,j}\in F\text{ for all }i,j\text{ and }
\{c_{i,j},0 \leq j \leq n, i \geq 0\}\text{ bounded }\}
\end{eqnarray*}
if $X$ is bounded and
\begin{eqnarray*}
\A_F(X) = \{f, f(x) = c_0
+ \sum_{j=1}^n \sum_{i> 0}c_{i,j}\left(\frac{t_j}{x-a_j}\right)^i
\\
\text{ with }c_{i,j}\in F\text{ for all }i,j\text{ and }
\{c_{i,j},0 \leq j \leq n, i \geq 0\}\text{ bounded }\}
\end{eqnarray*}
if $X$ is unbounded.

Moreover, $\|f\|_X = \sup_{i,j}|c_{i,j}|$ if $f$ is written as above. If
we write $f_0 = \sum_{i\geq 0}c_{i,0}\left(\frac{x-a_0}{t_0}\right)^i$ (or
$f_0 = c_0$ in the unbounded case), and $f_j = \sum_{i>
0}c_{i,j}\left(\frac{t_j}{x-a_j}\right)^i$ for $j>0$ so that $f =
\sum_{i=0}^nf_i$ then $\|f\|_X = \max_{0\leq i \leq n}\|f_i\|_X$.

In particular, $f\in\A^0_F(X)$ if and only if $c_{i,j}\in\OO_F$ for all $i,j$.
\end{prop}

\begin{rema}
\label{cstterm}
Assume that $X$ is unbounded. Then the value of the constant term $c_0$
is independent from the choice of the $a_i \in D(a_i,r_i)^+$ used to write
the decomposition, as it is the value of the function at $\infty$. 
\end{rema}

Let now $X$ be a standard subset. It can be written uniquely as $X =
\cup_{i=1}^n X_i$ where the $X_i$ are disjoint connected standard
subsets. Then we set $\A_F(X) = \oplus_{i=1}^n\A_F(X_i)$ and $\A^0_F(X) =
\oplus_{i=1}^n\A^0_F(X_i)$ where $F$ is a finite extension of $\Q_p$ such
that $X$ is well-defined over $F$.

\subsubsection{Field of definition and change of field}

Let $F$ be a finite extension of $\Q_p$. The field of definition of $X
\subset \PP^1(\bar\Q_p)$ over $F$ is the fixed field of $\{\sigma\in G_F,
\sigma(X) = X\}$. The field of definition of $X$ is the field of
definition of $X$ over $\Q_p$. Then $X$ is defined over $F$ (as in
Definition \ref{definedover}) if and only if $F$ contains the field of
definition of $X$.

Let $X$ be a standard subset defined over some finite extension $E$ of
$\Q_p$. Let $F$ be a finite Galois extension of $E$ such that $X$ is
well-defined over $F$.  In this case $\Gal(F/E)$ acts on $\A_F(X)$ and
$\A^0_F(X)$ by $(\sigma f)(x) = \sigma(f(\sigma^{-1}x))$.  We write
$\A_E(X)$ and $\A^0_E(X)$ for $\A_F(X)^{\Gal(F/E)}$ and
$\A^0_F(X)^{\Gal(F/E)}$. So for example, if $X = D(0,1)^-$, then $X$ is
defined over $\Q_p$, and $\A^0_E(X)$ is $\OO_E[[x]]$ for any finite
extension $E$ of $\Q_p$.
It is clear that the definition of $\A_E(X)$ and $\A_E^0(X)$ does not
depend on the choice of the extension $F$ over which $X$ is well-defined.

\begin{prop}
\label{funcunram}
Let $X$ be a standard subset defined over $E$.
Let $F$ be a finite extension of $E$. Then $\A_F(X) = F\otimes_E\A_E(X)$, 
and $\OO_F\otimes_{\OO_E}\A_E^0(X) \subset \A^0_F(X)$, with $\A_F^0(X)$
finite over $\OO_F\otimes_{\OO_E}\A_E^0(X)$.
If $F/E$ is unramified, then this inclusion is an isomorphism.
\end{prop}

\begin{proof}
When $X$ is well-defined over $E$, the formulas of Proposition
\ref{ringfuncs} give the isomorphisms $\A_F(X) = F\otimes_E\A_E(X)$ and
$\OO_F\otimes_{\OO_E}\A_E^0(X) = \A^0_F(X)$ (even if $F/E$ is ramified). 

We now treat the general case.
We define a map $\phi: F\otimes_E\A_E(X) \to \A_F(X)$ by $\phi(a\otimes
f) = af$. Let us describe the inverse $\psi$ of $\phi$.
Let $Q = G_E/G_F$. If $a$ is in $F$ and $f\in \A_F(X)$, $\sigma(a)$ and
$\sigma(f)$ are well-defined for $\sigma\in Q$ as $a$ and $f$ are
invariant by $G_F$. Moreover, for $a\in F$, we have that $\tr_{F/E}(a) =
\sum_{\sigma\in Q}\sigma(a)$.

Let $(e_1,\dots,e_n)$ be a basis of $F$ over $E$, and $(u_1,\dots,u_n)
\in F^n$ be the dual basis with respect to $\tr_{F/E}$, that is,
$\tr_{F/E}(e_iu_j) = \delta_{i,j}$. We see that for $\sigma\in
G_E$, we have $\sum_{i=1}^ne_i\sigma(u_i) = 1$ if $\sigma\in G_F$, and
$0$ otherwise.

For $f\in \A_F(X)$, we set $t_i(f) = \sum_{\sigma\in
Q}\sigma(u_if)$. Let $\psi(f) = \sum_{i=1}^ne_i\otimes t_i(f)$. 
Let us check that $\psi$ is the inverse of $\phi$.
Let $f\in \A_F(X)$, and $f' = \phi(\psi(f))$. Then $f' =
\sum_ie_i\sum_Q\sigma(u_i)\sigma(f) =
\sum_Q\sigma(f)\left(\sum_ie_i\sigma(u_i)\right)$, so $f' = f$.
Let $f\in \A_E(X)$, and $a\in F$. Let $g = \phi(a\otimes f)$.
Then $t_i(g) = \tr_{F/E}(au_i)f$, as $\sigma(f) = f$ for all $\sigma\in
Q$. So $\psi(g) = \sum_i e_i\otimes \tr_{F/E}(au_i)f =
\left(\sum_ie_i\tr_{F/E}(au_i)\right)\otimes f$ as $\tr_{F/E}(au_i)\in E$. Then
we check that $\sum_ie_i\tr_{F/E}(au_i) = a$, so $\psi(\phi(a\otimes f))
= a\otimes f$.
So we see that $\psi$ is the inverse map of $\phi$, so $\phi$ is an
isomorphism.

We see that $\phi$ induces a map $\phi^0$ from $\OO_F\otimes_{\OO_E}\A_E^0(X)$ 
to $\A^0_F(X)$. When $F/E$ is unramified, we can
choose $(e_i)$ and $(u_i)$ to be in $\OO_F$, and in this case the
restriction $\psi^0$ of $\psi$ to $\A^0_F(X)$ maps into
$\OO_F\otimes_{\OO_E}\A_E^0(X)$, and so $\psi^0$ is the inverse map of
$\phi^0$, and so $\phi^0$ is an isomorphism.
\end{proof}

\subsubsection{Some algebraic results}

Let $X$ be a standard subset of $\PP^1(\bar\Q_p)$ that is defined over $E$ for some
finite extension $E$ of $\Q_p$. Let $F$ be a finite extension of $E$. 
We say that $X$ is irreducible over $F$
if it cannot be written as a finite disjoint union of standard subsets
of $\PP^1(\bar\Q_p)$ that are defined over $F$.
There exists a unique decomposition of $X$ as a finite disjoint union of
standard subsets of $\PP^1(\bar\Q_p)$ that are irreducible over $F$.
A standard subset is connected if and only if it is irreducible over any field of
definition.

\begin{lemm}
\label{islocal}
Let $X$ be a connected standard subset of $\PP^1(\bar\Q_p)$ defined over
$E$. Then $\A_E(X)$ is a
domain, and $\A^0_E(X)$ is a local ring 
which has the same residue field as $E$.
\end{lemm}

\begin{proof}
Let $F$ be a finite Galois extension of $E$ such that $X$ is well-defined
over $F$. The result for $\A_F^0(X)$ holds from the description given in
Proposition \ref{ringfuncs}, and the result for $\A_E^0(X)$ follows from the fact
that it is equal to $\A_F^0(X)^{\Gal(F/E)}$ and the results of
Proposition \ref{funcunram}.
Note that the maximal ideal is the set of functions $f$ such 
that $|f(x)| < 1$ for all $x$ in $X$, that is, the functions $f$ that are
topologically nilpotent.
\end{proof}

\begin{lemm}
\label{irrcomp}
Let $X$ be a standard subset that is defined and irreducible over $E$, and let $X =
\cup_{i=1}^rX_i$ its decomposition in a finite union of connected
standard subsets. Let $F$ be the field of definition of $X_1$ over $E$.
Then the restriction map $\A^0(X) \to \A^0(X_1)$ induces an $\OO_E$-linear 
isomorphism $\A^0_E(X) \to \A^0_F(X_1)$. 
\end{lemm}

Note in particular that: $[F:E]$ is the number of connected components of
$X$, and the isomorphism class of $\A_F^0(X_1)$ as an $\OO_E$-algebra does
not depend on the choice of $X_1$.

\begin{proof}
The group $G_E$ acts transitively on the set of the
$(X_i)$ as $X$ is irreducible, and $G_F$ is the stabilizer of $X_1$. 
We fix a system $(\sigma_i)$ of representatives of $G_E/G_F$, numbered so
that $\sigma_i(X_1) = X_i$ for all $i$.

Let $f$ be an element of $\A^0_E(X)$. First note that $f$ is invariant
under the action of $G_F$, as it is by definition invariant under the
action of $G_E$, so $f_{|X_1}$ is in $\A_F^0(X_1)$. Moreover,
we have that for all $x\in X_i$, 
$$f(x) = 
\sigma_i((\sigma_i^{-1}f)(\sigma_i^{-1}(x)) 
= \sigma_i(f_{|X_1}(\sigma_i^{-1}x))
$$ 
So $f_{|X_i}$ is
entirely determined by $f_{|X_1}$, so the restriction map is injective,
and moreover for any $f\in \A_F^0(X_1)$ the formula above defines an
element of $\A_E^0(X)$, so the restriction map is bijective.
\end{proof}

\begin{coro}
\label{funcringirr}
If $X$ is defined and irreducible over $E$, then $\A_E(X)$ is a domain, and
$\A^0_E(X)$ is a local ring.
\end{coro}

\begin{proof}
We apply Lemma \ref{irrcomp}: $\A^0_E(X)$ is isomorphic as a ring to
$\A_F^0(X_1)$, which is local.
\end{proof}

\begin{defi}
If $X$ is defined and irreducible over $E$, we denote by $k_{X,E}$ the
residue field of $\A_E^0(X)$.
\end{defi}

By construction, $k_{X,E}$ is a finite extension of $k_E$. In the notation
of Lemma \ref{irrcomp}, we have $k_{X,E} = k_{X_1,F}$, and by Lemma
\ref{islocal}, $k_{X_1,F} = k_F$ as $X_1$ is connected.

\begin{ex}
\label{exresfield}
Let $a \in \Q_{p^2}$ such that $v_p(a) = 0$, and $\bar{a}$ is not in
$\F_p$. Let $a'$ be its Galois conjugate, so that the disks $D =
D(a,1)^-$ and $D' = D(a',1)^-$ are disjoint. Let $X$ be the union of $D$
and $D'$. Then $X$ is defined and irreducible over $\Q_p$, although it is not
connected. Moreover, $\A^0_{\Q_p}(X) = \A_{\Q_{p^2}}^0(D)$ is isomorphic
to $\Z_{p^2}[[w]]$ (where $w$ corresponds to $x-a$), so $k_{X,\Q_p} =
\F_{p^2}$.
\end{ex}

\subsection{Some maps from quasi-affinoid spaces to the unit disk}
\label{somemaps}

\begin{theo}
\label{immap}
Let $\X$ be a normal, Zariski geometrically connected quasi-affinoid
space over some finite extension of $\Q_p$, $\D$ be the closed unit disk,
and $f : \X \to \D$ a rigid analytic map that is an open immersion. Then
the image $f(\X)$ of $\X$ is a connected $R$-subset of $\D$, and $f$ is
an isomorphism from $\X$ to its image.
\end{theo}

Let us first recall what is known in the affinoid case. Let $\X$ and $\Y$ be
affinoid spaces, and let $f : \X \to \Y$ be a rigid analytic map which is
an open immersion. Then by \cite[Corollary 8.2/4]{BGRa}, the image of $f$
in $\Y$ is an affinoid subdomain of $\Y$ and $f$ is an isomorphism from
$\X$ to its image. 

But this does not hold in the quasi-affinoid case without extra
hypotheses, as illustrated by the following example: let $\X$ be the
disjoint union of the open unit disk and the unit circle, and $i$ the
natural map from $\X$ to the closed unit disk $\D$. Then $i$ is an open
immersion and is bijective, but is not an isomorphism (as $\X$ is not
connected, whereas the closed unit disk is).

We need some lemmas in order to prove Theorem \ref{immap}.

\begin{lemm}
\label{immbijadm}
Let $\X$ be a quasi-affinoid space and $\Y$ be a rigid space, both
defined over some finite extension of $\Q_p$, and let $f : \X \to \Y$ be
a rigid analytic map which is a surjective open immersion. 

Assume that there exists a covering $(\Y_i)$ of $\Y$
by affinoid subdomains, such that each $f^{-1}(\Y_i)$ is an affinoid
subdomain of $\X$, and $(\Y_i)$ is an admissible covering of $\Y$ (that
is, any affinoid subdomain $\Y'$ of $\Y$ can be covered by a finite
number of $\Y_i$).  Then $f$ is an isomorphism from $\X$ to $\Y$.
\end{lemm}

\begin{proof}
We need to construct the inverse $g : \Y \to \X$. It is enough to
construct $g' : \Y' \to \X$ satisfying $f \circ g' = \id$ for each
affinoid subdomain $\Y'$ of $\Y$ (as these are necessarily compatible and
glue to form $g$). Fix such a $\Y'$. Then it is covered by some $\Y_i$
for $i$ in some finite set $I$. Set $\X_i = f^{-1}(\Y_i)$, so that $\X_i$
is affinoid. As $\X$ is quasi-affinoid, there exists
some affinoid subdomain $\X'$ of $\X$ containing all the $\X_i$ for $i
\in I$. Note that we have, for each $i \in I$, a map $g_i : \Y_i \to
\X_i$ which is the inverse of the restriction of $f$ to $\X_i$, and these
are compatible. So they glue to form the function $g' : \Y' \to \X'$ by the
admissibility condition and Tate's Acyclicity Theorem (see
\cite[Corollary 8.2/3]{BGRa}).
\end{proof}

Looking back at the inclusion $i : \X \to \D$ as above, we see that $i$
is a surjective open immersion, but there does not exist a covering of
$\D$ satisfying the conditions of Lemma \ref{immbijadm}. Indeed, let $\Y$
be a connected affinoid subdomain of $\D$, then $i^{-1}(\Y)$ is affinoid
if and only if $\Y$ is either contained in the unit circle or in the open
unit disk. But it is not possible to have an admissible covering of $\D$
by affinoids satisfying this condition.

\begin{coro}
\label{injfiniiso}
Let $\X$ and $\Y$ quasi-affinoid spaces, both
defined over some finite extension of $\Q_p$, and let $f : \X \to \Y$ be
a rigid analytic map which is finite and an open immersion. 
Assume that $\Y$ is connected and $\X$ is non-empty, then $f$ is an isomorphism.
\end{coro}

\begin{proof}
Assume first that $\Y$ is affinoid. Then so is $\X$ as $f$ is finite,
and so $f$ is an isomorphism by \cite[Corollary 8.2/4]{BGRa}. In general,
$\Y$ has an admissible covering $(\Y_i)$ by connected affinoid
subdomains. Let $\X_i = f^{-1}(\Y_i)$, then each $\X_i$ is affinoid as $f$
is finite. Moreover, for each $i$, either $\X_i$ is empty or $f$ induces
an isomorphism between $\X_i$ and $\Y_i$. By connectedness of $\Y$ and
the fact that $\X$ is non-empty, we get that $f(\X_i) = \Y_i$ for all $i$
and in particular $f$ is surjective. So the conditions of Lemma
\ref{immbijadm} are satisfied.
\end{proof}

\begin{lemm}
\label{injfini}
Let $\X$ and $\Y$ be quasi-affinoid rigid spaces, $f : \X \to \Y$ a
quasi-affinoid map. Assume that $f$ is an open immersion.
There exists a finite
covering $(\Y_i)$ of $\Y$ by connected $R$-subdomains such that for each $i$, 
either $f^{-1}(\Y_i)$ is empty, or $f$
induces an isomorphism from $f^{-1}(\Y_i)$ to $\Y_i$.
\end{lemm}

\begin{proof}
As $f$ is an open immersion, it is in particular quasi-finite. So we can apply
\cite[Theorem 6.1.2]{LR00}: there exists a finite covering $(\Y_i)$ of $\Y$
by $R$-subdomains such that $f$ induces a finite map $f_i$ from $\X_i =
f^{-1}(\Y_i)$ to $\Y_i$. We can assume that each $\Y_i$ is connected. 
By Corollary \ref{injfiniiso}, for each $i$ we have that either $\X_i$ is
empty or $f_i$ is an isomorphism.
\end{proof}

\begin{proof}[Proof of Theorem \ref{immap}]
Let $f : \X \to \D$ be as in the statement of the Theorem.  First observe
that $f$ is a quasi-affinoid map from $\X$ to $\D$, as it is a bounded
analytic function on $\X$ and $\X$ is normal.  By Lemma \ref{injfini},
there exists a finite covering $(\Y_i)$ of $\D$ by $R$-subdomains of $D$
such that for all $i$, either $f^{-1}(\Y_i)$ is empty or $f$ induces an
isomorphism from $f^{-1}(\Y_i)$ to $\Y_i$. By Lemma \ref{Rspecial}, we
can assume that each $Y_i = \Y_i(\bar\Q_p)$ is a special subset of
$\bar\Q_p$.  We see $\X$ and $f$ as defined on some finite extension $F$
of $\Q_p$ that is large enough so that each of the $Y_i$ are well-defined
over $F$.  Let $Y$ be the union of the $Y_i$ for those $i$ such that
$f^{-1}(\Y_i)$ is not empty.  We see that $Y$ is a finite union of
$R$-subsets of $\bar\Q_p$, and is equal to $f(\X(\bar\Q_p))$. As $\X$ is
connected, so is $Y$, and so $Y$ is in fact a connected $R$-subset of
$\bar\Q_p$ and so is the set of points of a quasi-affinoid subdomain $\Y$
of $\D$.

We now want to prove that $f$ induces an isomorphism between $\X$ and
$\Y$. We want to apply Lemma \ref{immbijadm}, and so we want to construct
an appropriate covering of $\Y$. It is the same to work with
quasi-affinoid subdomains or their sets of points, so from now on we work
with subsets of $\bar\Q_p$.

We write the family $(Y_i)$ as $(S_i) \cup (A_i)$, where $S_i$ are
subsets of the form (2) of Definition \ref{defspecial} 
(and hence affinoid), and $A_i$ are subsets of
the form (1). We can cover each $A_i = \{ x, r < |x-a| < r'\}$ by a
family of affinoid subsets $A_{i,\eta} = \{ x, r/\eta \leq |x-a| \leq
r'\eta\}$ for $\eta>1$, $\eta\in \sqrt{|F^\times|}$, $\eta$ close enough
to $1$.  So we get a covering of $Y$ by affinoid subsets, that is, the
$S_i$ and the $A_{i,\eta}$. Their inverse images in $\X$ are affinoid, as
each of them is contained in one of the $Y_i$.

This covering is not necessarily admissible, so we add
some other affinoid subsets of $Y$ in order to get an admissible
covering. We know that if $\z$ is an affinoid subset of $\X$, then
$f(\z)$ is an affinoid subdomain of $\D$ and $f$ induces an isomorphism
between $\z$ and $f(\z)$ by \cite[Corollary 8.2/4]{BGRa}.
Let $\mathcal{C}$ be the covering of $Y$ by the union of families of
elements $(S_i)$, $(A_{i,\eta})$, and all the sets $f(\z)(\bar\Q_p)$ for
$\z$ an affinoid subset of $X$. We want to show that $\mathcal{C}$ is an
admissible covering of $Y$. Then it will satisfy the conditions of Lemma
\ref{immbijadm} and so the conclusion will follow.

Write $Y$ as $D(a_0,r_0) \setminus \cup_{i=1}^mD(a_i,r_i)$, where each of
the disks is rational and either open or closed and $a_0\in Y$. Let
$\eta>1$, $\eta\in\sqrt{|F^\times|}$. We set
$r_{0,\eta} = r_0$ if $D(a_0,r_0)$ is closed and $r_0/\eta$ otherwise, and
for $i>0$ let $r_{i,\eta} = r_i$ if $D(a_i,r_i)$ is open and $r_{i,\eta} =
r_i\eta$ otherwise. Let $Y_\eta = D(a_0,r_{0,\eta})^+ \setminus
\cup_{i=1}^mD(a_i,r_{i,\eta})^-$, so that $Y_\eta$ is an affinoid contained in
$Y$ (for $\eta$ close enough to $1$), and the family $(Y_\eta)$ forms an admissible
covering of $Y$. So it is enough to show that each $Y_\eta$ can be covered
by a finite number of elements of $\mathcal{C}$.

For each $0 \leq i \leq m$, let $b_i \in Y$ be such that $|a_i-b_i| =
r_{i,\eta}$. Let $c_i$ be an element of $A_i$ for each $i$. Writing $A_i$
as $\{r < |x-a| < r'\}$, we choose some $c'_i$ in $Y \cap D(a,r)^+$ if it
is not empty.  By \cite{Liu}, as $\X$ is connected, there is a connected
subset $Z$ of $\X$ that is a finite union of affinoid subdomains of
$\X$, such that $f(Z)$ contains $a_0$ and each of the $b_i$,
$c_i$ and $c'_i$.  Let $Z' = f(Z)$. Then it is a finite union
of elements of $\mathcal{C}$, and a finite union of connected closed
$R$-subsets, as it is a finite union of images of
affinoid subsets of $\X$. As $Z$ is connected, so is $Z'$, so it is 
a connected closed $R$-subset by Lemma \ref{interstd}.
By construction, there is a finite number of open disks $(D_i)$ that do
not meet $Z'$ such that $D_i \subset Y$ and $Y_\eta$ is contained in $Z'
\cup (\cup_i D_i)$.

So it suffices to show that each $D_i$ can be covered by a finite number
of elements of $\mathcal{C}$. If $D_i$ does not meet any $A_j$, then it
is covered by the elements of $\mathcal{C}$ of the form $S_j$. If $D_i$
meets $A_j$, then as $D_i$ does not contain $c_j$ (nor $c_j'$), then $D_i
\subset A_j$ and so $D_i$ is covered by $A_{j,t}$ for some $t>0$.
\end{proof}

\begin{coro}
\label{qaffstandard}
Let $\X$ be a normal rigid space that is quasi-affinoid space of open type
over some finite extension $E$ of $\Q_p$.  Let $\D$ be the rigid closed
unit disk.
Let $f : \X \to \D$ be a rigid analytic map over $E$ that
is an open immersion.  Let $Y = f(\X)(\bar\Q_p)$.  Then
$Y$ is an $R$-subset of open type defined over $E$, and if $\X$ is
geometrically Zariski connected then $Y$ is a connected $R$-subset.
Moreover, $f$ induces an $E$-algebra isomorphism between $\A_E(Y)$ and
$\A_E(\X)$, and between $\A_E^0(Y)$ and $\A_E^0(\X)$.
\end{coro}

\begin{proof}
Let $F$ be an finite extension of $E$ that is large enough so that each
geometric Zariski connected component is defined over $F$, and $F/E$ is
Galois.  

Write $\X$ as a disjoint union of $\X_i$ where each $\X_i$ is geometrically
Zariski connected. Let $f_i$ be the restriction of $f$ to $\X_i$, it is
still an open immersion, and is defined over $F$. We apply Theorem
\ref{immap} to $f_i$: $f_i$ induces an isomorphism between $\X_i$ and its
image $f(\X_i) = \Y_i$. In particular, $\A_F(\Y_i)$ and $\A_F(\X_i)$ are
isomorphic by the map $f_i^\#$. As $\X$ is of open type, so is $\X_i$ and
hence so is $\Y_i$.  By Proposition \ref{isopentype}, this implies that
$Y_i = \Y_i(\bar\Q_p)$ is a connected $R$-subset of open type. Moreover, the
$Y_i$ are disjoint as $f$ is injective. Let $Y$ be the disjoint union of
the $Y_i$.

So we get an $F$-algebra isomorphism $f^\#$ between $\A_F(Y) =
\oplus_{i=1}^n\A_F(Y_i)$ and $\A_F(\X)$, which is equal to
$\oplus_{i=1}^n\A_F(\X_i)$. 
As $\X$ is defined over $E$ and $f$ is an $E$-morphism, we see that $\Y$ is
defined over $E$.
We have an action of $\Gal(F/E)$ on both sides, and $f^\#$ is
$\Gal(F/E)$-equivariant. So $f^\#$ induces an isomorphism between the
$\Gal(F/E)$ invariants on both sides, hence the result.  
\end{proof}

\section{Complexity of standard subsets}
\label{complexity}

\subsection{Algebraic complexity of a standard subset over a field of definition}

\subsubsection{Definition}

Recall that we defined $\e$ in Section \ref{HSspecial}.

\begin{defi}
\label{defalgcomp}
Let $X$ be a standard subset of $\PP^1(\bar\Q_p)$ that is defined over $E$. If $X$
is irreducible over $E$, we define the complexity of $X$ over $E$ to be:
$$c_E(X) = [k_{X,E}:k_E]\e_{\OO_E}(\A^0_E(X))$$
In general, let
$X = \cup_{i=1}^rX_i$ be the decomposition of $X$ as a disjoint union of
standard subsets that are defined and irreducible over $E$. We define the
complexity of $X$ over $E$ to be $c_E(X) = \sum_{i=1}^rc_E(X_i)$.
\end{defi}

The above definition makes sense
as $\A^0_E(X)$ is a complete noetherian local $\OO_E$-algebra if $X$ is
irreducible over $E$ by Corollary \ref{funcringirr}.

Note that in particular if $X$ is connected then $c_E(X) =
\e_{\OO_E}(\A^0_E(X))$ as $k_{X,E} = k_E$ in this case.

\subsubsection{Some general results on algebraic complexity}

We now give explicit formulas for the complexity. It is enough to give
such formulas for subsets $X$ that are irreducible over $E$.

\begin{prop}
\label{ccomponent}
In the situation of Lemma \ref{irrcomp}, we have
$c_E(X) = [F:E]c_F(X_1)$.
\end{prop}

Note that $c_F(X_1)$ does not depend on the choice of $X_1$ among the connected
components.

\begin{proof}
Let $e_{F/E}$ be the ramification degree of $F/E$.
We have that $\A_F^0(X_1) = \A_E^0(X)$ as $\OO_E$-algebras, 
and $k_{X,E} = k_{X_1,F} = k_F$.
So $c_E(X) =
[k_F:k_E]\e_{\OO_E}(\A_E^0(X)) = [k_F:k_E]\e_{\OO_E}(\A_F^0(X_1))$ which is equal to
$[k_F:k_E]e_{F/E}\e_{\OO_F}(\A_F^0(X_1)) = [F:E]c_F(X_1)$ by Proposition
\ref{changeringram}.
\end{proof}

\begin{prop}
\label{changefieldconnected}
Let $X$ be a connected standard subset defined over $E$, and $F$ a
finite extension of $E$. Then $c_E(X) \geq c_F(X)$ with equality when
$F/E$ is unramified.
\end{prop}

\begin{proof}
From Lemma \ref{changeringany} we see that
$\e(\OO_F\otimes_{\OO_E}\A_E^0(X)) = \e(\A_E^0(X)) = c_E(X)$,
and from Propositions \ref{funcunram} and \ref{multsmaller} we see that
$\e(\OO_F\otimes_{\OO_E}\A_E^0(X)) \geq \e(\A_F^0(X))$ with equality when
$F/E$ is unramified.
\end{proof}

\begin{prop}
\label{changefieldany}
Let $X$ be a standard subset defined over $E$, and $F$ a
finite extension of $E$. Then $c_E(X) \geq c_F(X)$ with equality when
$F/E$ is unramified.
\end{prop}

\begin{proof}
By additivity of the complexity we can assume that $X$ is irreducible
over $E$. 
Write $X = \cup_{i=1}^nX_i$ where each $X_i$ is connected.
Let $E_i$ be the field of definition of $X_i$ over $E$, so that 
$c_E(X) = nc_{E_i}(X_i)$ for all $i$.
Then $FE_i$ is
the field of definition of $X_i$ over $F$.
Suppose that
the action of $G_F$ on the set of the irreducible components of $X$ has
$r$ orbits, with representatives say $X_1,\dots,X_r$.
Then 
$c_F(X) = \sum_{j=1}^r[FE_j:F]c_{FE_j}(X_j)$. We have that 
$c_{FE_j}(X_j) \leq c_{E_j}(X_j)$ by Proposition
\ref{changefieldconnected}, and $c_{E_j}(X_j)$ is independent of $j$, and
equal to $(1/n)c_E(X)$. Moreover, $[FE_j:F]$ is the cardinality of the
orbit of $X_j$, so $\sum_{j=1}^r[FE_j:F] = n$. Finally we get that
$c_F(X) \leq c_E(X)$, with equality if and only if 
$c_{FE_j}(X_j) = c_{E_j}(X_j)$ for all $j$, which happens in particular
if $F/E$ is unramified.
\end{proof}

\subsubsection{Does $c_E(X)$ characterize $\A_E^0(X)$?}
\label{recover}

We ask the following question: let $X$ be defined and irreducible over
$E$. Let $R \subset \A_E^0(X)$ be a local, noetherian, complete,
$\OO_E$-flat
$\OO_E$-subalgebra of $\A_E^0(X)$, such that $R[1/p] = \A_E(X)$. Suppose
moreover that $R$ and $\A_E^0(X)$ both have residue field $k_E$, and $\e(R)
= \e(\A_E^0(X))$, that is $\e(R) = c_E(X)$. Do we have $R = \A_E^0(X)$ ?

It follows from \cite[Lemme 5.1.8]{BM} that the equality holds if
$c_E(X) = 1$, and in this case both rings are isomorphic to $\OO_E[[x]]$,
and $X$ is a disk of the form $\{x,|x-a|<|b|\}$ for some $a,b\in E$.

But as soon as $c_E(X) >1$ there are counterexamples. We give a few, with
$E = \Q_p$.
\begin{enumerate}

\item
Let $X = \{x,0<v_p(x)<1\}$. 
Then $\A_{\Q_p}^0(X)$ is isomorphic to   
$\Z_p[[x,y]]/(xy-p)$. Let $R$ be the closure of the subring generated by 
$px$, $py$ and $x-y$. 
Here $\e(R) = c_{\Q_p}(X) = 2$.

\item
Let $X = \{x,v_p(x)>1/2\}$. 
Then $\A_{\Q_p}^0(X)$ is isomorphic to   
$\Z_p[[x,y]]/(x^2-py)$.
Let $R$ be the closure of the subring generated by 
$y$ and $px$.
Here $\e(R) = c_{\Q_p}(X) = 2$.

\item 
Let $X = \{x,|x-\pi|<|\pi|\}$ where $\pi^p=p$. 
Then $\A_{\Q_p}^0(X)$ is isomorphic to   
$\Z_p[[x,y]]/(x^p-p(y+1))$.
Let $R$ be the closure of the subring generated by 
$y$ and $px$.
Here $\e(R) = c_{\Q_p}(X) = p$.

\end{enumerate}

\subsection{Computations of the algebraic complexity in some special cases}

\subsubsection{Preliminaries}

If $P \in E[x]$, and $a\in \bar{\Q}_p$, let $P_a(x)=P(x+a) \in \bar{\Q}_p[x]$.

\begin{lemm}
\label{NPopen}
Let $D$ be an open disk defined over $E$, let $s$ be the smallest degree
over $E$ of an element in $D$. Let $a$ be an element of $D$ of degree $s$
over $E$.
Let $\lambda\in \R$ be such that $D = \{x, v_E(x-a) > \lambda\}$. 

Let $P\in E[x]_{<s}$, and write $P_a(x) = \sum_{i=0}^{s-1}b_ix^i$. 
Then: $v_E(b_i) \geq v_E(b_0) - i\lambda$ for all $i$.
In particular,
if $v_E(b_0) \geq 0$, then $v_E(b_i) \geq -i\lambda$ for all $i>0$, and
if $v_E(b_0) > 0$, then $v_E(b_i) > -i\lambda$ for all $i>0$.
\end{lemm}

\begin{proof}
Consider the Newton polygon of $P_a$: if the conclusion of the Lemma is
not satisfied, then it has at least one slope $\mu$ which is $<
-\lambda$. So $P_a$ has a root $y$ of valuation $-\mu > \lambda$. Let $b
= a+y$, then $b$ is a root of $P$, so of degree $< s$ over $E$. On the
other hand, $v_E(b-a) = v_E(y) > \lambda$ so $b$ is in $D$, which contradicts
the definition of $s$.
\end{proof}

A similar proof shows:

\begin{lemm}
\label{NPclosed}
Let $D$ be a closed disk defined over $E$, let $s$ be the smallest degree
over $E$ of an element in $D$. Let $a$ be in $D$ of degree $s$ over $E$.
Let $\lambda\in \R$ be such that $D = \{x, v_E(x-a) \geq \lambda\}$. 

Let $P\in E[x]_{<s}$, and write $P_a(x) = \sum_{i=0}^{s-1}b_ix^i$. 
Then: $v_E(b_i) > v_E(b_0) - i\lambda$ for all $i>0$.
In particular,
if $v_E(b_0) \geq 0$, then $v_E(b_i) > -i\lambda$ for all $i>0$.
\end{lemm}

\begin{defi}
\label{defregular}
Let $L/\Q_p$ be a finite extension. Let $f\in\OO_L[[w]]$, $f =\sum_{i\geq
0}f_iw^i$. We say that $f$ is regular of degree $n$ if
$f_n\in\OO_L^\times$ and $f_m\in\m_L$ for all $m<n$. 
\end{defi}

\begin{defi}
\label{valdom}
Let $L/\Q_p$ be a finite extension. Let $f\in\OO_L[[w]]$, $f =\sum_{i\geq
0}f_iw^i$. 
We define the valuation of $f$ as $v_E(f) = \min_iv_E(f_i)$, and the
leading term of $f$ as $w^i$ for the smallest $i$ such that $v_E(f) =
v_E(f_i)$.  In particular, $f$ is regular of degree $n$ if and only if
$v_E(f) = 0$ and the leading term of $f$ is $w^n$.
\end{defi}

We recall the following result (see for example \cite[Proposition
7.2]{Wash}):

\begin{lemm}[Weierstrass Division Theorem]
\label{division}
Let $f\in\OO_L[[w]]$ that is regular of degree $n$, and $g\in\OO_L[[w]]$. Then
there exists a unique pair $(q,r)$ with $q\in\OO_L[[w]]$, $r \in
\OO_L[w]_{<n}$ and $g = qf+r$.
\end{lemm}

Let $X$ be a connected standard subset defined over $E$. Then we have the
easy but useful result:

\begin{lemm}
\label{reducnorm}
Let $f \in \A_E^0(X)$. Then $f$ reduces to $0$ in $\A_E^0(X)/(\pi_E)$ if
and only if $\| f \|_X \leq |\pi_E|$. The image of $f$ in
$\A_E^0(X)/(\pi_E)$ is nilpotent if and only if $\| f \|_X < 1$.
\end{lemm}

\subsubsection{Open disks}

We want to give the general formula for the complexity of a disk. We
start with some examples.
We see that there are two kinds of difficulties: one from the radius that
is not necessarily the norm of an element of $E$, and one from the fact
that the disk does not necessarily contain an element in $E$.

\begin{ex}
\label{disk1}
Let $a$, $b$ be in $E$ with $b \neq 0$. Let $D$ be the disk
$\{x, |x-a| < |b|\}$. Then $c_E(D) = 1$. Indeed, $\A_E^0(D)$ is
isomorphic to $\OO_E[[w]]$, where $w$ corresponds to the function
$(x-a)/b$, so $\A_E^0(D)/(\pi_E) = k_E[[w]]$.
\end{ex}

\begin{ex}
\label{disk2}
Let $D$ be the disk $\{x, v_p(x) > 1/2 \}$. Then $c_{\Q_p}(D) = 2$.
Indeed, $\A_{\Q_p}^0(D)$ is isomorphic to $\Z_p[[w,t]]/(t^2-pw)$, where
$w$ corresponds to the function $x^2/p$ and $t$ to the function $x$. So
$\A_{\Q_p}^0(D)/(p)$ is isomorphic to $\F_p[[w,t]]/(t^2)$.
\end{ex}

\begin{ex}
\label{disk3}
Let $D = \{x,|x-\pi|<|\pi|\}$ where $\pi^p=p$. 
Then $c_{\Q_p}(D) = p$.
Indeed
$\A_{\Q_p}^0(D)$ is isomorphic to   
$\Z_p[[t,w]]/(t^p-p(w+1))$, where $t$ is the function $x$ and $w$ is the function $(x^p-p)/p$.
\end{ex}

\begin{prop}
\label{multdisk}
Let $D$ be an open disc of radius $r\in p^\Q$ defined over $E$.
Let $s$ be the smallest ramification degree of $E(a)/E$ for $a\in D$.
Let $t$ be the smallest positive integer such that $r^{st}\in
|E(a)^\times|$.
Then $c_E(D) = st$.
\end{prop}

\begin{proof}
There are two steps in the proof: the first is to find a description of
$\A_E^0(D)$, and the second to use this description to show that
$\A_E^0(D)/(\pi_E)$ satisfies the conditions of Corollary
\ref{computeebis} and apply this to compute $c_E(D)$.

\smallskip

\noindent {\bf Step 1:}
Let $a \in D$ be as in the statement. As the complexity does not change
by unramified extensions by Proposition \ref{changefieldany}, 
we can enlarge $E$ so that $E(a)/E$ is totally
ramified. Let $\mu$ be the minimal polynomial of $a$ over $E$, so that
$\mu$ has degree $s$. Write $F = E(a)$.
For $\nu\in \Q$, let $F_\nu$ be the set $\{x\in F, v_E(x) \geq \nu\}$ (so
that $F_0 = \OO_F$).

Let $\lambda$ be such that $D = \{x, v_E(x-a) > \lambda\}$.
Let also $\rho\in F$ such that $v_E(\rho) = st\lambda$, which is
possible by the definition of $t$. When $s>1$, we see that $v_E(a) \leq
\lambda$ (otherwise $0 \in D$), and if $a'$ is another root of $\mu$ then
$v_E(a-a') > \lambda$ as $D$ is defined over $E$.

For $n \in \Z$, let $\E_n$ be the subset of $E[x]_{<s}$ of polynomials that can be
written as $\sum_{i=0}^{s-1}b_i(x-a)^i$ with $v_E(b_i) \geq -(i+ns)\lambda$.
Note that by Lemma \ref{NPopen}, $\E_n$ is the
set of polynomials in $E[x]_{<s}$ with $v_E(b_0) \geq -ns\lambda$. 
In fact $\E_n$ is in bijection with the set $F_{-ns\lambda}$ by $P
\mapsto P(a)$, as any element of $F$ can be written uniquely as $P(a)$
for some $P\in E[x]_{<s}$.
Note that $\rho^{-1}\in F_{-st\lambda}$. We fix $R\in\E_t$ the unique
polynomial such that $R(a) = \rho^{-1}$. We set $\alpha = R\mu^t$. 

Let $L$ be a Galois extension of $E$ containing $F$ and an element $\xi$
such that $v_E(\xi) = \lambda$.
Then $\A_L^0(D)$ is isomorphic to $\OO_L[[w]]$, with $w$
corresponding to $(x-a)/\xi$.
We consider now $\alpha$ as a polynomial in $w = (x-a)/\xi$. Then an easy
computation shows that $\alpha \in \OO_L[w]$, and it is a polynomial of
degree at most $st+s-1$ which is regular of degree $st$ in the sense of
Definition \ref{defregular} when seen as an element of
$\A_L^0(D) =\OO_L[[w]]$.

Let $\E'$ be the subset of $E[x]_{<st}$ of polynomials that can be
written as $\sum_{i=0}^{st-1}b_i(x-a)^i$ with $v_E(b_i) \geq -i\lambda$.
Then 
$$
\A_E^0(D) = \{\sum_{n\geq 0}P_n\alpha^n, P_n\in
\E'\}
$$
and any element of $\A_E^0(D)$ can be written uniquely in such a
way. Indeed: Let $f\in\A_E^0(D)$, which we see as an element of $\A_L^0(D) =
\OO_L[[w]]$. Applying repeatedly the Weierstrass Division Theorem, $f$ can
be written uniquely as $\sum_{n\geq 0} P_n\alpha^n$ with
$P_n\in\OO_L[w]_{<st}$. The fact that $f$ is in $\A_E^0(D)$ means that $f$
is invariant under $\Gal(L/E)$. As $\alpha$ itself is invariant under this
group, this means that each $P_n$ is invariant, and so $P_n\in \E'$
(where we see $\E'\subset \OO_L[w]_{<st}$ by $w = (x-a)/\xi$).

\bigskip

\noindent {\bf Step 2:}
We now want to check to conditions of Corollary \ref{computeebis}.
We observe that $\E' = \bigoplus_{j=0}^{t-1}\mu^j\E_j$.
For $0\leq i < t$, let $(u_{i,j})_{1\leq i \leq s}$ be a basis of $\E_j$
as an $\OO_E$-module, where we take $u_{1,0} = 1$, and $v_E(u_{i,0}(a)) > 0$
for $i>1$. We can satisfy this condition as taking a basis of $\E_0$
is the same as taking a basis of $\OO_F$ over $\OO_E$, and $F$ is totally
ramified over $E$. We also observe that for $j >0$, we have
$v_E(u_{i,j}(a)) > -js\lambda$ by definition of $t$.

Write $y_{i,j} = u_{i,j}\mu^j$ and $z = \alpha$ (note that $y_{1,0}=1$). 
Then $\A_E^0(D)$ is a quotient of $\OO_E[[y_{i,j},z]]$, hence the ring
$A = \A_E^0(D)/(\pi_E)$ is a quotient of $k_E[[y_{i,j},z]]$. 
Let $\bar{y}_{i,j}$, $\bar{z}$ be the images of $y_{i,j}$, $z$ in $A$. 
Let $I$ be the ideal generated by the $\bar{y}_{i,j}$ for
$(i,j) \neq (1,0)$. Then the maximal ideal $\m$ of $A$ is generated by
$I$ and $\bar{z}$.

We show first that $I$ is nilpotent. We see $y_{i,j}$ as an element of
$\A_L^0(D) = \OO_L[[w]]$, then $\|y_{i,j}\|_X = \max_na_n$ where $y_{i,j}
= \sum_{n\geq 0}a_nw^n$. So we see that for $(i,j) \neq (1,0)$, we have
that $\|y_{i,j}\|_X < 1$, and so $y_{i,j}$ is nilpotent by Lemma
\ref{reducnorm}. 

Let us see now that $A$ has no $\bar{z}$-torsion. As before, we see
$\A_E^0(X)$ as a subalgebra of $\OO_L[[w]]$. From the existence of this
inclusion, we see that the norm on $\A_E^0(X)$ is actually
multiplicative. As $\|z\|_X= 1$, we deduce that $\|zf\|_X = \|f\|_X$ for
all $f \in \A_E^0(X)$, and so $A$ has no $\bar{z}$-torsion.

We deduce that the conditions of Corollary \ref{computeebis} are
satisfied. 
So $e(A) = \dim_kA/(\bar{z})$, and we see easily that $1$ and 
the $\bar{y}_{i,j}$, $1\leq i \leq s$ and $0 \leq i < t$, $(i,j) \neq
(1,0)$, form a $k$-basis of $A/(\bar{z})$.
\end{proof}

\subsubsection{Holes}

\begin{prop}
\label{multhole}
Let $X = \PP^1(\bar\Q_p) \setminus T$ where 
$T = \cup_{i=1}^N D_i$ is a $G_E$-orbit of closed disks of positive radius
$r\in p^\Q$, with each disk defined over a totally ramified extension of $E$.
Let $K$ be the field of definition of $D_1$.
Let $s$ be the smallest ramification degree of $K(a)/K$ for $a\in D_1$.
Let $t$ be the smallest positive integer such that $r^{st} \in |E(a)^\times|$. 
Assume that $K(a)/E$ is totally ramified.
Then $c_E(X) = Nst$.
\end{prop}

When $N = 1$, that is, when $T$ is a disk, then the formula and the proof
are similar to what happens in Proposition \ref{multdisk}.  But there are
additional difficulties when there is not only one hole, but a whole
Galois orbit of them, that is, when $N > 1$. 

\begin{proof}[Proof of Proposition \ref{multhole}]
We divide the proof in several steps.

\noindent {\bf Step 1:} We first give a description of the ring of functions in the
case where $N=1$, that is, when there is only one hole.
Write $X' = \PP^1(\bar\Q_p)\setminus D_1$, so that $X'$ is defined over
$K$.  We will give a description of the ring $\A_K^0(X')$, forgetting $D$
and $E$ for the moment.  This computation is similar to the computation
in Proposition \ref{multdisk}, although complicated by the fact that we
work with rational fractions and not only with polynomials.

Let $a \in D_1$ as in the statement of the Proposition. Note that $[K:E]
= N$.  Let $F = E(a)$. Note that $K \subset F$ so $E(a) = K(a)$. By
hypothesis, $F/E$ is totally ramified.  We write $[F:K] = s$.
Write $D_1$ as the set $\{x, v_E(x-a)\geq \lambda\}$ for some $\lambda\in
\Q$.  Let $\mu$ be the minimal polynomial of $a$ over $K$, so that $\mu$
has degree $s$. Let also $\rho\in F$ be such that $v_E(\rho) = st\lambda$,
which is possible by the definition of $t$.

Let $R$ be the unique element of $K[x]_{<s}$ such that $R(a) = \rho$.
Note that when we write $R(x) = \sum b_i(x-a)^i$, we have $v_E(b_i) >
(st-i)\lambda$ for all $i>0$ by Lemma \ref{NPclosed}.
Set $v = R/\mu^t$, and 
for $n\geq 1$, set $\alpha_n =
\rho\mu^{-t}v^{n-1}$.

Let $L$ be an extension of $E$ containing $a$ and an element $\xi$ such
that $v_E(\xi) = \lambda$,
and which is Galois over $E$.
Note that $\A_L^0(X')$ is isomorphic to $\OO_L[[w]]$, with $w$
corresponding to the function $\xi/(x-a)$. In this isomorphism,
observe that
$\alpha_n$ is regular of degree $nst$
and is divisible by $w^{st}$, 
and
$v = R\mu^{-t} = \rho^{-1}R\alpha_1$ is regular of degree $st$, and
divisible by $w^{st-s+1}$.

Let $f = wg \in w\A_L^0(X')$. Then by applying Lemma \ref{division} repeatedly
we can write $w^{st-1}f =
w^{st}g$ as $\sum_{n\geq 1}P_n(w)\alpha_n$ for $P_n\in\OO_L[w]_{<st}$
(there is no remainder as $w^{st}$ and $\alpha_1$ differ by a unit).
So $f = \sum_{n\geq 1}w^{1-st}P_n(w)\alpha_n$. 
So any element of $\A_L^0(X')$ can be written uniquely as 
$f = a_0 + \sum_{n\geq 1}\rho w^{1-st}P_n(w)\mu^{-t}v^{n-1}$, where $a_0
\in \OO_L$ and $P_n(w) \in \OO_L[w]_{<st}$.

We want to know when such an element is in $\A_K(X')$. As $v$ and
$\mu^{-t}$ are in
$\A_K(X')$, we see that it is the case if and only if $a_0 \in \OO_K$ and
$\rho w^{1-st}P_n(w)$ is invariant under the action of $\Gal(L/K)$.
Note that $\rho w^{1-st}P_n(w)$ in actually in $L[x]_{<st}$, so it is
invariant by $\Gal(L/K)$ if and only if it is in $K[x]$.
Let $\E'$ the set of elements $Q\in K[x]_{<st}$ such that when we write
$Q(x) = \sum_{i\geq 0}b_i(x-a)^i$, we have $v_E(b_i) \geq (st-i)\lambda$.
Then we see that $\E'$ is exactly the set of elements of $K[x]_{<st}$
that are of the form $\rho w^{1-st}P(w)$ for some $P\in \OO_L[w]_{<st}$.

Then we have shown that:

$$
\A_K^0(X') = \left\{a_0 + \sum_{n\geq 0}
\frac{Q_n(x)}{\mu(x)^t}v^n,\;
a_0\in\OO_K,Q_n\in\E'\right\}
$$
where $v = R/\mu^t$.

We make the following observation: if $f \in \A_K^0(X')$ is written as
$a_0 + \sum_{n\geq 0} \frac{Q_n(x)}{\mu(x)^t}v^n$ with
$a_0\in\OO_K$, $Q_n\in\E'$, then 
\begin{equation}
\label{normmax}
\|f\|_{X'} = \max(|a_0|,\max_n\|Q_n\mu^{-t}\|_{X'}).
\end{equation}
We can see $f$ as being in $\OO_L[[w]]$ and reason in terms of $v_E(f)$.
We have that $v_E(v) = 0$ as $v$ is regular of degree $st$.
Moreover, writing $Q_n\mu^{-t}$ as $\rho^{-1}Q_n\alpha_1$, we see that the
leading term of $Q_n\mu^{-t}$ is $w^{st}$. Using this, we see easily that
$v_E(f) = \min(v_E(a_0),\min_n(v_E(Q_n\mu^{-t})))$ which gives the result.

\bigskip
\noindent{\bf Step 2:} We introduce the tools that allow us to go from the
description of $\A_K^0(X')$ to the description of $\A_E^0(X)$.

Let $Q = \{\sigma_1,\dots,\sigma_N\}$ be a system of representatives in
$G_E$ of
$G_E/G_K$, numbered so that $\sigma_i D_1 = D_i$ (so we take $\sigma_1 =
\id$). 
Recall that $w = \xi/(x-a)$. Let $\xi_i$, $a_i$ be conjugates of $\xi$ and
$a_i$ by $\sigma_i$, and let $w_i = \xi_i/(x-a_i)$ (so that $w_i = w$).
If $f \in \OO_L[[w]] = \A_L^0(X') \subset \A_L^0(X)$, with $f = \sum
f_nw^n$, we denote by $\tr f
\in \A_L^0(X)$ the element $\sum_{i=1}^N\sum_n\sigma_i(f_n)w_i^n$.
Note that if $f \in \A_K(X')$, then $\tr f \in \A_E(X)$ and 
$\A_E^0(X) = \{a + \tr f,a\in \OO_E, f\in \A_K^0(X')\}$. 
We can actually make this more precise: let $\A_E^0(X)_0$ and
$\A_K^0(X')_0$ the subspaces of $\A_E^0(X)$ and $\A_K^0(X')$ of functions
with no constant term (see Remark \ref{cstterm}). 
Then $\tr$ induces a bijection between $\A_K^0(X')_0$ and $\A_E^0(X)_0$.
Moreover if $f \in \A_K^0(X')_0$ then $\|f\|_{X'} = \|\tr f\|_X$, as can
be seen from Proposition \ref{ringfuncs}.

Note that $\A_K^0(X') \subset \OO_L[[w]]$, and this injection is
multiplicative and preserves the norm. So we also get an injection
$\A_E^0(X)_0 \to \OO_L[[w]]$, which preserves the norm as noted earlier.
But this is not multiplicative, as in general $\tr(fg) \neq (\tr f)(\tr g)$.
However, we have for all $f$, $g$ in $w\OO_L[[w]]$:
\begin{equation}
\label{normtrace}
\|\tr(fg)-(\tr f)(\tr g)\|_X < \|fg\|_{X'}.
\end{equation} 
In particular, we have that $\|\tr(fg)\|_X = \|(\tr f)(\tr g)\|_X$.

Let us prove inequality (\ref{normtrace}).
Write $f = \sum_nf_nw^n$, and $g = \sum_ng_nw^n$. Then 
$(\tr f)(\tr g)$ is of the form $\tr(u)$ for some $u$ in $w\OO_L[[w]]$.
We can assume that $\|f\|_{X'} = \|g\|_{X'} = 1$, so that $\|fg|_{X'} =
1$. 
Let $w'$ be one of the $w_i$ for $i>1$.
Consider $w^n{w'}^m$ for some integers $n$, $m > 0$. It can be written as
a sum of an element of $w\OO_L[[w]]$ and an element of $w'\OO_L[[w']]$, and
we want to understand the term in $w\OO_L[[w]]$. Note that for $n>0$,
$m>0$ we can write $w^n{w'}^m = \sum_{i=1}^n\alpha_iw^i +
\sum_{i=1}^m\beta_i{w'}^i$, with $\alpha_i \in \OO_L$ and $\beta_i \in
\OO_L$ and $|\alpha_i|<1$ and $|\beta_i|<1$ for all $i$ as $|\xi| <
|a-a'|$.  So we see that all the terms contributing to $fg-u$ have norm
$< 1$.  This proves inequality (\ref{normtrace}).

\bigskip
\noindent{\bf Step 3:} We give a description of $\A_E^0(X)$. Combining
the description of $\A_K^0(X')$ in Step 1 and using the trace we get:

\begin{equation}
\label{descrringhole1}
\A_E^0(X) = \left\{ a_0 +
\sum_{n\geq 0}\tr \left(\frac{Q_n}{\mu^t}v^n \right), \;
a_0\in\OO_E, Q_n\in\E'
\right\}
\end{equation}
and elements of $\A_E^0(X)$ can be written uniquely in such a way.
We want to change this description to something more convenient. 
Let $z = \tr v$. Then:

\begin{equation}
\label{descrringhole2}
\A_E^0(X) = \left\{ a_0 +
\sum_{n\geq 0}\tr \left(\frac{Q_n}{\mu^t} \right)z^n, \;
a_0\in\OO_E, Q_n\in\E'
\right\}.
\end{equation}

In order to prove this, we transform an element written as in Formula
(\ref{descrringhole1}) into an element written as in Formula
(\ref{descrringhole2}) by successive approximation, using the inequality
(\ref{normtrace}) and the formula (\ref{normmax}) for the norm.

\bigskip
\noindent{\bf Step 4:}
We now give a set of generators of $\A_E^0(X)$ as a complete
$\OO_E$-algebra, which will be useful in the next steps.

We start by giving a basis of the $\OO_K$-module $\E'$.  
For $0\leq j < t$,
let $\E_j$ be the subset of $K[x]_{<s}$ of polynomials that can be
written as $\sum_{i=0}^{s-1}b_i(x-a)^i$ with $b_i\in F$, 
$v_E(b_i) \geq (s(t-j)-i)\lambda$. Note that by Lemma \ref{NPclosed}, $\E_j$ is
the subset of elements of $K[x]_{<s}$ with $v_E(b_0) \geq s(t-j)\lambda$, and
if $P\in \E_j$ then for all $i>0$, $v_E(b_i) > (s(t-j)-i)\lambda$. Moreover,
$\E_j$ is in bijection with the set 
$$
F_{s(t-j)\lambda} = \{b\in F, v_E(b) \geq s(t-j)\lambda\}
$$ 
by $P \mapsto P(a)$. Indeed, if $b\in F$, it can be written
uniquely as $b = P(a)$ for some $P\in K[x]_{<s}$ as $F = K(a)$.
Note that by definition, for $0<j<t$, $F_{s(t-j)\lambda}$ does not contain
an element of valuation $s(t-j)\lambda$.
We note that $\E' = \oplus_{j=0}^{t-1}\mu^j\E_j$.
We define bases for the $\E_j$ as $\OO_K$-modules as follows: fix $\delta_j$ in
$F_{s(t-j)\lambda}$ of minimal valuation (take $\delta_0=1$, and note
that $v_E(\delta_j)>s(t-j)\lambda$ if $j\neq 0$). Let
$\varpi$ be a uniformizer of $F$, so that $(1,\varpi,\dots,\varpi^{s-1})$
is a basis of $\OO_F$ as an $\OO_K$-module. Then let $Q_{i,j}\in\E_j$ be
the polynomial such that $Q_{i,j}(a) = \delta_j\varpi^{i-1}$ for $1\leq i
\leq s$.
So we deduce a basis
$(P_{i,j})_{0\leq j < t,1\leq i \leq s}$ of $\E'$ as an $\OO_K$-module by
taking $P_{i,j} = Q_{i,j}\mu^j$.

Finally let $u_{i,j} = P_{i,j}/\mu^t \in \A_K^0(X')$, so that $v = R/\mu^t =
u_{1,0}$. Let $\alpha$ be a
uniformizer of $K$, so that $\OO_K = \OO_E[\alpha]$ (recall that $K$ is a
totally ramified extension of degree $N$ of $E$). 
Let $y_{i,j,\ell} = \tr(\alpha^\ell u_{i,j})$. Then $\A_E^0(X)$ is
generated by $z$ and the $y_{i,j,\ell}$, and more precisely:

\begin{equation}
\label{descrringhole3}
\A_E^0(X) =
\{ a_0+\sum_{n\geq
0}\left(\sum_{i=1}^s\sum_{j=0}^{t-1}\sum_{\ell=0}^{N-1}
a_{i,j,\ell,n}y_{i,j,\ell}\right)z^n, 
a_0\in\OO_E, a_{i,j,\ell,n}\in\OO_E \;\forall i,j,\ell\}.
\end{equation}

\bigskip
\noindent{\bf Step 5:}
Let now $A = \A_E^0(X)/(\pi_E)$. We now show that the hypotheses of
Corollary \ref{computeebis} are satisfied by $A$.
Denote by $\bar{y}_{i,j,\ell}$ the image of $y_{i,j,\ell}$ in $A$, and by
$\bar{z}$ the image of $z$ (observe that $y_{1,0,0} = z$).
Let $I$ be the ideal of $A$ generated by the $\bar{y}_{i,j,\ell}$ for
$(i,j,\ell) \neq (1,0,0)$. Then it is clear from Formula
(\ref{descrringhole3})
that the maximal ideal of $A$ is generated by $I$ and $\bar{z}$.

Then $I$ is a nilpotent ideal. Indeed, consider $f$ one of the elements
$y_{i,j,\ell}$, that is, $f = \tr \alpha^{\ell}u_{i,j}$.  We see $f$
as an element of $\A_L^0(X)$. When we write $\alpha^{\ell}u_{i,j}$ as
an element of $\OO_L[[w]]$, with $w = \xi/(x-a)$ as before, we see that in
fact it is in $\pi_L\OO_L[[w]]$, as either $\ell>0$ or $(i,j) \neq (1,0)$.
So $f$ is in $\pi_L\A_L^0(X)$. By Lemma
\ref{reducnorm}, this means that the image of $f$ in $A$ is nilpotent.
So $I$ is nilpotent.

Let us show that $A$ has no $\bar{z}$-torsion. 
Let $f \in A$ which is not a unit. Then $f = \bar{g}$ for some $g \in
\A_E^0(X)$ that can be written as $\tr(h)$ for some $h \in \A_K^0(X')$.
We compute: $\|zg\|_X = \|(tr v)(tr h)\|_X = \|tr(vh)\|_X$ (by the
Formula (\ref{normtrace})), so finally $\|zg\|_X = \|vh\|_{X'} = \|v\|_{X'}\|h\|_{X'}$.
Moreover, $\|v\|_{X'} = 1$, so $\|zg\|_X = \|h\|_{X'} = \|g\|_X$. By
Lemma \ref{reducnorm}, this means that $\bar{z}f \neq 0$ if $f \neq 0$.

\smallskip

So finally we are in the conditions of Corollary \ref{computeebis}.

\bigskip
\noindent{\bf Step 6} We compute the dimension of $A/(\bar{z})$, which is
the complexity we are looking for by Corollary \ref{computeebis}.
It is clear from (\ref{descrringhole3}) that $\dim_kA/(\bar{z}) \leq
Nst$, as $A/(\bar{z})$ is generated as a $k$-vector space by $1$ and the
$\bar{y}_{i,j,\ell}$ for $(i,j,\ell) \neq (1,0,0)$. Let us show that it
is in fact an equality.

Let $x = \mu + \sum_{(i,j,\ell)\neq
(1,0,0)}\lambda_{i,j,\ell}\bar{y}_{i,j,\ell}$ in $A$ that reduces to $0$
in $A/(\bar{z})$, and let us show that all the coefficients are in fact
$0$. First, $\mu = 0$ otherwise $x$ in a unit in $A$. Lift each
$\lambda_{i,j,\ell}$ to some $a_{i,j,\ell} \in \OO_E$.
Let $f = \sum_{(i,j,\ell)
\neq (1,0,0)}a_{i,j,\ell}\alpha^\ell u_{i,j}$, so that $x = \bar{\tr f}$,
and assume that $f \neq 0$.

The fact that $x$ reduces to $0$ in $A/(\bar{z})$ means that there exists some $g\in
\A_E^0(X)$ such that $\|\tr f - zg\|_X \leq |\pi_E|$. Then 
$g$ is in the maximal ideal of $\A_E^0(X)$ (it cannot be a unit as
$\bar{\tr f}$ is nilpotent in $A$ but $\bar{z}$ is not). So we can take
$g$ to be of the form $\tr h$ for some $h \in \A_K^0(X')_0$. Let us
compare $\tr f$ and $(\tr v)(\tr h)$: they are both in $\A_E^0(X)_0$ so 
we see them in $\OO_L[[w]] = \A_K^0(X')$.

We compute easily that the valuation of $\alpha^\ell u_{i,j}$ is $\ell
v_E(\alpha) + (i-1)v_E(\varpi) + v_E(\delta_j)$ and the leading term is
$w^{s(t-j)}$. So we can determine $j$ from the leading term. Note
also that $v_E(\alpha) = 1/N$, $v_E(\varpi) = 1/sN$. As $0 \leq \ell <N$
and $0\leq i-1<s$, we see that for a given $j$, the valuations of
$\alpha^\ell u_{i,j}$ and $\alpha^{\ell'} u_{i',j}$ are not equal
modulo $\Z$ except if $i=i'$ and $\ell=\ell'$. This means that in $f$
there are no cancellations, and in particular the leading term of $f$ is
$w^{s(t-j)}$ for some $j< t$. On the other hand, the leading term of $vh$
is $w^n$ for some $n > st$. This contradicts the fact that
$\|\tr f - (\tr v)(\tr h)\| \leq |\pi_E|$.

So finally $e(A) = \dim_kA/(\bar{z}) = Nst$.
\end{proof}

\subsubsection{Additivity formula}

We know want to compute the complexity of any connected standard subset
defined over $E$. Using the fact that the complexity is invariant under
unramified extension of the definition field, we see that Proposition
\ref{additivityc}, combined with Propositions \ref{multdisk} and
\ref{multhole}, gives us a way do this computation.

\begin{prop}
\label{additivityc}
Let $X$ be a connected standard subset defined over $E$. Assume that $X$
is of the form $Y \setminus T$, where $Y$ is either $\PP^1(\bar\Q_p)$ or a
disk defined over $E$, $T= \cup_{i=1}^mT_i$ where each
$T_i$ is a disjoint union of closed disks $D_{i,j}$ such that the $T_i$
are pairwise disjoint, with each defined and irreducible over $E$,
contained in $Y$, and each $D_{i,j}$ is well-defined over an extension
of $E$ that is totally ramified over $E$.
Then 
$c_E(X) = c_E(Y) + \sum_{i=1}^mc_E(\PP^1(\bar\Q_p)\setminus T_i)$ if $Y$ is a disk,
and 
$c_E(X) = \sum_{i=1}^mc_E(\PP^1(\bar\Q_p)\setminus T_i)$ if $Y =
\PP^1(\bar\Q_p)$.
\end{prop}

\begin{proof}
Let $X_i = \PP^1(\bar\Q_p)\setminus T_i$ for $1 \leq i \leq m$, and set
$X_0 = Y$ if $Y$ is a disk. Then $X = \cap_i X_i$, and each $X_i$ is
defined over $E$, and if $i \geq 1$ then $X_i$ is of the form of the
subsets studied in Proposition \ref{multhole}.

Using the description of the ring of functions in Proposition
\ref{ringfuncs}, we see that
for each $i$, we can write $\A_E^0(X_i) = \OO_E \oplus \M_i$ for some
submodule $\M_i$, where $\OO_E$ is the subring of constant functions (note
that if $i \geq 1$ we can choose $\M_i$ canonically by taking the
functions that are zero at infinity). Then we have a natural injection
$\A_E^0(X_i) \to \A_E^0(X)$ for all $i$, such that
$\A_E^0(X) = 
\OO_E \oplus (\oplus_i \M_i)$ by Proposition \ref{ringfuncs}.
Let $f \in \A_E^0(X)$, and write $f$ in this decomposition. Then $f$ has
a non-zero component on $\M_i$ if and only $f$ has a pole in $T_i$.
Let $A_i = \A_E^0(X_i)/(\pi_E)$, and $M_i = \M_i/(\pi_E)$. 
Then each $A_i$ contains an element $z_i$
as in Lemma \ref{additivitye}: it is the element called $\bar{z}$ in
Propositions \ref{multdisk} (for $i=0$, if $X$ is bounded) and
\ref{multhole} (for $i \geq 1$). 

Let $A = \A_E^0(X)/(\pi_E)$. Then $A = k \oplus (\oplus_i V_i)$ where $V_i$
is the image of $M_i$, and $A$ is
nearly the sum of the $A_i$'s as in Definition \ref{nearlysum}. 
In order to compute the multiplicity of $A$, we want to apply Lemma
\ref{additivitye}. So we need to prove: for all $i \neq j$, there exist
some integers $N$ and $t$, with $t<N$, such that $V_i^nV_j \subset
V_i^{n-t}$ for all $n > N$. It is clear that $V_iV_j \subset V_i+V_j$. So
we can assume without loss of generality that $m \leq 2$.

We will treat only the case where $i=1$, $j=2$. The case where $i$ or $j$
is equal to $0$ (which can occur only when $X$ is bounded) is similar.
For simplicity, we will assume from now on that $T_1$ and $T_2$ are
actually connected, that is, each is a single closed disk $D_i$ defined
over $E$. The general case needs no new ideas but requires more
complicated notation.

We first describe a little the ring $\A_E^0(X)$.
We fix a finite Galois extension $L$ of $E$ such that $X$ is well-defined
over $L$. Let $t_0=e_{L/E}$.  So for $i = 1,2$ we write $D_i =
D(a_i,|\xi_i|)^{+}$, with $a_i$ and $\xi_i$ in $L$.  Note that
$|\xi_i/(a_i-a_j)| < 1$ if $\{i,j\} = \{1,2\}$, so $v_L(\xi_i/(a_i-a_j)) \geq
1$.  Let $y_i = \xi_i/(x-a_i)$ for $i=1,2$.
Then $\A_E^0(X_i) \subset \OO_L[[y_i]] = \A_L^0(X_i)$.  If $h\in
\A_E^0(X_i) \cap \pi_L^{t_0}\OO_L[[y_i]]$, then $h$ is in
$\pi_E\A_E^0(X_i)$.

We have a decomposition $\A_E^0(X) = \OO_E \oplus \M_1 \oplus \M_2$ as
before. We denote by $\alpha_i$ the projection to $\M_i$ in this
decomposition. We also have a decomposition of $A = \A_E^0(X)/(\pi_E)$ as
$k \oplus V_1 \oplus V_2$, and we denote by $\bar\alpha_i$ the map to
$V_i$ which is the composition of reduction modulo $\pi_E$ and
projection to $V_i$. The maps $\alpha_i$ extend to the decomposition
$\A_L^0(X) = \OO_L \oplus \M_{1,L} \oplus \M_{2,L}$.

Denote by $z_1$ the element that was called $z$ in the proof of
Proposition \ref{multhole} applied to $X_1$ (which is also the element
called $v$, as we are in the case where $N=1$), and denote by
$\tau$ the integer that was denoted by $st$.  Then in $\OO_L[[y_1]]$,
$z_1$ is equal to $\pi_Lh+ y_1^\tau u$ for some $h \in \OO_L[y_1]_{<\tau}$
and $u\in\OO_L[[y_1]]^\times$.  For $m \geq 0$, write $z_1^m= \sum_{j\geq
0}c_{m,j}y_1^j$ with $c_{m,j}\in\OO_L$. Then we have that $v_L(c_{m,j})
\geq m-j/\tau$. On the other hand, we can write $y_1^{m\tau} = \sum_{i\geq
0}q_iz_1^i$ with $q_i \in \pi_L^{\max(0,m-i)}\OO_L[[t_1]]$.

Let $\bar{z}_1$ be the image of $z_1$
in $A_1$. Then as in the proof of Proposition \ref{multhole}, $V_1$ is
generated by $\bar{z}_1$ and a nilpotent ideal $I$ of $A_1$. Let $t_1$ be an
integer such that $I^{t_1}=0$. Then any element of $V_1^n$ for $n$ large
enough is a multiple of $z_1^{n-t_1}$. 

Fix some $f\in \M_1$ such that its image in $V_1$ is in $V_1^n$, and
$g\in \M_2$. As we are interested only in working in $A$, we can assume
that $f$ is divisible by $z_1^{n-t_1}$.
So when we write $f$ (seen as an element of $\A_L^0(X_1)$) 
as $\sum_j{f_j}y_1^j$, we have $v_L(f_j) \geq n-t_1-j/\tau$. 

We have that $fg \in \M_1\oplus\M_2$, so its image in
$A$ is in $V_1\oplus V_2$. We want to show that in this decomposition,
the projection $\bar\alpha_2(fg)$ of $fg$ to $V_2$ is zero, and the
projection $\bar\alpha_1(fg)$ to $V_1$ is contained in $V_1^{n-t}$ (for some
$t$ independent of $n$ to be determined).

We see easily that for all integers $a,b$, we can write
$y_1^ay_2^b = \sum_{i=1}^a\lambda_{a,b,i}{y_1}^i +
\sum_{i=1}^b\mu_{a,b,i}y_2^i$ with $\lambda_{a,b,i}$ and $\mu_{a,b,i}$ in
$\OO_L$, and $v_L(\lambda_{a,b,i}) \geq a+b-i$ and $v_L(\mu_{a,b,i})
\geq a+b-i$.

We study first $\alpha_1(fg)$ in $\A_L^0(X)$. We have $\alpha_1(fg) = \sum_{j\geq
0}f_j\alpha_1(y_1^jg)$. As
$v_L(f_j) \geq n-t_1-j/\tau$, all terms $f_j\alpha_1(y_1^jg)$ for $j \leq
(n-t_0-t_1)\tau$ contribute elements that are in $\pi_L^{t_0}\OO_L[[y_1]]$.
Consider now $\alpha_1(y_1^jg)$ for $j>(n-t_0-t_1)\tau$. It contributes
to $y_1^i$ with a coefficent of valuation $\geq j-i$. So all terms in
$y_1^i$ with $i \leq (n-t_0-t_1)\tau-t_0$ are in
$\pi_L^{t_0}\OO_L[[y_1]]$. So we see that $\alpha_1(fg)$ is in 
$(\pi_L^{t_0}\OO_L[[y_1]]+y_1^{(n-t_2)\tau}\OO_L[[y_1]])\cap \A_E^0(X_1)$
for $t_2=t_1+2t_0$. We have that $y_1^{(n-t_2)\tau} = \sum_i q_iz_1^i$
with $q_i \in \pi_L^{\max(0,(n-t_2-i)\tau)}\OO_L[[y_1]]$. So finally, 
$\alpha_1(fg) \in (\pi_L^{t_0}\OO_L[[y_1]]+z_1^{(n-t_3)}\OO_L[[y_1]])\cap
\A_E^0(X_1)$ for $t_3=t_2+t_0$. From this we deduce that
$\bar\alpha_1(fg)$ is a multiple of $\bar{z}_1^{n-t_3}$, and so is in
$V_1^{n-t_3}$.

We see also that if $n \geq 2t_0+t_1$, then $\alpha_2(fg)$ goes to $0$
in $V_2$.

So we get the result we wanted by taking $t=t_3$ and any
$N>\max(t_3,2t_0+t_1)$.
\end{proof}

\subsection{Combinatorial complexity of a standard subset with respect to
a field}
\label{combcomp}

We give another definition of complexity of a standard subset. It is
defined in more cases than the algebraic complexity, as we do not require
$X$ to be defined over $E$ to define the complexity of $X$ with respect
to $E$.

\subsubsection{Definition}
\label{defcombcomp}

Let $X$ be a standard subset, and $E$ be a finite extension
of $\Q_p$. We define an integer $\gamma_E(X)$ which we call combinatorial
complexity of $X$.

Let $D$ be a disk (open or closed). Let $F$ be the field of
definition of $D$ over $E$. Let $s$ be the smallest integer such that there exists
an extension $K$ of $F$, with $e_{K/F} = s$, and $K \cap D \neq
\emptyset$. Let $t$ be the smallest positive integer such that $D$ can be written as
$\{x, stv_E(x-a) \geq v_E(b)\}$ or as $\{x, stv_E(x-a)>v_E(b)\}$ for elements $a,b$
in $K$. Then we set $\gamma_E(D) = st$. 
We also set $\gamma_E(\PP^1(\bar\Q_p))=0$.

Then if $X$ is a standard subset, we define $\gamma_E(X)$ to be the sum
of the combinatorial complexities of its defining disks. That is:
If $X$ is a connected standard subset, it can be written uniquely as
$D_{0} \setminus \cup_{j=1}^{n}D_{j}$ with $D_0$ an open disk or
$D_0=\PP^1(\bar\Q_p)$, $D_j$ a closed disk for
$j>0$, and the $D_j$ are disjoint for $j>0$. We set $\gamma_E(X) =
\sum_{j=0}^n \gamma_E(D_j)$.
If $X$ be a standard subset, we can write uniquely $X =
\cup_{i=1}^sX_i$ where $X_i$ is a connected standard subset and the $X_i$
are disjoint.  Then we set $\gamma_E(X) =
\sum_{i=1}^s\gamma_E(X_{i})$. 

\subsubsection{Some properties of the combinatorial complexity}

\begin{lemm}
\label{changefieldanygamma}
Let $X$ be a standard subset.
Let $F/E$ be a finite extension. Then $\gamma_E(X) \geq
\gamma_F(X)$, with equality when $F/E$ is unramified, or when $F$ is
contained in the field of definition of $X$.
\end{lemm}

\begin{proof}
It suffices to show that $\gamma_E(D) \geq \gamma_F(D)$, with equality
when $F/E$ is unramified, for any disk $D$
(open or closed), and then it is clear from the definition.
\end{proof}

\begin{prop}
\label{gcomponent}
Let $X$ be a standard subset defined and irreducible over $E$, 
and write $X = \cup_{i=1}^s X_i$ its decomposition
in connected standard subsets. Let $E_1$ be the field of definition of
$X_1$ over $E$. Then $\gamma_E(X) = [E_1:E]\gamma_{E_1}(X_1)$.
\end{prop}

\begin{proof}
We have $\gamma_E(X) = \sum_{i=1}^s\gamma_E(X_i) = \sum_{i=1}^s\gamma_{E_i}(X_i)$.
Observe first that $\gamma_{E_i}(X_i)$ does not depend on $i$. Indeed,
for all $i$ there exists $\sigma \in G_E$ such that $\sigma(X_1) = X_i$
and $\sigma(E_1) = E_i$. Such a $\sigma$ transforms an equation
$\{x,v_E(x-a) \geq v_E(b)\}$ (or $\{x,v_E(x-a) > v_E(b)\}$) of a defining
disk of $X_1$ to an equation defining the corresponding disk
in $X_i$. 
Moreover, $s = [E_1:E]$, as $G_E$ acts transitively on the set of $X_i$
because we have assumed $X$ to be irreducible over $E$. 
\end{proof}

\subsection{Comparison of complexities}

The important result is that the two definitions of complexity
actually coincide when both are defined. 

\begin{theo}
\label{samecompl}
Let $X$ be a standard subset defined over $E$. 
Then $c_E(X) = \gamma_E(X)$.
\end{theo}

\begin{proof}
We can assume that $X$ is irreducible over $E$, as both multiplicities
are additive with respect to irreducible standard subsets.

Write now $X = \cup X_i$ where the $X_i$ are connected standard
subsets, and let $E_i$ be the field of definition of $X_i$.
Then $c_E(X) = [E:E_1]c_{E_1}(X_1)$ by Proposition \ref{ccomponent}, and
$\gamma_E(X) = [E:E_1]\gamma_{E_1}(X_1)$ by Proposition \ref{gcomponent}.

So we can assume that $X$ is a connected standard subset defined over
$E$.
Note that $c_E(X) = c_{E'}(X)$ and $\gamma_E(X)=\gamma_{E'}(X)$ for any
finite unramified extension $E'/E$ by Propositions \ref{changefieldany}
and \ref{changefieldanygamma}.
So we can enlarge $E$ if needed to an
unramified extension, and we can assume that we have written
$X = D \setminus \cup Y_i$ satisfying the hypotheses of Proposition
\ref{additivityc}. So we have $c_E(X) = c_E(D) + \sum_i c_E(\PP^1(\bar\Q_p)\setminus
Y_i)$ by Proposition \ref{additivityc}, and the analogous result for
$\gamma_E$ follows from the definition. So we need only prove the
equality for these standard subsets.

Let $D$ be a disk defined over $E$, of the form $\{x,v_E(x-a)>\lambda\}$.
Let $s$ be the minimal ramification degree of an extension $F$ of $E$
such that $F \cap D \neq \emptyset$, and $t>0$ be the smallest integer
such that $st\lambda\in (1/s)\Z$. Then $c_E(D) = \gamma_E(D) = st$.
For $c_E(D)$ it follows from Proposition \ref{multdisk}, 
and for $\gamma_E(D)$ it is the definition.
So we get that $c_E(D) = \gamma_E(D)$.

Let now $X = \PP^1(\bar\Q_p)\setminus T$, where $T$ is defined and irreducible over
$E$, and $T = \cup_{i=1}^ND_i$ where the $D_i$ are disjoint closed disks
defined over a totally ramified extension of $E$. We have $\gamma_E(X) =
\sum\gamma_E(D_i) = N\gamma_E(D_1)$ as the $D_i$ are $G_E$-conjugates.
Let $F$ be the field of definition of $D_1$. Then $\gamma_E(X) =
N\gamma_F(D_1) = N\gamma_F(\PP^1(\bar\Q_p)\setminus D_1)$. On the other hand, it
follows from Propostion \ref{multhole} that 
$c_E(X) = Nc_F(\PP^1(\bar\Q_p)\setminus D_1)$. 
Now the proof that $\gamma_F(\PP^1(\bar\Q_p)\setminus D_1) = c_F(\PP^1(\bar\Q_p)\setminus D_1)$
is the same as in the case of a disk. So finally $c_E(X) = \gamma_E(X)$.
\end{proof}

From now on we only write $c_E$ to denote either $c_E$ or $\gamma_E$ (so we can
consider $c_E(X)$ even for $X$ that is not defined over $E$, or for $X$ a
disjoint union of closed disks).

\begin{coro}
\label{multnbcomp}
The complexity of $X$ is at least equal to the number defining disks of
$X$. It is at least equal to the number of connected
components of $X$.
\end{coro}

\subsection{Finding a standard subset from a finite set of points}

\subsubsection{Approximations of a standard subset}
\label{approx}

Let $X = \cup_{n=1}^N(D_{n,0}  \setminus \cup_{i=1}^{m_n}D_{n,i})$ be a
bounded standard subset, where the $D_{n,0}  \setminus \cup_{i=1}^{m_n}D_{n,i}$
form the decomposition of $X$ as a disjoint union of connected standard
subsets, so that the disks $D_{i,j}$ are the defining disks of $X$.

For $J \subset \{1,\dots,N\}$ and $I_n \subset \{1,\dots,m_n\}$
for $n\in J$, we set $Y_{J,I} = \cup_{n\in J}(D_{n,0} \setminus
\cup_{i\in I_n}D_{n,i})$. This is a standard subset with $c_E(Y_{J,I})
\leq c_E(X)$ and equality if and only if $X = Y_{J,I}$. 
Such standard subsets are called approximations of $X$.

For a bounded connected standard subset $X$, written as
$D(a,r)^-\setminus \Delta$ for some finite union of closed disks
$\Delta$, we define its outer part as $D(a,r)^-$.  If $X$ is any bounded
standard subset, we define its outer part as the union of the outer parts
of its connected components. Note that if $X$ is defined over a field
$E$, then so is its outer part $X'$, and $X'$ is an approximation of $X$,
and it contains $X$.

We make similar definitions for unbounded standard subsets.  If $X$ is an
unbounded standard subset, then we define its outer part to be
$\PP^1(\bar\Q_p)$.

\subsubsection{Main results}

\begin{theo}
\label{algo}
Let $X$ be a standard subset defined over $E$. Let $m$
be an integer such that $c_E(X) \leq m$.
Then there exists a finite set $\E$ of finite extensions of $E$,
depending only on $E$ and $m$, such
that $X$ is entirely determined by the sets $X \cap F$ for all extensions
$F\in \E$.
\end{theo}

We can actually take the set $\E$ to be the set of all extensions of $E$
of degree at most $N$ for some $N$ depending only on $E$ and $m$.
This Theorem will be proved in Paragraph \ref{algoproof}, after we establish some
preliminary results.

\begin{coro}
\label{algoeps}
Let $X$ be a standard subset of $D(0,1)^-$ defined over $E$. Let $m$
be an integer such that $c_E(X) \leq m$.
Let $\eps>0$ be such that for all $x\in X$, $D(x,\eps)^- \subset X$, and for all
$x\not\in X$, $D(x,\eps)^- \cap X = \emptyset$. Then there exists a
finite subset $\mathcal{P}$ of $D(0,1)^-$, depending only on $E$, $m$, and
$\eps$, such that $X$ is entirely determined
by $X \cap \mathcal{P}$.
\end{coro}

\begin{proof}
Let $\E$ be the set of extensions of $E$ given by Theorem \ref{algo}. For
each extension $F$ of $E$ which is in $\E$, the set $F \cap D(0,1)^-$ can
be covered by a finite number of open disks of radius $\eps$, and we
define a finite set $\mathcal{P}_F$ by taking an element in each of these
disks.  Then $X \cap F$ can be entirely determined from $X \cap
\mathcal{P}_F$.
We set $\mathcal{P}$ to be the union of the sets $\mathcal{P}_F$ for the
extensions $F$ of $E$ that are in $\E$. This is a finite set, as $\E$ is
finite, and $X$ is determined by $X \cap \mathcal{P}$, as it is
determined by the intersections $X \cap F$ for $F \in \E$ by Theorem
\ref{algo}.
\end{proof}

\begin{rema}
As is clear from the proof, the set $\mathcal{P}$ can be huge. 
However in practice for a given $X$ we need only test points in a very small
proportion of this subset.
\end{rema}

\subsubsection{Notation}

Let $c \in \bar\Q_p$.
If $a<b$ are rational numbers, denote by $A_c(a,b)$ the annulus $\{x,
b<v_E(x-c)<a\}$. If $a$ is a rational number, denote by $C_c(a)$ the circle $\{x,
v_E(x-c) = a\}$. 

If $t \in \Q$, let $\den(t)$ be the
denominator of $t$, that is, the smallest integer $d$ such that $t \in
(1/d)\Z$. 
Note that $[E(x):E] \geq \den(v_E(x))$.

\subsubsection{Preliminaries}

\begin{lemm}
\label{mincompl}
Let $x,z \in \bar\Q_p$, with $\den(v_E(x-z)) > te_{E(x)/E}$ for some integer
$t$. Let $D$ be a rational disk (open or closed) containing $z$
but not $x$. Then $c_E(D) > t$.
\end{lemm}

\begin{proof}
It is enough to prove that $c_{E(x)}(D) \geq t$, as $c_E(D) \geq
c_{E(x)}(D)$. Let $D$ be such a disk. As $D$ does not contain $x$, we
have $D \subset C_x(\lambda)$ where $\lambda = v_E(x-z)$. So for all $y
\in D$, $\den(v_{E(x)}(y-z)) = \den(v_{E(x)}(x-z)) > t$, which implies
that $e_{E(x,y)/E(x)} > t$. By the definition of combinatorial complexity, we deduce
that $c_{E(x)}(D) > t$.
\end{proof} 

Fix an integer $B$.
We say that $\lambda \in \Q$ has a large denominator (with respect to
$B$) if $\den(\lambda)
> B$, and a small denominator otherwise. The set of elements of $\Q$ with
small denominator can be enumerated as a strictly increasing sequence
$(t_i)_{i\in \Z}$.

\begin{coro}
\label{annulusin}
Let $x \in \bar\Q_p$, $m\in \Z$, and $X$ be a standard subset defined over $E$ with
$c_E(X) \leq m$. Let $B \geq me_{E(x)/E}$, and define the sequence
$(t_i)$ of rationals that have a small denominator with respect to $B$.
Let $i \in \Z$,
and let $D$ be a defining disk of $X$ (open or closed). Then either
$A_x(t_i,t_{i+1}) \cap D = \emptyset$, or $A_x(t_i,t_{i+1}) \subset D$
(and then $x \in D$).
\end{coro}

\begin{proof}
Assume that $A_x(t_i,t_{i+1}) \cap D$ is not empty, and let $z
\in\bar\Q_p$ be an element of this set. Then $\den(v_E(x-z)) >
B$, so in particular $\den(v_E(x-z)) > me_{E(x)/E}$. By Lemma
\ref{mincompl}, either $c_E(D) > m$ or $x \in D$. As the first is
impossible because $c_E(X) \leq m$, we get that $x \in D$. 
Assume that $D$ is a closed disk (the case of an open disk being
similar).
So $D$ is a set of the form $\{y,v_E(x-y) \geq t\}$ for some $t\in \Q$,
or equivalently of the form $\{y,v_{E(x)}(x-y) \geq t'\}$ for $t' =
te_{E(x)/E}$.
We have $c_{E(x)}(D) = \den(t')$ and $c_{E(x)}(D) \leq c_E(D) \leq
c_E(X) \leq m$, so $\den(t') \leq m$, and so $t$ has a small
denominator and hence is one of the $t_j$. As $D$ contains an element of
$A_x(t_i,t_{i+1})$, we see that $t_j \leq t_i$ and so
$A_x(t_i,t_{i+1})\subset D$.
\end{proof}

\begin{prop}
\label{pointexistsin}
Let $E$ be a finite extension of $\Q_p$. There exists a function $\psi_E$
such that for any bounded standard subset $X$ defined over $E$, 
if $c_E(X) \leq m$, then there exists an extension $F$ of $E$ with $[F:E]
\leq \psi_E(m)$ and $X \cap F \neq \emptyset$.
\end{prop}

\begin{lemm}
\label{pointindisk}
Let $E$ be a finite extension of $\Q_p$. There exists a function $\psi^0_E$
such that for any open or closed disk $D$ of $\bar\Q_p$ defined over $E$, 
if $c_E(D) \leq m$ then there exists an extension $F$ of $E$ with $[F:E]
\leq \psi^0_E(m)$ and $D \cap F \neq \emptyset$ and the radius of $D$ is
in $|F^\times|$. 
For $m<p^2$ or $p=2$ we can take $\psi^0_E(m) = m$ and consider only
extensions $F/E$ that
are totally ramified.
\end{lemm}

\begin{proof}
We write the proof for $D$ open, the proof for $D$ closed being nearly
identical.

Let $s$ be the minimal ramification degree of an extension $K$ of $E$
with $K \cap D \neq \emptyset$, and fix $a \in K \cap D$. Let $t$ be
the smallest positive integer such that $D$ can be written as
$\{x,stv_E(x-a) > v_E(b)\}$ for an element $b\in K$. So by definition
$c_E(D) = st$. 
By Theorem \ref{illdeftheo}, there exists an extension $K$ of $E$
with $e_{K/E} = s$ and $[K:E] \leq s^2$ and $K \cap D \neq \emptyset$.
Then if $F$ is a totally ramified extension of degree $t$ of $K$, then
$F$ satisfies the conditions, and we have $[F:E] \leq s^2t$. As $st \leq
m$, this means that we can take $\psi^0_E(m) = m^2$.

Note that $s$ is a power of $p$ by Theorem \ref{illdeftheo}, and $s\leq
m$. So if $m<p^2$ then $s=1$ or $s=p$ so we can take $[K:E] \leq s$ and
$K/E$ totally ramified instead of $[K:E] \leq s^2$, and so we can take
$[F:E] \leq m$. 

When $p=2$ the result comes from applying Theorem \ref{illdeftheo2}
instead of Theorem \ref{illdeftheo}.
\end{proof}

\begin{proof}[Proof of Proposition \ref{pointexistsin}]
We show first that there exists a function $\psi_E^1$ such that for all
$X$ a standard connected subset defined over $E$ with $c_E(X) \leq m$,
there exists an extension $L$ of $E$ with $[L:E]
\leq \psi^1_E(m)$ and $X \cap L \neq \emptyset$.

\smallskip

Consider first the case where $X$ is of the form $D(0,1)^-\setminus Y$,
with $Y$ a disjoint union of closed disks.
Then: either $0\not\in Y$, in which case $0 \in X$, so $E\cap
X \neq\emptyset$ and there is nothing more to do, or $0\in Y$. We assume
from now on that $0 \in Y$. Then $m > 1$ and we can write $Y$ as
$D(0,|a|)^+ \cup Z$, with $Z$ a union of disjoint closed disks.
By the additivity formula for complexity, we have that
$c_E(\PP^1(\bar\Q_p)\setminus D(0,|a|)^+) +
c_E(\PP^1(\bar\Q_p)\setminus Z) \leq m-1$. 
Let $\lambda \in \Q^\times$ with $\den(\lambda) \geq m$. Then $Z \cap
C_0(\lambda) = \emptyset$, by Lemma \ref{mincompl} and the
fact that $c_E(\PP^1(\bar\Q_p)\setminus Z) < m$.
Let $s = \den(v_E(|a|))$. Then by definition of the combinatorial
complexity, we have that $c_E(\PP^1(\bar\Q_p)\setminus D(0,|a|)^+) = s$,
so we know that $s < m$, and in particular $v_E(a) > 1/m$.  So we see
that $C_0(1/2m) \subset X$ by the two previous remarks.  Let $L$ be a
totally ramified extension of $E$ of degree $2m$, then $L \cap C_0(1/2m)
\neq \emptyset$, and so $L \cap X \neq \emptyset$. So there exists an
extension of $E$ of degree at most $2m$ such that $X$ has points in this
extension.

\smallskip

Consider now the case where $X$ is of the form $D \setminus Y$, but $D$
is not necessarily $D(0,1)^-$ anymore. By Lemma \ref{pointindisk}, there
exists an extension $F$ of $E$ of degree at most $\psi_E^0(m)$ such that
$D$ contains a point in $F$ and has a radius in $|F^\times|$. Moreover,
$c_F(X) \leq c_E(X) \leq m$. By doing some affine transformation defined
over $F$, we can reduce to the case where $D = D(0,1)^-$, so we see that
$X$ contains a point in some extension $L$ of $F$ with $[L:F] \leq 2m$,
and so $X$ contains a point in $L$ with $[L:E] \leq \psi_E^1(m)$, where
$\psi_E^1(m) = 2m\psi_E^0(m)$.

\smallskip

Now we go back to the general case, where $X$ is not necesssarily
connected.  Write $X$ as a disjoint union of irreducible components
over $E$. Each of them has complexity at most $m$, and it is enough to
find a point in one of them. So we can assume that $X$ is irreducible
over $E$.

Suppose now that $X$ is irreducible over $E$: write $X = \cup_{i=1}^sX_i$ where
the $X_i$ form a $G_E$-orbit. Let $F$ be the field of definition of
$X_1$, and $s = [F:E]$. Then $c_E(X) = sc_F(X_1)$, so $c_F(X_1) \leq m' =
\lfloor m/s \rfloor$. There exists an extension $L$ of $F$ of degree at
most $\psi_F^1(m')$ such that $K \cap X_1 \neq \emptyset$. As $L$ is an
extension of $E$ of degree at most $s\psi_F^1(m')$, we see that we can take
$\psi_E(m) = \sup_{1 \leq s \leq m}\sup_{[F:E]=s}s\psi_F^1(\lfloor m/s
\rfloor)$, which is finite as $E$ has only a finite number of extensions
of a given degree.
\end{proof}

By inverting the role of closed disks and open disks, we obtain the
following statement:

\begin{prop}
\label{pointexistsout}
Let $E$ be a finite extension of $\Q_p$. There exists a function $\phi_E$
such that for any standard subset $X$ defined over $E$ and different from
$\PP^1(\bar\Q_p)$, if $c_E(X) \leq m$, then there exists an extension $F$
of $E$ with $[F:E] \leq \phi_E(m)$ and there exists an element of $F$
that does not belong to $X$.
\end{prop}

Following the proofs above, we see that we can actually take $\psi_E^0(m)
= m^2$, $\psi_E^1(m) = 2m^3$, and $\psi_E(m) = \phi_E(m) = 2m^3$.

\subsubsection{Proof of Theorem \ref{algo}}
\label{algoproof}

To help with the understanding of the method, we will explain the steps
on the following example (for $p>2$): let $X = (D_{1,0}\setminus D_{1,1}) \cup
(D_{2,0} \setminus D_{2,1}) \cup D_{3,0} \cup D_{4,0}$ where
$D_{1,0} = D(0,1)^-$, $D_{1,1} = D(0,|p|)^+$, 
$D_{2,0} = D(0,|p|)^-$, $D_{2,1} = D(0,|p^2|)^+$,
$D_{3,0} = D(1/\sqrt{p},1)^-$,  
$D_{3,0} = D(-1/\sqrt{p},1)^-$. Here $X$ is defined over $\Q_p$ and
$c_{\Q_p}(X) = 6$ so we can take any $m \geq 6$.

We assume first that we know that $X$ is bounded.

Write $X$ as $\cup_{n=1}^N(D_{n,0}  \setminus \cup_{i=1}^{m_n}D_{n,i})$
where the $D_{n,0} \setminus \cup_{i=1}^{m_n}D_{n,i}$ form the
decomposition of $X$ as a disjoint union of connected standard subset (in
particular $D_{n,0}$ is open and $D_{n,i}$ is closed for $i>0$).  We
number the disks so that the $D_{n,0}$ for $0 \leq n \leq M$ are maximal,
that is, they are not included in any other defining disk of $X$, and the
$D_{n,0}$ for $M < n \leq N$ are all included in another defining disk of
$X$. In this way, the outer part of $X$ is $\cup_{n=1}^MD_{n,0}$.
In the example the outer part of $X$ is $D(0,1)^- \cup D(1/\sqrt{p},1)^-
\cup D(-1/\sqrt{p},1)^-$.

Let $\E$ be the set of extensions of $E$ of degree at most
$2m^2\max(\psi_E(m),\phi_E(m))^3$, where $c_E(X) \leq m$. Let
$\mathcal{P} = \cup_{F\in \E}F$. We have to show that $X$ can be
recovered from the knowledge of $X \cap \mathcal{P}$.

We work by constructing a sequence $(X_i)$ of approximations of $X$, such
that each $X_i$ is defined over $E$ and is an approximation of $X_{i+1}$
and $c_E(X_{i+1}) > c_E(X_i)$, so that at some point $X_i = X$ and we
stop.

\medskip

We first describe how to solve the following problem: given some fixed
$x\in X$, with $[E(x):E] \leq \psi_E(m)$, find the largest defining disk
$D$ of $X$ containing $x$ (note that $D$ is necessarily open).

Let the sequence $(t_i)$ be as before Corollary \ref{annulusin}, with $B =
m\psi_E(m)$.
For each $i \in \Z$, let $\lambda_i = (t_i + t_{i+1})/2$. By construction
$\lambda_i$ has a large denominator, but $\den(\lambda_i) \leq
2m^2\psi_E(m)^2$. Choose some
$z_i \in \mathcal{P}$ such that $v_E(x-z_i) = \lambda_i$. This is possible as
we can choose $z_i$ in a totally ramified extension of $E(x)$ of degree
at most $2m^2\psi_E(m)^2$. Then: 

\begin{lemm}
\label{largestdefdisk}
Let $i \in \Z$ be the smallest element such that $z_i \in X$.
The largest defining disk of $X$ containing $x$ is the disk $D = \{z,
v_E(x-z) > t_i\}$.
\end{lemm}

\begin{proof}
Let $D$ be the largest defining disk of $X$ containing $x$.  
Then $z_i \in D$. Otherwise, $z_i$ is contained in some (open) defining disk of
$X$ that does not contain $x$, which contradicts Corollary \ref{annulusin}.

We can write $D$ as $\{z, v_E(x-z) > t\}$ for some $t \in \Q$. Moreover
$t = t_j$ for some $j$, as $c_E(D) \leq m$. Then $t \leq t_i$ as $D$ contains $z_i$ so $j
\leq i$.

Let us show now that $j=i$. If $j<i$ then by definition of $i$, $z_j
\not\in X$. As $z_j\in D$, it means that $z_j \in D'$ for some closed
defining set of $X$ contained in $D$. 
But then $A_x(t_j,t_{j+1}) \subset D'$ by Corollary
\ref{annulusin}, so $D'$ is of the form $\{y, v_E(y-x) \geq s\}$ for some
$s \leq t_j$, which contradicts the fact that $D' \subset D$.
\end{proof}

\medskip

Next we describe how to solve the following problem: given some fixed
$x\not\in X$ but $x$ in the outer part of $X$, with $[E(x):E] \leq
\phi_E(m)$, find the largest defining disk $D$ of $X$ containing $x$
(note that $D$ is necessarily closed).

Let the sequence $(t_i)$ be as before Corollary \ref{annulusin}, with $B =
m\phi_E(m)$. Then $D$ is of the form $\{z,v_E(x-z)\geq t_i\}$ for some
$i$, as $c_{E(x)}(D) \leq m$. As $x$ is in the outer part of $X$, there
exists a largest defining disk $D'$ of $X$ containing $x$. Let $i_0\in
\Z$ be the integer such that $D' = \{z,v_E(x-z) > t_{i_0}\}$.
For each $i$, we can find an element $z_i$ in $C_x(t_i)$, with
$[E(x,z_i):E(x)] \leq m^2\phi_E(m)$, such that the $G_{E(x)}$-orbit of
$z_i$ contains at least $m$ elements $z_i^{(1)} = z_i,\dots,z_i^{(m)}$
satisfying $v_E(z_i^{(j)}-z_i^{(\ell)}) = t_i$ for all $j \neq \ell$.
We can find such a $z_i$ as follows: first find $y_i$ in a totally
ramified extension of $E(x)$ of degree at most $m\phi_E(m)$ such that
$v_E(x-y_i) = t_i$. Next, find $u_i$ generating the unramified extension of
$E(x)$ of degree $m$ and such that $|u_i| = 1$ and the $G_{E(x)}$-conjugates
of $u_i$ have distinct reductions modulo $p$.  Let $z_i = u_iy_i$. Then
$z_i$ satisfies the property we want, as it has $m$
$G_{E(x,y_i)}$-conjugates $(z_i^{(\ell)})_{1\leq \ell\leq m}$ 
and they satisfy the property 
about $v_E(z_i^{(j)}-z_i^{(\ell)})$.

\begin{lemm}
\label{largestdefdisk2}
Let $i \in \Z$ be the smallest element $> i_0$ such that $z_i \not\in
X$. Then $D = \{z, v_E(x-z) \geq t_i\}$.  
\end{lemm}

\begin{proof}
Let $j\in \Z$, $j > i_0$ be such that $z_j \not\in X$. As $X$ is
$G_E$-stable, this means that $z_j^{(\ell)}$ is not in $X$ for all $1
\leq \ell \leq m$.
As each $z_j^{(\ell)}$ is in the outer part of $X$ (in fact in $D'$), it
means that each $z_j^{(\ell)}$ is contained in some closed defining
$D_\ell$ disk of $X$ contained in $D'$. Then in fact there exists a
closed defining disk of $X$ containing $x$ and all the $z_j^{(\ell)}$ for
$1 \leq \ell \leq m$. Indeed, if a disk contains two of the
$z_j^{(\ell)}$ it contains all of them and also $x$, due to the condition
on the $v_E(z_j^{(\ell)}-z_j^{(\ell')})$.  So if there is not a closed
defining disk containing all the $z_j^{(\ell)}$, then the disks $D_\ell$
are all distinct, which gives that $c_E(X) > m$. So in particular: $z_j$
is contained in $D$.

On the other hand, assume that $D = \{z,v_E(x-z) \geq t_j\}$ for some
$j$. Necessarily $j > i_0$ as $D \subset D'$. 
Then $z_j \not\in X$: if $z_j$ is in $X$, then so is $z_j^{(\ell)}$
for all $1 \leq \ell \leq m$. So each $z_j^{(\ell)}$ is contained in an
open defining disk of $X$ contained in $D$, and so as before there exists
some open defining disk of $X$ containing $z_j$ and $x$ and contained in
$D$. But this is impossible as $v_E(x-z_j) = t_j$.
\end{proof}

\medskip

We show now how to find the outer part $X_1$ of $X$, that is, $X_1 =
\cup_{n=0}^MD_{n,0}$ in the notation of the beginning of the proof. 
Start with $X_1 = \emptyset$.

\begin{enumerate}
\item
find some $x \in X \cap \mathcal{P}$ that is not in $X_1$, if there is
one. If there is not, then $X_1$ is the outer part of $X$.

\item
find the largest defining (open) disk $D$ of $X$ that contains $x$, using
Lemma \ref{largestdefdisk}. Add
to $X_1$ the $G_E$-orbit of $D$. Go back to the first step.

\end{enumerate}

For the first step, Proposition \ref{pointexistsin} ensures that if $X$
is not contained in $X_1$, we can find some element of $X \setminus X_1$
that is also in $\mathcal{P}$.  For the second step, note that by
construction the $G_E$-orbit of $D$ is disjoint from $X_1$, and during
the construction the set $X_1$ is always defined over $G_E$.

In the example: note that $X \cap \Q_p = \emptyset$, so we find points in
$X \cap \Q_p(\sqrt{p})$. The fact that $1/\sqrt{p} \in X$ gives us the
defining disk $D_{3,0} = D(1/\sqrt{p},1)^-$, and the fact that $X$
is defined over $\Q_p$ gives the other defining disk $D_{4,0} =
D(-1/\sqrt{p},1)^-$. The fact that $\sqrt{p} \in X$ gives us the defining
disk $D_{1,0} = D(0,1)^-$. At this point we have the outer part $X_1 =
D_{1,0} \cup D_{3,0} \cup D_{4,0}$.

\medskip

We now want to find $X_2 = \cup_{n=0}^M(D_{n,0}\setminus
\cup_{i=0}^{m_n}D_{n,i}$). Note that $X_2$ is defined over $E$. 

The method as it is very similar
to the method to find $X_1$. For each $D_{n,0}$, $n \leq M$, find if
there is an element $x$ that is in $X_1$ but not in $X$. If such an
element exists, we can take it with $[E(x):E] \leq \phi_E(m)$ by
Proposition \ref{pointexistsout}. Then we find the largest (closed)
defining disk of $X$ containing $x$ and contained in $D_{n,0}$ using 
Lemma \ref{largestdefdisk2}.

In the example $X_2 = (D_{1,0}\setminus D_{1,1}) \cup D_{3,0} \cup
D_{4,0}$. The fact that $p$ is in $D_{1,0}$ but not in $X$ gives us the
defining disk $D_{1,1}$.

\medskip

Once $X_2$ is found, we have the decomposition $X = X_2 \cup X'$ with $X' =
\cup_{n=M+1}^N(D_{n,0}\setminus \cup_{i=1}^{m_n}D_{n,i})$, and $X_2$ and
$X'$ are both approximations of $X$. 
In particular, we have that $c_E(X') =
c_E(X) - c_E(X_2)$. Let $m' = m-c_E(X_2)$. If $m' = 0$ then $X_2 = X$.
Otherwise, we must find $X'$, given that $X'$ is defined over $E$ and
$c_E(X') \leq m'$.  Moreover, as we know $X_2$ entirely, if we know $X
\cap \mathcal{P}$ then we also know $X' \cap \mathcal{P}$.
Applying the same steps as before, we can find an approximation of $X'$
and so work recursively. 

In the example, $X' = D_{2,0}\setminus D_{2,1}$, and as $c_{\Q_p}(X_2) =
4$ we get $m' = m-4$, and we need to find $X'$. 

\medskip

Finally, we want to remove the hypothesis that we know that $X$ is
bounded. First, we can determine whether $X$ is bounded by
considering only $X \cap \Q_p$. If $X$ is bounded apply the algorithm
described. If $X$ is not bounded, then set $X_1 = X$, and then apply the
algorithm starting at the step where we determine $X_2$.

\section{Application to potentially semi-stable deformation rings}
\label{potcrys}

\subsection{Definition of the potentially semi-stable deformation rings}
\label{kisinrings}

We recall the definition and some properties of the rings defined by
Kisin in \cite{Kis08} (see also \cite{Kis10}).

Let $\rho : G_{\Q_p} \to \GL_2(\bar\Q_p)$ be a potentially semi-stable
representation. Then we know from \cite{Fo} that we can attach to $\rho$
a Weil-Deligne representation $\WD(\rho)$, that is, a smooth
representation $\sigma$ of the Weil group $W_{\Q_p}$ with values in  $\GL_2(\bar\Q_p)$, and a
$\bar\Q_p$-linear, nilpotent
endomorphism $N$ of $\bar\Q_p^2$ such that $N\sigma(x) = p^{\deg
x}\sigma(x)N$ for all $x \in W_{\Q_p}$.  We say that $\sigma$ is the
extended type of $\rho$, and $\sigma_{|I_{\Q_p}}$ the inertial type of
$\rho$, where $I_{\Q_p}$ is the inertia subgroup of $W_{\Q_p}$.

Kisin defines deformation rings that parametrize potentially semi-stable
representations with fixed (distinct) Hodge-Tate weights and a fixed
inertial type. However, this is not entirely adapted to our purposes: we
would like each of these families of representations to be classified by
one parameter (see Theorem \ref{parameter}). This is not the case for the
rings defined by Kisin: for example, if we take the trivial inertial
type, the deformation ring classifies a family of crystalline
representations, and a family of semi-stable, non-crystalline
representations, and we cannot classify all of these with a single
parameter. So we introduce a refinement of Kisin's rings, where in some
cases we will consider deformations with a fixed extended type instead,
and use a refinement of Kisin's rings defined in \cite{Roz15}.

\subsubsection{Definition of the Galois types}

We make the following definition:
\begin{defi}
\label{deftype}
A Galois type of dimension $2$ is one of the following representations
with values in $\GL_2(\bar\Q_p)$:
\begin{enumerate}
\item
a scalar smooth representation $\tau =  \chi \oplus \chi$ of $I_{\Q_p}$,
such that $\chi$ extends to a character of $W_{\Q_p}$.

\item
a smooth representation $\tau = \chi_1 \oplus \chi_2$ of $I_{\Q_p}$,
where both $\chi_1$ and $\chi_2$ extend to characters of $W_{\Q_p}$, and
$\chi_1 \neq \chi_2$.

\item
if $p>2$,
a smooth representation $\tau = \chi_1 \oplus \chi_2$ of $W_{\Q_p}$, such
that $\chi_1$ and $\chi_2$ have the same restriction to inertia, and
$\chi_1(F) = p\chi_2(F)$ for any Frobenius element $F$ in $W_{\Q_p}$.

\item
if $p>2$,
a smooth irreducible representation $\tau$ of $W_{\Q_p}$.

\end{enumerate}
\end{defi}

We call Galois types of the form (1) and (2) inertial types, and those
of the forms (3) and (4) discrete series extended types.
If $\rho$ is a potentially semi-stable
representation of $G_{\Q_p}$ of dimension $2$ and $p>2$, then we know
from the classification of $2$-dimension smooth representations of
$W_{\Q_p}$ that either its
inertial type is isomorphic to a Galois type of the form (1) or (2), or
its extended type is isomorphic to a Galois type of the form (3) or
(4) (if $p=2$ there are other possibilities).
Of course the possibilities are not mutually exclusive, as a
representation that has its extended type of the form (3) also has its
inertial type of the form (1), but we will define different deformation rings
using these Galois types.
Note that if the Galois type of $\rho$ is of the form (2) or (4)
then it is potentially crystalline (that is, the endomorphism $N$ of the
Weil-Deligne representation is zero), and that if $\rho$ is potentially
semi-stable but not potentially crystalline (that is, $N \neq 0$) then
its Galois type is of the form (3).

\subsubsection{Definition of the deformation rings}

\begin{defi}
\label{defdata}
A deformation data $(k,\tau,\bar\rho,\psi)$ is the data of:
\begin{enumerate}
\item
an integer $k \geq 2$.

\item
a Galois type $\tau$.

\item
a continuous representation $\bar\rho$ of $G_{\Q_p}$ of dimension $2$, with trivial
endomorphisms, over some finite extension $\F$ of $\F_p$.

\item
a continuous character $\psi: G_{\Q_p} \to \bar\Q_p^\times$ lifting
$\det\bar\rho$ such that $\psi$ and $\cyc^{k-1}\det\tau$ coincide.

\end{enumerate}

If the type $\tau$ is a discrete series extended type, we will assume
that $p>2$.
\end{defi}

Let $(k,\tau,\bar\rho,\psi)$ be a deformation data, and let $E$ be a
finite extension of $\Q_p$ over which $\tau$ and $\psi$ are defined, and
such that its residue field contains $\F$.  Let $R(\bar\rho)$ be the
universal deformation ring of $\bar\rho$ over $\OO_E$, it is a local
noetherian complete $\OO_E$-algebra. Let $R^\psi(\bar\rho)$ the quotient
of $R(\bar\rho)$ that parametrizes deformations of determinant $\psi$.
Then Kisin in \cite{Kis08} defines deformation rings
$R^\psi(k,\tau,\bar\rho)$ that are quotients of $R^\psi(\bar\rho)$. We
will also use a refinement of these rings introduced in \cite{Roz15},
which are better for our purposes in view of Theorem \ref{parameter}.  

If the Galois type $\tau$ is an inertial type, we denote by
$R^\psi(k,\tau,\bar\rho)$ the ring classifying potentially crystalline
representations with Hodge-Tate weights $(0,k-1)$, inertial type $\tau$,
determinant $\psi$ with reduction isomorphic to $\bar\rho$, as defined by
Kisin in \cite{Kis08}.  

If the Galois type $\tau$ is a discrete series extended type, we denote
by $R^\psi(k,\tau,\bar\rho)$ the $\OO_E$-algebra which is a quotient of
$R^\psi(\bar\rho)$, classifying potentially semi-stable representations
with Hodge-Tate weights $(0,k-1)$, extended type $\tau$, determinant
$\psi$ with reduction isomorphic to $\bar\rho$ defined in
\cite[2.3.3]{Roz15}. 

We know that $R^\psi(k,\tau,\bar\rho)$ is a complete flat local
$\OO_E$-algebra (in particular it has no $p$-torsion), such that $\spec
R^\psi(k,\tau,\bar\rho)[1/p]$ is formally smooth of dimension $1$.

The characterizing property of these potentially semi-stable deformation
rings is the following: There is a bijection between the maximal ideals
of $R^\psi(k,\tau,\bar\rho)[1/p]$ and the set of isomorphism classes of
lifts $\rho$ of $\bar\rho$ of determinant $\psi$, potentially crystalline
of inertial type $\tau$ (resp.  potentially semi-stable of extended type
$\tau$), and Hodge-Tate weights $0$ and $k-1$. 
In this
bijection, a maximal ideal $x$, corresponding to a finite extension $E_x$
of $E$, corresponds to a representation
$\rho_x : G_{\Q_p} \to \GL_2(E_x)$ such that there exists a lattice
giving the reduction $\bar\rho$ (as we consider only representations
$\bar\rho$ that have trivial endomorphisms, the lattice is unique up to
homothety if it exists, so there is no need to specify it).

The Breuil-Mézard conjecture gives us some information about these rings
(\cite{BM}, proved in \cite{Kis09}, \cite{Pas15}, \cite{Pas16}; and
\cite{Roz15} for the cases of discrete series extended type): 

\begin{theo}
\label{BM}
Let $\bar\rho$ be a continuous representation of $G_{\Q_p}$ of dimension
$2$, with trivial endomorphisms. If $p=3$, assume that $\bar\rho$ is not
a twist of an extension of $1$ by $\omega$, 
and let $(k,\tau,\bar\rho,\psi)$ be a deformation data.
Then there is an explicit integer
$\mua(k,\tau,\bar\rho)$ such that
$\e(R^\psi(k,\tau,\bar\rho)/(\pi_E)) = \mua(k,\tau,\bar\rho)$.
\end{theo}

For our purposes, what is important to know about 
$\mua(k,\tau,\bar\rho)$ is that it can be easily computed in
a combinatorial way, in terms of $\bar\rho$, $k$ and $\tau$. 
For more details on the formula for this integer see
the introduction of \cite{BM}.

\begin{defi}
\label{defgood}
We will say that a representation $\bar\rho$ with trivial endomorphisms
is good if it satisfies the hypothesis of Theorem \ref{BM}, that is, if
$p=3$ then $\bar\rho$ is not a twist of an extension of $1$ by $\omega$.
\end{defi}

Note that the condition of trivial endomorphisms implies that $\bar\rho$
is not reducible with scalar semi-simplification.

\subsection{Rigid spaces attached to deformation rings}
\label{rigiddefring}

As $R^\psi(k,\tau,\bar\rho)$ is a complete noetherian $\OO_E$-algebra, the
$E$-algebra $R^\psi(k,\tau,\bar\rho)[1/p]$ is an $E$-quasi-affinoid
algebra of open type as in Paragraph \ref{qaffalg}, and
$R^\psi(k,\tau,\bar\rho)$ is an $\OO_E$-model of it.  We denote by
$\X^{\psi}(k,\tau,\bar\rho)$ the rigid space attached to
$R^\psi(k,\tau,\bar\rho)[1/p]$ by the construction of Berthelot as
recalled in Paragraph \ref{dJconstr}.

Let $\fp_1,\dots,\fp_n$ the minimal prime ideals of
$R^\psi(k,\tau,\bar\rho)$, and let $R_i = R^\psi(k,\tau,\bar\rho)/\fp_i$. 
As $R^\psi(k,\tau,\bar\rho)$ has no $p$-torsion by construction, the set
of ideals $(\fp_i)$ is in bijection with the set of minimal prime ideals
$(\fp_i')$ of $R^\psi(k,\tau,\bar\rho)[1/p]$, with $R_i[1/p] =
R^\psi(k,\tau,\bar\rho)[1/p]/\fp_i'$. Let $\X_i$ be the rigid space
attached to $R_i[1/p]$, then $\X^{\psi}(k,\tau,\bar\rho) =
\cup_{i=1}^n\X_i$, and each $\X_i$ is an $E$-quasi-affinoid space of open
type.

Let $R_i^0$ be the integral closure of $R_i$ in $R_i[1/p]$, so that $R_i
\subset R_i^0 \subset R_i[1/p]$ and $R_i^0$ is finite over $R_i$.
As $R_i[1/p]$ is formally smooth, it is normal, hence so is $R_i^0$.
Hence we see that $R_i^0$ is equal to the ring
$\Gamma(\X_i,\OO_{X_i}^0)$ of analytic functions on $\X_i$ that are
bounded by $1$, that $R_i[1/p]$ is equal to the ring of bounded
analytic functions on $\X_i$. We deduce that
$R^\psi(k,\tau,\bar\rho)[1/p]$ is equal to the ring
$\A_E(\X^{\psi}(k,\tau,\bar\rho))$ and $\oplus_iR_i^0$ is equal to its
subring $\A_E^0(\X^{\psi}(k,\tau,\bar\rho))$.

\subsection{Results}

\subsubsection{Parameters on deformation spaces}

\begin{theo}
\label{parameter}
For all deformation data $(k,\tau,\bar\rho,\psi)$, there exist
a finite extension $E = E(k,\tau,\bar\rho,\psi)$ of $\Q_p$ such that
$\X^{\psi}(k,\tau,\bar\rho)$ is defined over $E$, and 
an analytic function $\lambda_{(k,\tau,\bar\rho,\psi)} :
\X^{\psi}(k,\tau,\bar\rho) \to
\PP^{1,rig}_E$ defined over $E$,
satisfying the following condition:
for all $\bar\rho$ and $\bar\rho'$, and $(k,\tau,\psi)$ such that
$(k,\tau,\bar\rho,\psi)$ and $(k,\tau,\bar\rho',\psi)$ are deformation
data, then $\lambda_{(k,\tau,\bar\rho,\psi)}(x) =
\lambda_{(k,\tau,\bar\rho',\psi)}(x')$ if and only if $x$ and $x'$
correspond to isomorphic representations.

In particular, each $\lambda_{(k,\tau,\bar\rho,\psi)}$ is 
injective on $\X^{\psi}(k,\tau,\bar\rho)(\bar\Q_p)$, and if there exist
$x$ and $x'$ such that $\lambda_{(k,\tau,\bar\rho,\psi)}(x) =
\lambda_{(k,\tau,\bar\rho',\psi)}(x')$ then $\bar\rho$ and $\bar\rho'$
have the same semi-simplification.
\end{theo}

The existence of the functions $\lambda_{(k,\tau,\bar\rho,\psi)}$ will be
proved as Propositions \ref{paramcrys}, \ref{crystabfam}, \ref{Linv}, and
\ref{Lgenan}, with an explanation of the choice of the field
$E(k,\tau,\bar\rho,\psi)$. 

\begin{coro}
\label{imparameter}
In the conditions of Theorem \ref{parameter},
the map $\lambda_{(k,\tau,\bar\rho,\psi)}$ defines an open immersion of
analytic spaces. The image of $\X^{\psi}(k,\tau,\bar\rho)(\bar\Q_p)$ by
$\lambda_{(k,\tau,\bar\rho,\psi)}$ is a standard subset
$X^{\psi}(k,\tau,\bar\rho)$ of
$\PP^1(\bar\Q_p)$ that is defined over $E(k,\tau,\bar\rho,\psi)$.
Moreover we have that 
$\A^0_{E(k,\tau,\bar\rho,\psi)}(\X^{\psi}(k,\tau,\bar\rho)) = 
\A^0_{E(k,\tau,\bar\rho,\psi)}(X^{\psi}(k,\tau,\bar\rho))$.
\end{coro}

\begin{proof}
Let $\X$ be a rigid analytic space that is smooth of dimension $1$, and
$f : \X \to \mathbb{P}^{1,rig}$ a rigid map that induces an injective map 
$\X(\bar\Q_p) \to \PP^1(\bar\Q_p)$. Then $f$ is an open immersion. Indeed,
this follows from the well-known fact that
an analytic function $f$ from some open disk $D$ to
$\bar\Q_p$ that is injective on $\bar\Q_p$-points satisfies $f'(x) \neq 0$ for all $x\in D$.
Now we apply this to $\X = \X^{\psi}(k,\tau,\bar\rho)$ and $f = \lambda_{(k,\tau,\bar\rho,\psi)}$.
We write $\lambda$ for $\lambda_{(k,\tau,\bar\rho,\psi)}$.
Let $X = X^\psi(k,\tau,\bar\rho)$ be the image of $\X(\bar\Q_p)$ by $\lambda$.
It is clear that $X$ is defined over $E$.

Assume first that $X$ is contained in some bounded subset of $\bar\Q_p$
(this is automatic when $\tau$ is an inertial type, see Paragraphs
\ref{cryssect} and \ref{crystabsect}).
Then $\lambda$
is an analytic open immersion from the quasi-affinoid space $\X$ to some
quasi-affinoid space $\D$ attached to an open disk in $\mathbb{A}^{1,rig}$.
By Corollary \ref{qaffstandard}, $X$ is a bounded standard subset of
$\PP^1(\bar\Q_p)$, and $\lambda$ induces an isomorphism between $\A_E(X)$ and
$\A_E(\X)$, and between $\A_E^0(X)$ and $\A_E^0(\X)$.

We do not assume anymore that $X$ is contained in some bounded subset of
$\bar\Q_p$.
By the Breuil-Mézard conjecture, 
there is an infinite number of $\bar\rho'$ with trivial
endomorphisms such that 
$X' = X^\psi(k,\tau,\bar\rho')$ is non-empty. 
For such a $\bar\rho'$, $X'$ contains a disk $D(a,r)^-$ for some $r>0$ as
it is open. 
For any $\bar\rho'$ with trivial endormophisms such that its
semi-simplification is not the same as the semi-simplification of
$\bar\rho$, we have that the intersection of $X$ and 
$X'$ is empty. So there exists some $a\in \PP^1(\bar\Q_p)$ and
$r>0$ such that $D(a,r)^- \cap X = \emptyset$. Let $u$ be
the rational function $u(x) = 1/(x-a)$, so that it sends $a$ to $\infty$,
then $u(X)$ is a bounded subset of $\PP^1(\bar\Q_p)$.  This means that
$u\circ \lambda$ is a bounded analytic function on $\X$.  So we can apply
the same reasoning as before to show that $u(X)$ is a bounded standard
subset of $\PP^1(\bar\Q_p)$, and so $X$ is a standard subset of
$\PP^1(\bar\Q_p)$.  \end{proof}

\subsubsection{Complexity bounds}

Recall that we denote by $X^{\psi}(k,\tau,\bar\rho)$ the subset
$\lambda_{(k,\tau,\bar\rho,\psi)}(\X^{\psi}(k,\tau,\bar\rho)(\bar\Q_p))$
of $\PP^1(\bar\Q_p)$.
Now we give more information on the sets 
$X^{\psi}(k,\tau,\bar\rho)$.

\begin{theo}
\label{bound}
Let $(k,\tau,\bar\rho,\psi)$ be a deformation data.
Then $X^{\psi}(k,\tau,\bar{\rho})$ is a standard subset of
$\PP^1(\bar\Q_p)$, 
defined over $E = E(k,\tau,\bar\rho,\psi)$, with 
$c_E(X^{\psi}(k,\tau,\bar{\rho})) \leq
e(R^{\psi}(k,\tau,\bar\rho)/(\pi_E))$. In particular, 
$c_E(X^{\psi}(k,\tau,\bar\rho)) \leq
\mua(k,\tau,\bar\rho)$ if $\bar\rho$ is good.
\end{theo}

\begin{rema}
Note that the right-hand side of the inequality does not depend on the
choice of $E$, whereas the left-hand side can get smaller when $E$ has
more ramification. In particular, to get a statement as strong as
possible we want to take $E$ with as little ramification as possible. 
\end{rema}

\begin{proof}
Let $\fp_1,\dots,\fp_n$ be the minimal prime ideals of $R^{\psi}(k,\tau,\bar\rho)$,
$R_i = R^{\psi}(k,\tau,\bar\rho)/\fp_i$ and $R_i^0$ be the integral closure of $R_i$
in $R_i[1/p]$ as in Section \ref{rigiddefring}. Let $\X_i$ be the rigid
space attached to $R_i[1/p]$, then $X^{\psi}(k,\tau,\bar\rho)$ is the disjoint union
of the $X_i = \lambda(\X_i(\bar\Q_p))$, and each of the $X_i$ is a standard subset
of $\PP^1(\bar\Q_p)$ which is defined over $E$. Then $\A^0_E(X_i) = R_i^0$, so
$c_E(X_i) = [k_{X_i,E}:k_E]\e(R_i^0)$ by definition. Note that
$k_{X_i,E}$ is the residue field of $R_i^0$, while $k_E$ is the residue
field of $R_i$. So by Proposition \ref{multsmaller}, we have
$c_E(X_i) \leq \e(R_i)$. So we get
$c_E(X^{\psi}(k,\tau,\bar\rho)) \leq \sum_{i=1}^n\e(R_i)$.
Finally, $\sum_{i=1}^n\e(R_i) = \e(R^{\psi}(k,\tau,\bar\rho))$ by \cite[Lemme
5.1.6]{BM}.
\end{proof}

Note that in the proof above, the decomposition $X^{\psi}(k,\tau,\bar\rho) =
\cup_iX_i$ is the decomposition of $X^{\psi}(k,\tau,\bar\rho)$ in standard subsets
that are defined and irreducible over $E$.
So we also have the following result:

\begin{prop}
\label{defring1/p}
Let $X^{\psi}(k,\tau,\bar\rho) = \cup_i X_i$ the decomposition of
$X^{\psi}(k,\tau,\bar\rho)$ in
standard subsets that are defined and irreducible over $E$. Then
$R^{\psi}(k,\tau,\bar\rho)[1/p] = \oplus_i\A_E(X_i)$.
\end{prop}

Finally, we have the following result:

\begin{theo}
\label{algogen}
Let $(k,\tau,\bar\rho,\psi)$ be a deformation data, and assume that
$\bar\rho$ is good.
There
exists a finite set $\E$ of finite extensions of $E =
E(k,\tau,\bar\rho,\psi)$,
depending only on $\mua(k,\tau,\bar\rho)$,
such that
$X^{\psi}(k,\tau,\bar\rho)$ is
determined by the sets $X^{\psi}(k,\tau,\bar\rho) \cap F$ for $F\in \E$.
\end{theo}

\begin{proof}
This is a consequence of Theorem \ref{BM} and Theorem \ref{algo}, where we take
$m = \mua(k,\tau,\bar\rho)$.
\end{proof}

\subsection{The case of crystalline deformation rings}
\label{crysrings}

We are interested here in the case of the deformation ring of crystalline
representations, that is, we take $\tau$ to be the trivial representation.
This case is of particular interest as we are able to deduce additional
information.

In this case $R^{\psi}(k,\text{triv},\bar\rho)$ is zero unless $\psi$ is
a twist of $\cyc^{k-1}$ by an unramified character.
Note that $R^{\psi}(k,\text{triv},\bar\rho)$ and
$R^{\psi'}(k,\text{triv},\bar\rho)$ are isomorphic as long as
$\psi/\psi'$ is an unramified character with trivial reduction modulo
$p$. So without loss of generality we will assume from now on that $\psi
= \cyc^{k-1}$ and $\det\bar\rho = \omega^{k-1}$.

We denote by $R(k,\bar\rho)$ the ring
$R^{\cyc^{k-1}}(k,\text{triv},\bar\rho)$. It parametrizes the set of
crystalline lifts of $\bar\rho$ with determinant $\cyc^{k-1}$ and
Hodge-Tate weights $0$ and $k-1$. We also write
$X(k,\bar\rho)$ for $X^{\cyc^{k-1}}(k,\text{triv},\bar\rho)$.
and 
$\mua(k,\bar\rho)$ for $\mua(k,\text{triv},\bar\rho)$.

Let $\F$ be the extension of $\F_p$ over which $\bar\rho$ is defined (so
$\F = \F_p$ when $\bar\rho$ is irreducible), and $E(\bar\rho)$ the
unramified extension of $\Q_p$ with residue field $\F$ (so
$E(\bar\rho)=\Q_p$ when $\bar\rho$ is irreducible). Then $R(k,\bar\rho)$
is an $\OO_{E(\bar\rho)}$-algebra with residue field $\F$.

\subsubsection{Classification of filtered $\phi$-modules}

For $a_p \in \bar\Z_p$ and $F$ a finite extension of $\Q_p$ containing
$a_p$, we define a filtered $\phi$-module $D_{k,a_p}$ as follows:
\begin{align*}
& D_{k,a_p} = Fe_1\oplus Fe_2 \\
& \phi(e_1) = p^{k-1}e_2, \quad  \phi(e_2) = -e_1+a_pe_2 \\
& \Fil^iD_{k,a_p} = D_{k,a_p} \text{ if }i\leq 0 \\
& \Fil^iD_{k,a_p} = Fe_1 \text{ if }1 \leq i\leq k-1 \\
& \Fil^iD_{k,a_p} = 0\text{ if }i \geq k
\end{align*}

Denote by $V_{k,a_p}$ the crystalline representation such that
$D_{cris}(V_{k,a_p}^*) = D_{k,a_p}$. Then: $V_{k,a_p}$ has Hodge-Tate
weights $(0,k-1)$ and determinant $\cyc^{k-1}$. Moreover, $V_{k,a_p}$ is
irreducible if $v_p(a_p)>0$, and a reducible non-split extension of an
unramified character by the product of an unramified character by
$\cyc^{k-1}$ if $v_p(a_p)=0$.
We have the following well-known result:

\begin{lemm}
\label{isomclass}
Let $V$ be a crystalline representation with Hodge-Tate
weights $(0,k-1)$ and determinant $\cyc^{k-1}$. If $V$ is irreducible
there exists a unique $a_p\in \m_{\bar\Z_p}$ such that $V$ is isomorphic
to $V_{k,a_p}$. If $V$ is reducible non-split there exists a unique
$a_p\in\bar\Z_p^\times$ such that $V$ is isomorphic
to $V_{k,a_p}$. 
\end{lemm}

\subsubsection{The parameter $a_p$}

We show in Proposition \ref{paramcrys} that the parameter $a_p$ actually
defines a rigid analytic function. This is the function that plays the
role of $\lambda$ of Theorem \ref{parameter} for crystalline
representations.


From Theorem \ref{parameter} we can already deduce some results. It is a
well-known
conjecture (see \cite[Conjecture 4.1.1]{BGslope}) that if $p>2$, $k$ is
even, and $v(a_p)\not\in \Z$ then $\bar{V}^{ss}_{k,a_p}$ is irreducible.
From this we get:

\begin{prop}
\label{annulusconj}
Let $p>2$, $k$ even, $n\in \Z_{\geq 0}$. If the conjecture above is true, then
there is an irreducible representation $\bar\rho$ (depending on $n$, $k$)
such that the set
$\{x,n<v_p(x)<n+1\}$ is contained in $X(k,\bar\rho)$.
\end{prop}

\begin{proof}
If the conjecture holds, then the set $C = \{x,n<v_p(x)<n+1\}$ is the
union of the $C \cap X(k,\bar\rho)$ for $\bar\rho$ irreducible. So we
have written $C$ as a finite disjoint union of standard subsets, which
means that one of these subsets is equal to $C$.
\end{proof}

\subsubsection{Reduction and semi-simplification}

We know want to show that the case of crystalline deformation rings is
accessible to numerical computations. However we must change slightly our
setting: indeed, we can compute numerically only the semi-simplified
reduction of $V_{k,a_p}$. So we need to express the result of Theorem
\ref{bound} in terms of semi-simple representations instead of in terms
of representations with trivial endomorphisms.

Let $\bar{r}$ be a semi-simple representation of $G_{\Q_p}$ with values
in $\GL_2(\bar\F_p)$. We define $Y(k,\bar{r})$ to be the set $\{a_p\in
D(0,1)^-, \bar{V}_{k,a_p}^{ss} = \bar{r}\}$. 
Let $\bar{\rho}$ be a representation of $G_{\Q_p}$ with trivial
endomorphisms with semi-simplification isomorphic to $\bar{r}$. 
Let $X'(k,\bar\rho) = X(k,\bar\rho) \cap D(0,1)^-$. This
means we are only interested in elements in $X(k,\bar\rho)$ that
correspond to irreducible representations $V_{k,x}$.
Then we have that $X'(k,\bar\rho) \subset Y(k,\bar{r})$. 
We want to know when this is an equality.

\begin{defi}
We say that a representation $\bar\rho$ with trivial endomorphisms is
nice if either $\bar\rho$ is irreducible, or $\bar\rho$ is a non-split
extension of $\alpha$ by $\beta$ where $\beta/\alpha \not\in
\{1,\omega\}$. 

We say that a semi-simple representation $\bar{r}$ is nice if $\bar{r}$ is not scalar,
and in addition when $p=3$ if $\bar{r}$ is not of the form $\alpha \oplus
\beta$ with $\alpha/\beta \neq \omega$. 
\end{defi}

Note that any $\bar\rho$ with trivial endomorphisms that is nice is also
good (in the sense of Definition \ref{defgood}), hence
satisfies the hypotheses of Theorem \ref{BM}. Moreover, its
semi-simplification is a nice semi-simple representation.
If $\bar{r}$ is semi-simple and 
nice, then there exists a nice $\bar\rho$ with trivial endomorphisms
such that $\bar\rho^{ss} = \bar{r}$, so we have $Y(k,\bar{r}) =
X'(k,\bar\rho)$. Note that we can choose such a $\bar\rho$ so that in
addition, $E(\bar\rho) = E(\bar{r})$.

\begin{prop}
\label{sameimage}
Let $\bar\rho$ be a nice representation with trivial endomorphisms.
Then $X'(k,\bar\rho) = Y(k,\bar\rho^{ss})$.
\end{prop}

\begin{proof}
The result is clear when $\bar\rho$ is irreducible.
Recall that $\dim \Ext^1(\alpha,\beta) > 1$ if and only if 
$\beta/\alpha \in \{1,\omega\}$. Suppose that 
$\bar\rho$ is an
extension of $\alpha$ by $\beta$ where $\beta/\alpha \not\in
\{1,\omega\}$.
Let $x\in Y(k,\bar\rho^{ss})$. By Ribet's Lemma, there exists a $G_{\Q_p}$-invariant 
lattice $T \subset V_{k,x}$ such that $\bar{T}$ is a non-split extension
of $\alpha$ by $\beta$, and so is isomorphic to $\bar\rho$. This means that
$x\in X'(k,\bar{\rho})$.
\end{proof}

We know some information about the difference between $X(k,\bar\rho)$ and
$X'(k,\bar\rho)$:

\begin{prop}
\label{unitcircle}
Let $\bar\rho$ be a representation of $G_{\Q_p}$  with trivial
endomorphisms.

If $\bar\rho$ is not an extension of $\unr(u)$ by 
$\unr(u^{-1})\omega^n$ for some $n$ which is equal to $k-1$ modulo $p-1$,
and $u\in\bar\F_p^\times$,
then
$X(k,\bar{\rho}) \subset D(0,1)^-$. 
If $\bar\rho$ is an extension of
$\unr(u)$ by $\unr(u^{-1})\omega^n$ for some $u\in\bar\F_p^\times$ and $0 \leq n
< p-1$, and $n = k-1$ modulo $p-1$, and $u\not\in\{\pm 1\}$ if $n=0$ or
$n=1$,
then $X(k,\bar\rho) \cap \{x,|x|=1\}$ is the disk $\{x,\bar{x}=u\}$.
\end{prop}

\begin{proof}
For $a_p \in \bar\Z_p^\times$, the representation $V_{k,a_p}$ is the
unique crystalline 
non-split extension of $\unr(u)$ by $\unr(u^{-1})\cyc^{k-1}$, where
$u\in\bar\Z_p^\times$ and $u$ and $u^{-1}p ^{k-1}$ are the roots of
$X^2-a_pX+p^{k-1}$. In particular, for any invariant lattice $T \subset
V_{k,a_p}$ such that $\bar{T}$ is non-split, we get that $\bar{T}$ is an
extension of $\unr(\bar{u})$ by $\unr(\bar{u}^{-1})\omega^{k-1}$.
So $X(k,\bar\rho)$ does not meet $\{x,|x|=1\}$ unless $\bar\rho$ has the
specific form given.
Moreover, $\bar{u} = \bar{a_p}$. So $X(k,\bar\rho) \cap \{x,|x|=1\}
\subset\{x,\bar{x}=u\}$. If $\bar\rho$ is an extension of
$\unr(u)$ by $\unr(u^{-1})\omega^n$ for some $u\in\bar\F_p$ and $0 \leq n
< p-1$, the conditions on $(n,u)$ imply there is a unique non-split
extension of $\unr(u)$ by $\unr(u^{-1})\omega^n$, and so 
$X(k,\bar\rho) \cap \{x,|x|=1\} = \{x,\bar{x}=u\}$
\end{proof}

\begin{coro}
\label{boundbis}
Let $\bar\rho$ be a representation with trivial endomorphisms.
Let $X'(k,\bar\rho) = X(k,\bar\rho) \cap D(0,1)^-$.  If $\bar\rho$ is not an extension $\unr(u)$ by 
$\unr(u^{-1})\omega^n$ for some $n$ which is equal to $k-1$ modulo $p-1$,
then $X'(k,\bar\rho) = X(k,\bar\rho)$ and $c_E(X'(k,\bar\rho)) \leq \e(R(k,\bar\rho))$.
If $\bar\rho$ is an extension $\unr(u)$ by 
$\unr(u^{-1})\omega^n$ for some $n$ which is equal to $k-1$ modulo $p-1$,
and $u\not\in\{\pm 1\}$ if $n=0$ or $n=1$,
then $c_E(X'(k,\bar\rho)) \leq \e(R(k,\bar\rho)) -1$.
\end{coro}

\begin{proof}
The first part is clear by Proposition \ref{unitcircle}. 

For the second part, we can write
$X(k,\bar\rho)$ as a disjoint union of $X'(k,\bar\rho)$ and
$X^+(k,\bar\rho) = X(k,\bar\rho) \cap \{x,|x|=1\}$, and both are standard
subsets defined over $E$, so
$c_E(X(k,\bar\rho)) = c_E(X'(k,\bar\rho)) + c_E(X^+(k,\bar\rho))$.
By Proposition \ref{unitcircle}, $c_E(X^+(k,\bar\rho))=1$ under the
hypotheses, hence the result.
\end{proof}

\subsubsection{Local constancy results}

We recall the following results:

\begin{prop}
\label{locconst}
Let $a_p\in\m_{\bar\Z_p}$. If $a_p \neq 0$, then for all $a_p'$ such that 
$v_p(a_p-a_p') > 2v_p(a_p) + \lfloor p(k-1)/(p-1)^2\rfloor$, we have
$\bar{V}_{k,a_p}^{ss} \simeq \bar{V}_{k,a_p'}^{ss}$. moreover,
$\bar{V}_{k,a_p}^{ss}
\simeq \bar{V}_{k,0}^{ss}$ for all $a_p$ with $v_p(a_p) > \lfloor (k-2)/(p-1)
\rfloor$.
\end{prop}

\begin{proof}
The result for $a_p \neq 0$ is Theorem A of \cite{Ber12}. The result for
$a_p=0$ is the main result of \cite{BLZ}.
\end{proof}

\subsubsection{Computation of $Y(k,\bar{r})$}

We explain now how we can compute numerically the sets $Y(k,\bar{r})$ for
$\bar{r}$  semi-simple and nice (and hence the sets $X(k,\bar\rho)$ for $\bar\rho$ with
nice semi-simplification).

From Corollary \ref{boundbis} we deduce (using the fact that a nice
representation with trivial endomorphisms is good and so satisfies the
hypotheses of Theorem \ref{BM}):

\begin{prop}
\label{boundss}
Suppose that $\bar{r}$ is a nice semi-simple representation, 
and let $\bar\rho$ be a nice representation with trivial endomorphisms with
$\bar\rho^{ss} = \bar{r}$. Then $Y(k,\bar{r})$ is a standard subset of
$D(0,1)^-$ defined over $E = E(\bar{r})$, with $c_E(Y(k,\bar{r})) \leq
\mua(k,\bar\rho)$. Moreover if $\bar\rho$ is an extension of an
unramified character by another character then 
$c_E(Y(k,\bar{r})) \leq \mua(k,\bar\rho)-1$.
\end{prop}

Theorem \ref{algogen} specializes here to:

\begin{theo}
\label{algocris}
Let $\bar{r}$ be a nice semi-simple representation. 
Then there
exists a finite set $\E$ of finite extensions of $E = E(\bar{r})$,
depending only on $k$ and $\bar{r}$, such that $Y(k,\bar{r})$ is
determined by the sets $Y(k,\bar{r}) \cap F$ for $F\in \E$.
\end{theo}

\begin{proof}
This is Corollary \ref{algo}, where we take for $E$ the field
$E(\bar{r})$, and for $m$ the bound given by Proposition \ref{boundss},
that is $m = \mua(k,\bar\rho)$ or $\mua(k,\bar\rho)-1$ where $\bar\rho$
is some nice representation with $\bar\rho^{ss}=\bar{r}$.
\end{proof}

\begin{theo}
\label{algoepscris}
Let $\bar{r}$ be a nice semi-simple representation. 
Then there
exists a finite set of points $\p \subset D(0,1)^-$, depending only on
$k$ and $\bar{r}$, such that $Y(k,\bar{r})$ is determined by
$Y(k,\bar{r}) \cap \p$.
\end{theo}

\begin{proof}
This is Corollary \ref{algoeps}, where we take for $E$ the field
$E(\bar{r})$, for $m$ the bound given by Proposition \ref{boundss}, and for
$\eps$ we can take the norm of an element of valuation 
$\lfloor 3p(k-1)/(p-1)^2\rfloor$ 
by Proposition \ref{locconst}.
\end{proof}

\begin{coro}
\label{algocrisbis}
Let $\bar\rho$ be a nice representation with trivial endomorphisms.
Then there
exists a finite set of points $\p \subset D(0,1)^-$, depending only on
$k$ and $\bar\rho^{ss}$, such that $X(k,\bar\rho)$ is determined by
$X(k,\bar\rho) \cap \p$.
\end{coro}

\begin{proof}
Let $\bar{r} = \bar\rho^{ss}$. Then $\bar{r}$ is a nice semi-simple
representation, so we can apply Theorem \ref{algoepscris} to compute
$Y(k,\bar{r}) = X(k,\bar\rho) \cap D(0,1)^-$, and Proposition
\ref{unitcircle} to determine the rest of $X(k,\bar\rho)$.
\end{proof}

As a consequence, we see that if we are able to compute
$\bar{V}^{ss}_{k,a_p}$ for given $p$, $k$, $a_p$, then we can compute
$Y(k,\bar{r})$ for $\bar{r}$ nice in a finite number of such computations, 
bounded in terms of $E(\bar{r})$ and $k$. We give some examples of such
computations in Section \ref{examples}.

We give a last application of these results:
It follows from the formula giving $\mua(k,\bar\rho)$ that there exists
an integer $m(k)$, depending only on $k$, such that $\mua(k,\bar\rho)
\leq m(k)$ for all $\bar\rho$. The optimal value for $m(k)$ is of the
order of $4k/p^2$ when $k$ is large.

In general, the value of $\bar{V}_{k,a_p}^{ss}$ depends on more
information than just the valuation of $a_p$. But there are some cases
where it depends only on $v_p(a_p)$:

\begin{coro}
Fix $k$, and let
$m$ be an integer such that $m \geq \e(R(k,\bar\rho))$ for all nice
$\bar\rho$ with trivial endomorphisms. Let $a$ and $b$ be rational
numbers such that for all rational $c$ between $a$ and $b$, the
denominator of $c$ is strictly larger than $m$. Then either for all $a_p$
with $a<v_p(a_p)<b$, $\bar{V}_{k,a_p}^{ss}$ is not nice, or
$\bar{V}_{k,a_p}^{ss}$ is constant on the annulus $A_0(a,b)$.

In particular, let $c \in \Q$ with denominator strictly larger than $m$.
Then either for all $a_p$
with $v_p(a_p)=c$, $\bar{V}_{k,a_p}^{ss}$ is not nice, or
$\bar{V}_{k,a_p}^{ss}$ is constant on the circle $C_0(c)$.
\end{coro}

Note that if $p>3$ and $k$ is even, $\bar{V}_{k,a_p}^{ss}$ is always
nice.

\begin{proof}
Suppose that there exists at least an $a_p$ in $A_0(a,b)$ such that $\bar{r}
= \bar{V}_{k,a_p}^{ss}$ is nice. Then $c_E(Y(k,\bar{r})) \leq m$ for $E =
E(\bar\rho)$ which is an unramifed extension of $\Q_p$. So we can apply
Corollary \ref{annulusin}: the annulus $A_0(a,b)$ is a subset of
$Y(k,\bar{r})$.
\end{proof}

\section{Numerical examples}
\label{examples}

We give some numerical examples for the deformations rings of crystalline
representations.
We have computed some examples of $X(k,\bar\rho)$ using Theorem
\ref{algoepscris} and a computer program written in SAGE
(\cite{sage}) that implements the algorithm
described in \cite{Roz16}. We also used the fact that
$\bar{V}_{k,a_p}^{ss}$ is known for $v_p(a_p) < 2$ in all cases for $p
\geq 5$, by the results of \cite{BG09,BG13,GG15,BG15,BGR,GR}, which
reduces the number of computations that are necessary to determine
$X(k,\bar\rho)$.

We make the following remark: let $\bar\rho$ be a
representation such that $\bar\rho\otimes\unr(-1)$ is isomorphic to
$\bar\rho$. Then $X(k,\bar\rho)$ is invariant by $x \mapsto -x$. Indeed,
$V_{k,-a_p}$ is isomorphic to $V_{k,a_p}\otimes\unr(-1)$. This applies
in particular when $\bar\rho$ is irreducible.

\subsection{Observations for $p=5$}

We have computed $X(k,\bar\rho)$ for $p=5$, $k$ even, $k \leq 102$, or $k$
odd and $k \leq 47$, and
$\bar\rho$ irreducible (so in this case we have $E(\bar{\rho}) =
\Q_p$).

We summarize here some observations from these computations:
\begin{enumerate}
\item
in each case, we have $\bar{V}_{k,a_p}^{ss} =\bar{V}_{k,0}^{ss}$ for all
$a_p$ with $v_p(a_p) > \lfloor (k-2)/(p+1) \rfloor$, and not only
$v_p(a_p) > \lfloor (k-2)/(p-1) \rfloor$ which is the value predicted by \cite{BLZ}.

\item
in each case, we have $c_{\Q_p}(X(k,\bar\rho)) = \e(R(k,\bar\rho))$, that
is, the inequality of Proposition \ref{boundss} is an equality.

\item
each defining disk $D$ of a $X(k,\bar\rho)$ has $\gamma_{\Q_p}(D) = 1$. 

\item
each defining disk $D$ of a $X(k,\bar\rho)$ is defined over an extension
of $\Q_p$ of degree at most $2$, which is unramified if $k$ is even and
totally ramified if $k$ is odd.

\item
for each defining disk $D$ of a $X(k,\bar\rho)$, either $0 \in D$, or 
$D$ is included in the set $\{x, v_p(x) = n\}$ for some $n\in\Z_{\geq 0}$ if
$k$ is even, and in the set $\{x, v_p(x) = n+1/2\}$ for some
$n\in\Z_{\geq 0}$ if $k$ is odd.

\end{enumerate}

It would be interesting to know which of these properties hold in
general. Property (1) is expected to be in fact true for all $p$ and
$k$, but nothing is known about the other properties.
We comment further on Property (2) in Section \ref{conjecture}.

\subsection{Some detailed examples}
\label{detailled_examples}

Let $p=5$.
Let $\bar{r}_0 = \ind\omega_2$ and $\bar{r}_1 = \ind\omega_2^3$, and for
all $n$, $\bar{r}(n) = \bar{r}\otimes\omega^n$. We describe a few
examples of sets $X(k,\bar{r})$. In each case, the sets given contain all
the values of $a_p$ for which $\bar{V}_{k,a_p}^{ss}$ is irreducible.
We also give the generic fibers of the deformation rings.

\subsubsection{The case $k=26$}

We get that: 
\begin{itemize}

\item
$$
X(26,\bar{r}_0) = \{x, v_p(x) < 2\} \cup \{x, v_p(x) > 2\}
$$
with $c_{\Q_p}(X(26,\bar{r}_0)) = 3$, 
and $R(26,\bar{r}_0)[1/p] = 
\left(\Z_p[[X]]\otimes\Q_p\right)
\times 
\left(\Z_p[[X,Y]]/(XY-p)\otimes\Q_p\right)$.

\item
$$
X(26,\bar{r}_0(2)) = \{x, v_p(x-a)>3\} \cup \{x, v_p(x+a) > 3\}
$$ 
where $a = 4\cdot 5^2$,
with $c_{\Q_p}(X(26,\bar{r}_0(2))) = 2$,
and $R(26,\bar{r}_0(2))[1/p] = 
\left(\Z_p[[X]]\otimes\Q_p\right)^2$.

\item
$$
X(26,\bar{r}_1(1)) = \{x, 2 < v_p(x-a) < 3\} \cup \{x, 2 < v_p(x+a)
< 3\}
$$ 
with $c_{\Q_p}(X(26,\bar{r}_1(1))) = 4$,
and $R(26,\bar{r}_1(1))[1/p] = 
\left(\Z_p[[X,Y]]/(XY-p)\otimes\Q_p\right)^2$.

\end{itemize}

Here we see an example where the geometry begins to be a little
complicated, 
with annuli that do not have $0$ as a center.

\subsubsection{The case $k=28$}
\label{ex28}

We get that: 
\begin{itemize}

\item
$$
X(28,\bar{r}_1) =  \{x, v_p(x) > 2\text{ and }
v_p(x-a)<4\text{ and }v_p(x+a)<4\}\cup \{x, 0 < v_p(x) < 1\}
$$
where $a = 4\cdot 5^3 + 5^4$,
with $c_{\Q_p}(X(28,\bar{r}_1)) = 5$, 
and we get that 
\begin{eqnarray*}
& R(28,\bar{r}_1)[1/p] = 
 \left(\Z_p[[X,Y]]/(XY-p)\otimes\Q_p\right)
\times \\
& \left(\Z_p[[X,Y,Z]]/(XY-p^2-(a/p^2)Y,XZ-p^2+(a/p^2)Z,YZ-(p^4/2a)(Y-Z))\otimes\Q_p\right).
\end{eqnarray*}

\item
$$
X(28,\bar{r}_0(1)) = \{x, 1 < v_p(x) < 2\}
$$
with $c_{\Q_p}(X(28,\bar{r}_0(1))) = 2$,
and $R(28,\bar{r}_0(1))[1/p] = 
\left(\Z_p[[X,Y]]/(XY-p)\otimes\Q_p\right)$.

\item
$$
X(28,\bar{r}_0(3)) = \{x, v_p(x-a) > 4\} \cup \{x, v_p(x+a) > 4\}
$$ 
with $c_{\Q_p}(X(28,\bar{r}_0(3))) = 2$,
and $R(28,\bar{r}_0(3))[1/p] = 
\left(\Z_p[[X]])\otimes\Q_p\right)^2$.
\end{itemize}

Here we see an example 
with an irreducible component that has complexity $3$.

\subsubsection{The case $k=30$}
\label{ex30}

We get that: 
\begin{itemize}

\item
$$
X(30,\bar{r}_0) = \{x, 0 < v_p(x) < 1\} \cup \{x, v_p(x) > 4\}
$$
with $c_{\Q_p}(X(30,\bar{r}_0)) = 3$,
and $R(30,\bar{r}_0)[1/p] = 
\left(\Z_p[[X]]\otimes\Q_p\right)
\times 
\left(\Z_p[[X,Y]]/(XY-p)\otimes\Q_p\right)$.

\item
$$
X(30,\bar{r}_0(2)) = \{x, v_p(x-a)>3\} \cup \{x, v_p(x+a) > 3\}
$$ 
where $a = 5^3\cdot\sqrt{3}$,
with $c_{\Q_p}(X(30,\bar{r}_0(2))) = 2$,
and $R(30,\bar{r}_0(2))[1/p] = 
\left(\Z_{p^2}[[X]]\otimes\Q_{p^2}\right)$.

\item
$$
X(30,\bar{r}_1(1)) = \{x, 1 < v_p(x) < 3\} \cup \{x, 3 < v_p(x)
< 4\}
$$ 
with $c_{\Q_p}(X(30,\bar{r}_1(1))) = 4$,
and 
$$
R(30,\bar{r}_1(1))[1/p] = 
\left(\Z_p[[X,Y]]/(XY-p)\otimes\Q_p\right)
\times
\left(\Z_p[[X,Y]]/(XY-p^2)\otimes\Q_p\right).
$$

\end{itemize}

The interesting part here is $X(30,\bar{r}_0(2))$: we see that
$\A_{\Q_p}^0(X(30,\bar{r}_0(2)))$, which is a domain, has residue field $\F_{p^2}$, whereas 
$R(30,\bar{r}_0(2))$ has residue field $\F_p$. So 
$R(30,\bar{r}_0(2)) \neq \A_{\Q_p}^0(X(30,\bar{r}_0(2)))$.

\subsection{Criteria for non-normality}
\label{nonnormal}

Recall the notation of Section \ref{rigiddefring}. Then we see, by
Proposition \ref{defring1/p}, that if we know $X(k,\bar\rho)$ then we
know $R(k,\bar\rho)[1/p] = \oplus_i R_i[1/p] = \oplus_i\A_E(X_i)$.
We can ask whether we can recover each $R_i$, that is, if $R_i =
\A^0_E(X_i)$, or equivalently if $R_i = R_i^0$ for all $i$ (the
description of $X(k,\bar\rho)$ gives no indication about how the $R_i$
glue together so we cannot hope for complete information on
$R(k,\bar\rho)$ anyway if it is not irreducible). We do not
expect this to hold, as this would mean that each of the $R_i$ is a
normal ring. So we can ask instead, how can we recognize when $R_i$ is
not $R_i^0$?

A first criterion is when they have different residue fields, as in the
example of $R(30,\bar{r}_0(2))$ in Paragraph \ref{ex30}. 
Another criterion is when $R_i$ and $R_i^0$ have the same residue field
(a situation that we can always obtain by replacing $E$ by an unramified
extension, which does not change the complexities), but $\e(R_i^0) <
\e(R_i)$. This is a situation that does not seem to arise often, see
Section \ref{conjecture}.

We give a last, more subtle criterion. Let $X_i$ be one of the components
of $X(k,\bar\rho)$, and assume that each of the disks that appears in the
description of $X_i$ is defined over $\Q_p$, and has complexity $1$. In
this case, a closer look at the proof of Proposition \ref{additivityc}
show that $\spec(\A^0_{\Q_p}(X_i)/p)$ has exactly $c_{\Q_p}(X_i)$ distinct irreducible
components. 
On the other hand, the geometric version of the Breuil-Mézard conjecture,
proved in \cite{EG}, shows that if $\bar\rho$ is irreducible then
$\spec(R(k,\bar\rho)/p)$ has at most two irreducible components (which
can have large multiplicity), and so $\spec(R_i/p)$ also has at most two
irreducible components. So if $c_{\Q_p}(X_i) >2$ then we certainly
have that $R_i \neq R_i^0$. This happens for example for the second
irreducible component of $X(28,\bar{r}_1)$.
It would be interesting in this case to understand how the irreducible
components of $\spec(R_i^0/p)$ map to the irreducible components of
$\spec(R_i/p)$.

\subsection{Complexity and multiplicity}
\label{conjecture}

An interesting result coming from our computations is the following: for
$p= 5$, for all irreducible representation $\bar\rho$, for all $k \leq
47$ and all even $k \leq 102$, we have that 
$c_{\Q_p}(X(k,\bar\rho)) = \e(R(k,\bar\rho))$, 
instead of simply the inequality $c_{\Q_p}(X(k,\bar\rho)) \leq  \e(R(k,\bar\rho))$. 
Given this, it is tempting to make the following conjecture:
For all $p>2$, for all $k \geq 2$ and for all irreducible $\bar\rho$, we
have that $c_{\Q_p}(X(k,\bar\rho)) = \e(R(k,\bar\rho))$.

Note that this equality between complexity and multiplicity does not
necessarily hold when $\bar\rho$ is reducible. 
Consider the following example: let $p=5$ and $k=16$. Then we can
compute that $X(16,\bar{r}_1)$ is the set $\{x, v(x) >0, v(x) \neq 1\}$.
So the set of $a_p$ for which the reduction is reducible is contained in
the set of $a_p$ with $v(a_p)= 0$ or $v(a_p)=1$.
For $v(a_p)=0$, the reduction is of the form $\omega^3\oplus 1$ when
restricted to inertia.
The reduction for the values of $a_p$ with $v(a_p)=1$ is entirely
computed in \cite{BGR}, from which we get that for $\lambda \in
\bar\F_p^\times$, the (semi-simplified) reduction is
$\unr(\lambda)\omega^2\oplus\unr(\lambda^{-1})\omega$ for exactly the
values of $a_p$ of valuation $1$ for which $\lambda =
2(\bar{a_p/p}-\bar{p/a_p})$. So for each $\lambda$, we get that
$X(16,\unr(\lambda)\omega^2\oplus\unr(\lambda^{-1})\omega)$ is the union
of two disks, and has complexity $2$, except for $\lambda = \pm 1$, where
this is just one disk, and has complexity $1$. On the other hand, for any
$\lambda\in\bar\F_p^\times$ we have
$\e(R(16,\unr(\lambda)\omega^2\oplus\unr(\lambda^{-1})\omega))=2$.
So we do not always have the equality of multiplicity and complexity in
the case where $\bar\rho$ is reducible.
However, it may be true
that for all $p>2$, for all $k \geq 2$, there is only a finite number of
reducible (nice) representations $\bar\rho$ for which the equality does
not hold.

We can also reformulate this equality in a different way: recall the
notation of Section \ref{rigiddefring}. So $R(k,\bar\rho)$ has a family
of quotients $R_i$ that are integral domains, and $\e(R(k,\bar\rho)) =
\sum_i \e(R_i)$. On the other hand, $c_{\Q_p}(X(k,\bar\rho)) = \sum_i
[k_{R_i^0}:\F_p]\e(R_i^0)$ where $k_{R_i^0}$ is the residue field of
$R_i^0$. The equality between complexity and multiplicity can be
reformulated as saying that for all $i$, $\e(R_i) =
[k_{R_i^0}:\F_p]\e(R_i^0)$. Written in this way without any reference to
the sets $X(k,\bar\rho)$, the equality can be
generalized to any potentially semi-stable deformation ring, including
those that are of dimension larger
than $1$, such as the deformation rings classifying
representations of dimension $>2$ or representations of
$G_K$ for some finite extension $K/\Q_p$. 

\section{Parameters classifying potentially semi-stable representations}
\label{paramsect}

This Section is devoted to the proof of Theorem \ref{parameter}. We start
with some preliminaries, and then give the proof for the various cases
starting in Paragraph \ref{cryssect}.

\subsection{Results on Weil representations}
\label{WDrep}

\subsubsection{Field of definition}

Let $W_{\Q_p}$ be the Weil group of $\Q_p$. A Weil representation is a
representation of $W_{\Q_p}$ with coefficients in $\bar\Q_p$
that is trivial on an open subgroup of $I_{\Q_p}$.

Let $\tau$ be a Weil representation. The field of definition of $\tau$,
denoted by $E(\tau)$, is
the subfield of $\bar\Q_p$ generated by the $\tr\tau(x)$, $x\in
W_{\Q_p}$. This is a finite extension of $\Q_p$, as a Weil representation
factors through a finitely generated group.

Let $E$ be a finite extension of $\Q_p$. We say that $\tau$ is realizable
over $E$ if there is a representation $\tau': W_{\Q_p} \to \GL_n(E)$ that
is isomorphic to $\tau$. Then we have:

\begin{lemm}
\label{unramW}
Let $\tau$ be an irreducible Weil representation. 
Then there exists a finite unramified extension $E$ of $E(\tau)$ such
that $\tau$ is realizable over $E$.
\end{lemm}

\begin{proof}
From the results of \cite[1.4]{Kra}, we see that
the obstruction to realizing $\tau$ over $E(\tau)$ is in the
Brauer group of $E(\tau)$. 
An element of the Brauer group can be killed by taking a finite
unramified extension, hence the result.
\end{proof}

\subsubsection{$(\phi,\GF)$-modules}

We fix a finite Galois extension $F$ of $\Q_p$, and denote by $F_0$ the
maximal subextension of $F$ that is unramified over $\Q_p$.

Let $A$ be a $\Q_p$-algebra. Then a $(\phi,\GF)$-module $M$ over
$F_0\otimes_{\Q_p} A$ is a
free $F_0\otimes_{\Q_p}A$-module of finite rank, endowed with commuting
actions of an automorphism $\phi$ and the group $\GF$. The action of
$\phi$ is $A$-linear and $F_0$-semi-linear (with respect to the Frobenius
automorphism of $F_0$), and the action of $\GF$ is $F_0$-semi-linear
(with respect to the action of $\GF$ on $F_0$) and $A$-linear.

Then:
\begin{prop}
\label{Weil}
Let $A$ be an $F_0$-algebra. Then there is an equivalence of categories
between $(\phi,\GF)$-modules over $F_0\otimes_{\Q_p} A$ and Weil representations over a
free $A$-module that are trivial on $I_F$, and this equivalence preserves
rank. Moreover this construction is functorial in $A$ (in the category of
$F_0$-algebras).
\end{prop}

\begin{proof}
For a given $A$,
the construction of the Weil representation from the $(\phi,\GF)$-module
is explained in \cite{BM}, and the converse construction is immediate.
\end{proof}

We will make use of this equivalence as some things are more naturally
expressed in terms of $(\phi,\GF)$-modules, whereas others are more
easily proved in terms of representations of the Weil group (for example
Proposition \ref{familyrepr}).

In the same situation, we also define a $(\phi,N,\GF)$-module over
$F_0\otimes_{\Q_p}A$ to be a
$(\phi,\GF)$-module over $F_0\otimes_{\Q_p}A$ that is additionally endowed with a
$F_0\otimes_{\Q_p}A$-linear endomorphism $N$ satisfying $N\phi = p\phi N$
that commutes with the action of $\GF$.

\subsection{Universal (filtered) $(\phi,N)$-modules with descent data}
\label{universal}

We recall a few definitions concerning objects attached to $p$-adic
representations of $G_{\Q_p}$.
If $F/\Q_p$ is a finite extension, we denote by $F_0$ be maximal
unramified extension of $\Q_p$ contained in $F$. 

Let $V$ be a continuous representation of $G_{\Q_p}$ over an $E$-vector space
for some finite $E/\Q_p$. Let $F$ be a finite Galois extension of $\Q_p$.
We denote by $\Dcr{F}(V)$ the $F_0\otimes_{\Q_p}E$-module 
$(B_{crys}\otimes_{\Q_p}V)^{G_F}$. It is a $(\phi,\GF)$-module
over $F_0\otimes_{\Q_p}E$.
If $V$ becomes crystalline over $F$ then $\Dcr{F}(V)$ is a free
$F_0\otimes_{\Q_p}E$-module of rank $\dim_E(V)$.
We denote by $\Dst{F}(V)$ the $F_0\otimes_{\Q_p}E$-module 
$(B_{st}\otimes_{\Q_p}V)^{G_F}$. It is endowed with a structure of
$(\phi,N,\GF)$-module over $F_0\otimes_{\Q_p}E$.
If $V$ becomes semi-stable over $F$ then is it a free
$F_0\otimes_{\Q_p}E$-module of rank $\dim_E(V)$. If $V$ becomes
crystalline over $F$ then $\Dst{F}(V)$ and $\Dcr{F}(V)$ coincide as 
$(\phi,\GF)$-modules, and $N=0$. 
We denote by $\Ddr{F}(V)$ the $F\otimes_{\Q_p}E$-module 
$(B_{dR}\otimes_{\Q_p}V)^{G_F}$. It is a $F\otimes_{\Q_p}E$-module with a
semi-linear action of $\GF$, and is endowed with a separated
exhaustive decreasing filtration by sub-$F\otimes_{\Q_p}E$-modules that
is stable under the action of $\GF$, and satisfies an additional
condition called admissibility. If $V$ is potentially semi-stable, then
$\Ddr{\Q_p}(V)$ is an $E$-vector space of dimension $\dim_E(V)$.
Moreover, we have that $\Ddr{F}(V) = F\otimes_{F_0}\Dst{F}(V)$ as an
$F\otimes_{\Q_p}E$-module, so this endows $F\otimes_{F_0}\Dst{F}(V)$ with a
filtration as above, that is, a structure of filtered
$(\phi,N,\GF)$-module.

\begin{theo}
\label{modulest}
Let $F$ be a finite Galois extension of
$\Q_p$. Let $X$ be a reduced rigid analytic space, let $\V$ be a locally
free $\OO_X$-module of rank $n$ with a continuous action of $G_{\Q_p}$. Assume
that for all $x \in X$, $\V_x$ is potentially semi-stable with weights
independent of $x$, and becomes semi-stable over $F$. Then there exists a
projective $F_0\otimes_{\Q_p}\OO_X$-module $D$ of rank $n$, endowed with
a structure of $(\phi,N,\GF)$-module over $F_0\otimes_{\Q_p}\OO_X$,
such that for
all $x$, $D_x$ is isomorphic, as a $(\phi,N,\GF)$-module,
to $\Dst{F}(\V_x)$.
\end{theo} 

\begin{proof}
This follows immediately from \cite[Theorem 5.1.2]{Bel}: we take the
module $D$ to be the module called $\D_{B_{st}}(\V)$ there,
considering $\V$ as a representation of $G_F$ (see also
\cite[Théorème C]{BC}).
\end{proof}

\begin{theo}
\label{moduledR}
Let $F$ be a finite Galois extension of
$\Q_p$. Let $X$ be a reduced rigid analytic space, let $\V$ be a locally
free $\OO_X$-module of rank $n$ with a continuous action of $G_{\Q_p}$. Assume
that for all $x \in X$, $\V_x$ is potentially semi-stable with weights
independent of $x$, and becomes semi-stable over $F$. Let $D$ be as in
the conclusion of Theorem \ref{modulest}.
Then
$F\otimes_{F_0}D$ is endowed of a filtration by locally free
sub-$F\otimes_{\Q_p}\OO_X$-modules, such that the graded parts are also
locally free, such that for
all $x$, $(F\otimes_{F_0}D)_x$ is isomorphic, as a filtered
$(\phi,N,\GF)$-module,
to $\Ddr{F}(\V_x)$.
\end{theo}

\begin{proof}
This follows from \cite[Theorem 5.1.7]{Bel}, as $F\otimes_{F_0}D$ is the
$F\otimes_{\Q_p}\OO_X$-module that is called $\D_{B_{dR}}(\V)$ there,
considering $\V$ as a representation of $G_F$. Indeed the filtration, and
the graded parts, are given by the modules called
$\D_{B_{dR}}^{[a,b]}(\V)$. The point that we need to check is that for
all $[a,b]$, the $F\otimes_{\Q_p}E_x$-modules
$\D_{B_{dR}}^{[a,b]}(\V_x)$ are actually free (then their rank is
independent of $x$ by the condition on the weights). This comes from
\cite[Lemma 2.1]{Sav}, and here we use the fact that we start
from a representation of $G_{\Q_p}$.
\end{proof}

Let now $(k,\tau,\bar\rho,\psi)$ be a deformation data, as defined in
Definition \ref{defdata}. Let $E$ be a finite
extension of $\Q_p$ satisfying the following conditions:
\begin{enumerate}
\item
the residual representation $\bar\rho$ can be realized on the residue
field of $E$

\item
the type $\tau$ can be realized on $E$

\item
the character $\psi$ takes its values in $E^\times$
\end{enumerate}

Let $R^\psi(k,\tau,\bar\rho)[1/p]$ be the ring defined by Kisin attached to
this data, as recalled in Section \ref{kisinrings}. It is an
$\OO_E$-algebra.
We can apply Theorems \ref{modulest} and \ref{moduledR} to the rigid
analytic space $X = \X^\psi(k,\tau,\bar\rho)$ attached to the Kisin
ring $R^\psi(k,\tau,\bar\rho)[1/p]$. Indeed, we know that these rings are
reduced, and the hypotheses come from the definition of the rings.

\subsection{Working in families}

\subsubsection{Reduction of an endomorphism}

\begin{prop}
\label{diago}
Let $K$ be a field and $A$ be a $K$-algebra. 
Let $\phi$ be an $A$-linear
endomorphism of $A^2$, and assume that the characteristic polynomial of
$\phi$ is in fact in $K[X]$, and that it is split over $K$ with distinct
eigenvalues. Then, Zariski-locally on $A$, $\phi$ is diagonalizable.
\end{prop}

\begin{proof}
Let $\lambda$ and $\mu$ be the roots of the characteristic polynomial of
$\phi$, and let $\matr  abcd$ be the matrix of $\phi$ in the canonical
basis of $A^2$ (so that $a+d = \lambda+\mu$ and $ad-bc = \lambda\mu$).

We are looking for a basis $(f_1,f_2)$ of $A^2$, with $f_1 = xe_1+ye_2$, $f_2 = e_2$,
such that the matrix of $\phi$ in this basis is upper triangular.
The new basis is as wanted if $x,y$ satisfy one of the following
systems of equations:
$$
(a-\lambda)x+by=0 \qquad\text{and}\qquad
cx + (d-\lambda)y = 0
$$
or
$$
(a-\mu)x+by=0 \qquad\text{and}\qquad
cx + (d-\mu)y = 0.
$$

Assume that $u = d-\lambda$ is invertible. We solve the first system by
setting $x=1$, $y=-c/(d-\lambda)$.
In the first case, in our new basis $\phi$ has a matrix of the form
$\matr {\lambda}{b}{0}{d}$, and actually $d = \mu$ by the trace
condition. As $\lambda-\mu$ is invertible, we can change the basis again
so that in the new basis, $\phi$ has matrix $\matr
{\lambda}00{\mu}$.

Assume now that $v = a-\lambda$ is invertible. Then so is $d-\mu = -v$.
We solve the second system by setting $x=1$, $y = -c/(d-\mu)$.
In this case we do the same thing after exchanging $\lambda$ and $\mu$.

Note that $u+v = \mu-\lambda$ is invertible by assumption.
We set $f = (d-\lambda)/(\mu-\lambda)$, $A_1 = A[f^{-1}]$, $A_2 =
A[(1-f)^{-1}]$. Then as we just saw in $A_1$ and $A_2$ there is a basis
in which the matrix of $\phi$ is $\matr \lambda{0}{0}\mu$, which gives
the result.
\end{proof}

\subsubsection{Isomorphism of group representations}

\begin{theo}
\label{familyrepr}
Let $K$ be a field of characteristic zero, and $A$ a $K$-algebra.

Let $G$ be a group.
Let $\rho : G \to \GL_n(K)$ be a representation that is absolutely
irreducible.
Let $\rho' : G \to \GL_n(A)$ be a representation.
Assume that for all $g$ in $G$, we have $\tr \rho(g) = \tr \rho'(g)$.

Then, Zariski-locally on $A$,
there is an $M \in \GL_n(A)$ such that $\rho'(g) = M \rho(g) M^{-1}$
for all $g\in G$.
\end{theo}

\begin{proof}
By \cite[Théorème 5.1]{Rou}, there is an $A$-algebra automorphism $\tau$
of $M_n(A)$ such that for all $g \in G$, $\rho'(g) = \tau\rho(g)$. 
By \cite[IV. Proposition 1.3]{KO}, there is a family $(f_i)$ in $A$ 
generating the unit ideal such that for all $i$, the automorphism of $M_n(A[1/f_i])$
induced by $\tau$ is inner. Hence the result.
\end{proof}

\subsubsection{Variations on Hilbert 90}

\begin{prop}
\label{h90fam}
Let $K$ be an infinite field, and $L/K$ be a finite Galois extension of
fields. 

\begin{enumerate}
\item
Let $M$ be a finite
$K$-algebra. Then $H^1(\Gal(L/K),(L\otimes_KM)^\times) = 0$.

\item
Let $A$ be a $K$-algebra. 
Assume that for every maximal ideal $\m$ of $A$, $A/\m$ is a finite
extension of $K$.
Let 
$c \in H^1(\Gal(L/K), (L\otimes_KA)^\times)$. 
There exists a family of elements 
$(f_i)$ in $A$ that generate the unit ideal such that the image
of $c$ in $H^1(\Gal(L/K), (L\otimes_KA[f_i^{-1}])^\times)$ is zero for
all $i$.

\end{enumerate}
\end{prop}

\begin{proof}
Let $M$ be a $K$-algebra, and 
$c \in H^1(\Gal(L/K), (L\otimes_KM)^\times)$.
Let $x\in L$. We set 
$\phi(c,x) =
\sum_{\gamma\in\Gal(L/K)}\gamma(x)c(\gamma) \in L\otimes_KM$. 
We have for all $g\in \Gal(L/K)$, $c(g)g(\phi(c,x)) = \phi(c,x)$, 
so $c = 0$ as soon
as we can find an $x$ such that $\phi(c,x)$ is
invertible in $L\otimes_KM$. Point (1) is well-known, and is proved by showing that if
$M$ is finite over $K$ then such an $x$ exists, with a proof similar to
the case where $M = M_n(K)$ (here we do not need $M$ to be commutative).

For any commutative $K$-algebra $M$, the $M$-algebra $L\otimes_KM$ is
finite. We denote by $N_M$ the norm map $L\otimes_KM \to M$, so that for
all $x\in L\otimes_KM$, we have $x\in (L\otimes_KM)^\times$ if and only
if $N_M(x) \in M^\times$.

Moreover the norm map commutes with base
change: let $u : M \to M'$ be a map of $K$-algebras,
then
$N_{M'}(1\otimes u)(x) = u(N_{M}(x))$ for all $x\in L\otimes_KM$.

Let now $A$ be as in point (2)
and let $c \in H^1(\Gal(L/K), (L\otimes_KA)^\times)$.
For an extension $A'$ of $A$, denote by $c_{A'}$ the image of $c$ in 
$H^1(\Gal(L/K), (L\otimes_KA')^\times)$.

Let $\m$ be a maximal ideal of $A$, and $K_{\m} = A/\m$. Then $K_{\m}$ is
a finite extension of $K$. So there exists an $x\in L$ such that
$\phi(c_{K_{\m}},x)$ is invertible in $L\otimes_KK_{\m}$. Let $f =
N_A(\phi(c,x)) \in A$. 
Then $D_f$ is a neighborhood of $\m$ in $\spec A$.
Moreover the image of $\phi(c,x)$ in $L\otimes_KA[f^{-1}]$ is
invertible, so $c_{A[f^{-1}]} = 0$. 

So we see that there is a covering of $\spec A$ by open subsets of the
form $D_f$ with $c_{A[f^{-1}]} = 0$, which is what we wanted.
\end{proof}

\subsection{The crystalline case}
\label{cryssect}

We want to prove Theorem \ref{parameter} for the case where the Galois
type is of the form (1), that is, $\tau = \chi \oplus \chi$ for some
smooth character $\chi$ of $I_{\Q_p}$ that extends to $W_{\Q_p}$. By
twisting by the character $\chi$, we can reduce to the case where $\tau$
is the trivial representation of $I_{\Q_p}$, that is, the case
of crystalline deformation rings. Recall from Section \ref{crysrings} the
definition of the parameter $a_p$.

\begin{prop}
\label{paramcrys}
There is an element $a_p \in R(k,\bar\rho)[1/p]$ such that for any finite
extension $E_x$ of $E$ and $x : R(k,\bar\rho)[1/p] \to E_x$ corresponding to
a representation $\rho_x$, $a_p(x)$ is the value of $a_p$ corresponding
to $\rho_x$ by the classification of Lemma \ref{isomclass}.

In particular, we can see $a_p$ as an analytic map from $\X(k,\bar\rho)$ to
$\mathbb{A}^{1,rig}$. 
Moreover, $a_p$ induces an injective map from $\X(k,\bar\rho)(\bar\Q_p)$
to $D(0,1)^+$.
\end{prop}

\begin{proof}
Consider the $\phi$-module $D$ which is
obtained from applying Theorem
\ref{modulest} to the rigid space $\X(k,\bar\rho)$ attached to 
the ring $R(k,\bar\rho)[1/p]$. It is a projective module of rank $2$ over 
$R(k,\bar\rho)[1/p]$ 
and is such that for all $x : R(k,\bar\rho)[1/p] \to E_x$
corresponding to a representation $\rho_x$, 
$D\otimes_{R(k,\bar\rho)[1/p]}E_x$ is the $\phi$-module 
$D_x$ attached to $\rho_x$ (forgetting the filtration). 
Now observe that $a_p$, as defined in Lemma \ref{isomclass}, 
is the trace of $\phi$ on the dual of 
$D$, so it is an element of $R(k,\bar\rho)[1/p]$, and $a_p(x)$ is the
evaluation at $x$ of the
trace of $\phi$ on the dual of $D$.
\end{proof}

\subsection{The crystabelline case}
\label{crystabsect}

We suppose here that $\tau = \chi_1 \oplus \chi_2$, where $\chi_1$ and
$\chi_2$ are distinct characters of $I_{\Q_p}$ with finite image that
extend to characters of $W_{\Q_p}$, so that the
representations classified by 
$R^{\psi}(k,\tau,\bar\rho)$ become crystalline on an abelian extension of
$\Q_p$. In this case we show the existence of a function $\lambda$ as in
Proposition \ref{parameter}
when $\chi_1 \neq \chi_2$. We make use of the results of 
\cite{GM}, which classifies the filtered $\phi$-modules with descent data
that give rise to a Galois representation of inertial type $\tau$ and
Hodge-Tate weights $(0,k-1)$. We summarize their results for such a
$\tau$. 

The characters $\chi_i$ factor through $F = \Q_p(\zeta_{p^m})$ for some
$m \geq 1$, so the Galois representations we are interested in become
crystalline on $F$, and so are given by filtered $(\phi,\GF)$-modules. Note that here $F_0 = \Q_p$.

Let $E$ be a finite extension of $\Q_p$ containing the values of $\chi_1$
and $\chi_2$.
Let $\alpha$, $\beta$ be in $\OO_E$
with $v_p(\alpha) + v_p(\beta) = k-1$. We define a $(\phi,\GF)$-module
$\Delta_{\alpha,\beta}$ as follows:
let $\Delta_{\alpha,\beta} = Ee_1 \oplus Ee_2$, with 
$g(e_1) = \chi_1(g)e_1$ and $g(e_2)= \chi_2(g)e_2$ for all
$g\in\Gal(F/\Q_p)$.
The action of $\phi$ is given by:
$\phi(e_1) = \alpha^{-1} e_1$ and $\phi(e_2) =
\beta^{-1} e_2$.
We are looking at filtrations on $\Delta_{\alpha,\beta,F} =
F\otimes_{\Q_p}\Delta_{\alpha,\beta}$ satisfying 
$\Fil^i\Delta_{\alpha,\beta,F} = 0$ if $i \leq 1-k$, 
$\Fil^i\Delta_{\alpha,\beta,F} = \Delta_{\alpha,\beta}$ if $i>0$, and
$\Fil^i\Delta_{\alpha,\beta,F} = \Fil^0\Delta_{\alpha,\beta,F}$ for $1-k < i
\leq 0$ is a $F\otimes_{\Q_p} E$-line.

We summarize now the results that are given in \cite[Section 3]{GM}.

\begin{prop}
Fix $\alpha$, $\beta$ in $\OO_E$ with $v_p(\alpha)+v_p(\beta)=k-1$. Then there
exists a way to choose $\Fil^{0}(\Delta_{\alpha,\beta,F}) 
\subset \Delta_{\alpha,\beta,F} = \Delta_{\alpha,\beta}\otimes F$ that
makes it an admissible filtered $(\phi,\GF)$-module.

If neither $\alpha$ nor $\beta$ is a unit, then all such choices give
rise to isomorphic filtered $(\phi,\GF)$-modules, which are
irreducible.

If $\alpha$ or $\beta$ is a unit, the choices give rise to two
isomorphism classes of filtered $(\phi,\GF)$-modules, one
being reducible split and the other reducible non-split.
\end{prop}

We denote by $D_{\alpha,\beta}$ the isomorphism class of
admissible filtered $(\phi,\GF)$-module given by a choice of
filtration that makes it into either an irreducible module (if neither
$\alpha$ nor $\beta$ is a unit) or a reducible non-split module (if 
$\alpha$ or $\beta$ is a unit).

Then it follows from the computations of \cite[Section 3]{GM} that:

\begin{prop}
Let $V$ be a potentially crystalline representation with coefficients in
$E$, of inertial type
$\tau$ and Hodge-Tate weights $(0,k-1)$ that is not reducible split. Then
there exists a unique pair $(\alpha,\beta) \in \OO_E$ with
$v_p(\alpha)+v_p(\beta)=k-1$ such that $\Dcr{F}(V)$ is isomorphic to
$D_{\alpha,\beta}$ as a filtered $(\phi,\GF)$-module.
\end{prop}

Let $E = E(k,\tau,\bar\rho,\psi)$ be a finite extension of $\Q_p$ such that
$\bar\rho$ can be defined over the residue field of $E$, $E$ contains the
images of $\chi_1$ and $\chi_2$ and of the character $\psi$. Then the
ring $R^\psi(k,\tau,\bar\rho)$ can be defined over $E$. Moreover:

\begin{prop}
\label{crystabfam}
Let $\bar\rho$ be a representation with trivial
endomorphisms. 
There are elements $\alpha,\beta\in R^{\psi}(k,\tau,\bar\rho)[1/p]$ such
that for each closed point $x$ of $\spec R^{\psi}(k,\tau,\bar\rho)[1/p]$
corresponding to a representation $\rho_x$,
$\Dcr{F}(\rho_x)$ is isomorphic to $\Delta_{\alpha(x),\beta(x)}$ as a
$(\phi,\GF)$-module.
\end{prop}

\begin{proof}
By Theorem \ref{modulest} applied to the rigid analytic space
$\X^\psi(k,\tau,\bar\rho)$ attached to $R^\psi(k,\tau,\bar\rho)[1/p]$, there exists 
a $\phi$-module $D$ with descent data by $\Gal(F/\Q_p)$,
where $D$ is a projective module of rank $2$ over
$R^\psi(k,\tau,\bar\rho)[1/p]$,
such that for each closed point $x$ of $\spec R^{\psi}(k,\tau,\bar\rho)[1/p]$,
$\Dcr{F}(\rho_x)$ is isomorphic to $D\otimes_RE_x$ (where $E_x$ is the
field of coefficients of $\rho_x$) as a $(\phi,\GF)$-module.

Applying Proposition \ref{diago}, we see that
the action of $\Gal(F/\Q_p)$ on $D$ is given as the action of
$\Gal(F/\Q_p)$ on each $\Delta_{\alpha,\beta}$: that is, Zariski-locally
on $\spec R^{\psi}(k,\tau,\bar\rho)[1/p]$, we can write $D =
Re_1\oplus Re_2$, with $g(e_1) = \chi_1(g)e_1$ and $g(e_2)=\chi_2(g)e_2$.

As the action of $\phi$ on $D$ commutes with the action of
$\Gal(F/\Q_p)$, this shows that the eigenvalues of $\phi$ acting on
$D$ are in fact in $R^{\psi}(k,\tau,\bar\rho)[1/p]$, that is,
$\alpha$ and $\beta$ are elements of $R^{\psi}(k,\tau,\bar\rho)[1/p]$.
\end{proof}

Moreover, if we fix the determinant of the Galois representation
corresponding to $D_{\alpha,\beta}$ then we 
fix $\alpha\beta$.  
So the function $\alpha$ is injective on points, 
so it can play the role of the function $\lambda$ of Theorem
\ref{parameter}.

Let $X^{\psi}(k,\tau,\bar\rho)$ be the image of
$\X^{\psi}(k,\tau,\bar\rho)(\bar\Q_p)$ in $\bar\Q_p$, then we see that 
$X^{\psi}(k,\tau,\bar\rho)$ is contained in the set $\{x, 0 \leq v_p(x)
\leq k-1\}$, with the irreducible representations corresponding to the
subset of elements that are in $\{x, 0 < v_p(x) < k-1\}$.

\subsection{Semi-stable representations}
\label{sstablesect}
We now assume $p>2$ and we study the case of the deformation rings attached to a discrete
series extended type
of the form $\tau = \chi_1 \oplus \chi_2$, where $\chi_1$ and $\chi_2$
are characters of $W_{\Q_p}$ that have the same restriction to inertia, and
such that $\chi_1(F) = p\chi_2(F)$ for any Frobenius element $F$. As in
the case of crystalline representations, we can twist by a smooth
character of $W_{\Q_p}$ and reduce to the case where $\chi_1$ and
$\chi_2$ are trivial on inertia. Then the deformation rings
$R^\psi(k,\tau,\bar\rho)$ classify representations that are semi-stable,
and only a finite number of the representations that appear can be
crystalline.

Let $\rho$ be a semi-stable, non-crystalline representation of dimension
$2$ of $G_{\Q_p}$, with Hodge-Tate weights $(0,k-1)$ for some $k\geq 2$.
Then we know (see for example \cite[Section 3.1]{GM}), that the filtered
$(\phi,N)$-module $\Dst{}(\rho)$ is isomorphic to exactly one
$D_{\alpha,\LL}$ for some $\alpha$ with $v(\alpha) = k/2$, some $\LL\in
\bar\Q_p$ and some finite extension $E$ containing $\alpha$ and $\LL$, for
$(\phi,N)$-modules $D_{\alpha,\LL}$ defined as follows:
$D_{\alpha,\LL} = Ee_1 \oplus Ee_2$, $\phi(e_1) = p\alpha^{-1}e_1$, 
$\phi(e_2) = \alpha^{-1}e_2$, $Ne_1 = e_2$, $\Fil^0D_{\alpha,\LL} =
E(e_1-\LL e_2)$. Then $\LL$ is the $\LL$-invariant of Fontaine, as defined
in \cite[\S 9]{Maz}. Let $\rho$
be a crystalline representation of dimension $2$ of $G_{\Q_p}$, we set
its $\LL$-invariant to be $\infty$.

\begin{prop}
\label{Linv}
Let $\X$ be a rigid analytic space defined over some finite extension $E$
of $\Q_p$. Assume that $\X$ is endowed with a $2$-dimensional
representation $\rho$ of $G_{\Q_p}$ such that for all $x \in \X$,
$\rho_x$ is semi-stable with Hodge-Tate weights $(0,k-1)$, the Weil
representation attached to $\rho_x$ is independent of $x$, there exists
at least one $x$ such that $\rho_x$ is not crystalline, and none of the
$\rho_x$ are reducible split.
Then
there exists a rigid analytic map $\LL : \X \to \PP^1_E$, defined over
$E$, such that for all $x$, $\LL(x)$ is the $\LL$-invariant of $\rho_x$.
\end{prop}

Note that under these conditions, the $\alpha$ of $D_{\alpha,\LL}$ is
independent of $x$, and is in $E$.

This proposition applies in the following situation: let $p>2$, let $\X =
\X^\psi(k,\tau,\bar\rho)$ be the deformation space for the extended type
$\tau$, and $\bar\rho$ is not reducible split. Then the function $\LL$ can
play the role of $\lambda$ of Proposition \ref{parameter}. 

\begin{proof}
In order to prove this result, it is enough to prove it for an admissible
covering of $\X$. Indeed, the condition that $\LL(x)$ is the
$\LL$-invariant of $\rho_x$ ensures that the functions defined on each
subset of the covering will glue. In particular, we can assume that $\X$
is affinoid, coming from a Tate algebra $A$ over $E$.

By Theorems \ref{modulest} and \ref{moduledR}, there is a projective $A$-module $D$ of
rank $2$ over $A$, endowed with a structure of filtered
$(\phi,N)$-module, such that for all $x \in \Max(A)$, $D_x$ is
$\Dst{}(\rho_x)$. Consider the action of $\phi$ on $D$: it has eigenvalues
$p\alpha^{-1}$ and $\alpha^{-1}$. By Proposition \ref{diago}, we can
assume, after replacing $A$ by a Zariski covering, that $D$ is free over
$A$, with a basis $e_1,e_2$ such that $\phi(e_1) = p\alpha^{-1}e_1$ and
$\phi(e_2) = \alpha^{-1}e_2$. By the commutation relations between $\phi$
and $N$, there is a $\lambda\in A$ such that $Ne_1 = \lambda e_2$.
Moreover, we can assume that there is a
free $A$-module $L$ of rank $1$ in $D$, with quotient that is also free
of rank $1$, that gives the non-trivial step of the filtration. We fix a
basis $f$ of $L$.

Let
$h = \det(f,\phi(f))$. Let us show that $N$ and $h$ do not vanish
simultaneously. If this is the case, let $x$ be a point where they both
vanish. Then $\rho_x$ is crystalline, as $N_x = 0$,
and the filtration of
the associated filtered $\phi$-module is generated by an eigenvector of
$\phi$, as $h_x= 0$. Then the representation $\rho_x$ is necessarily split reducible.
But by hypothesis this cannot happen.
So by replacing $\Max(A)$ by a Zariski cover, we can assume that either
$N$ never vanishes, or $h$ in a unit in $A$.

Assume first that $N$ never vanishes, that is, $\rho_x$ is never 
crystalline.
Then the $\lambda$ as defined above is
actually a unit in $A$, so we can modify the basis $(e_1,e_2)$ so that
$\lambda = 1$. 
Write $f$ in this basis as $ae_1 + be_2$,
with $a,b\in A$. By
specializing at each $x\in\Max(A)$, we see that $a(x) \neq 0$ for all
$x$, as this would contradict the admissibility condition of the filtered
module. So $a \in A^\times$. Then by definition of the $\LL$-invariant, we
have $\LL(x) = -(b/a)(x)$ for all $x \in \Max(A)$. So the function $\LL$ is
indeed an analytic function on $\Max(A)$.

Assume now that $h$ is a unit in $A$.
Let $(e_1,e_2)$ be the basis of $D$ defined above such that each $e_i$ is an
eigenvector for $\phi$.
We can write $f = ae_1 + be_2$ for some $a,b\in
A$. Then the condition on $h$ implies that $a$ and $b$ are in $A^\times$,
that is, $(ae_1,be_2)$ is also a basis of $D$ over $A$. So we can modify
the basis so that we have moreover $f = e_1+e_2$.
After specializing at $x \in \Max(A)$ an easy computation shows that
$\lambda(x) = -1/\LL(x)$ (and in particular the condition on $h$ implies
that $\LL$ does not take the value $0$).  So we have defined an analytic
function $\Max(A) \to \PP^1$ by taking $\LL = 1/\lambda$.
\end{proof}

\subsection{Supercuspidal types}
\label{sssect}

In this Section, assume that $p>2$.
We consider now the case where the type is supercuspidal, that is, the
Weil representation is (absolutely) irreducible.

\subsubsection{Defining the generalized $\LL$-invariant}

We fix once and for all a supercuspidal extended type $\tau$, that is, a
smooth absolutely irreducible representation $\tau : W_{\Q_p} \to \GL_2(E_0)$ for some
finite extension $E_0$ of $\Q_p$. This corresponds to cases (2) and (3) of the
classification of types of \cite[Lemma 2.1]{GM}. Note that we can take
$E_0$ to be an unramified extension of the definition field of $\tau$ by
Lemma \ref{unramW}.

Let $F$ be a finite Galois extension of $\Q_p$ such that $\tau$ is
trivial on $I_F$, and let $F_0$ be the maximal unramified extension of
$\Q_p$ contained in $F$. We assume, after taking an unramified extension
of $E_0$ if necessary, that $F_0 \subset E_0$.

Let $\DcrO$ be the $(\phi,\GF)$-module corresponding
to $\tau$ via the correspondence of Proposition \ref{Weil}. Let
$\DdRO = F\otimes_{F_0}\DcrO$. It is endowed with an action of
$\GF$ coming from the one on $\DcrO$. Then:

\begin{lemm}
\label{dim2}
Assume that there exists as least one potentially crystalline
representation $\rho$ with coefficients in $E$ for some finite extension
$E$ of $E_0$, such that $\Ddr{F}(\rho)$ is isomorphic to 
$\DdRO^{\GF}\otimes_{E_0}E$ as a $F\otimes_{\Q_p}E$-module with an action
of $\GF$.
Then $\DdRO^{\GF}$ is an $E_0$-vector space of dimension $2$.
\end{lemm}

\begin{proof}
Let $D = \DdRO\otimes_{E_0}E$, with its action of $\GF$, which is
isomorphic to the $\phi$-module $\Ddr{F}(\rho)$ with its action of
$\GF$ for some potentially crystalline representation $\rho$. Then 
$\Ddr{F}(\rho)^{\GF} = \Ddr{\Q_p}(\rho)$ is an $E$-vector space of
dimension $2$, as $\rho$ is de Rham as a $G_{\Q_p}$-representation.
The action of $\GF$ on $\DdRO$ is $E_0$-linear. So the dimension of
its subspace of fixed elements is invariant by extension of scalars.
Hence the result.
\end{proof}

\begin{rema}
We could also make use of the results of \cite{GM}, which give an
explicit basis of the $E$-vector space $(\DdRO\otimes_{E_0}E)^{\GF}$ for some
extension $E$ of $E_0$. 
\end{rema}

We denote by $V_{\tau}$ the $E_0$-vector space of dimension $2$ given 
by Lemma \ref{dim2}.

Any potentially semi-stable representation of
extended type $\tau$ becomes crystalline when restricted to $G_F$. 
For any such representation $\rho$, with coefficients in an extension $E$
of $E_0$, $\Dcr{F}(\rho)$ is a
$(\phi,\GF)$-module over $F_0\otimes_{\Q_p} E$. We
have that $\Ddr{F}(\rho)$ is canonically isomorphic to $F \otimes_{F_0}
\Dcr{F}(\rho)$, and is endowed with an admissible filtration. Moreover,
$\Ddr{F}(\rho)^{\GF} = \Ddr{\Q_p}(\rho)$ is an $E$-vector space of
dimension $2$.

We also fix an integer $k \geq 2$, a continuous character
$\psi: G_{\Q_p} \to E_0^\times$. Note that there is no loss of generality
in considering only characters with values in $E_0$, as the
compatibility condition between type and determinant shows that if
$R^\psi(k,\tau,\bar\rho)$ is non-zero then $\psi$ takes its values in $E_0$.

Let $\E_{\tau}$ be the set of Galois
representations $\rho : G_{\Q_p} \to \GL_2(\bar{\Q}_p)$ that are
potentially crystalline of extended type $\tau$, Hodge-Tate weights
$(0,k-1)$, and determinant $\psi$. Then:

\begin{theo}
\label{Lgen}
There exists a map $\ltau : \E_{\tau} \to
\PP(V_{\tau}\otimes_{E_0}\bar{\Q}_p)$ such that two elements $\rho$, $\rho'$ of $\E_{\tau}$
are isomorphic if and only if $\ltau(\rho) = \ltau(\rho')$.
\end{theo}

\begin{proof}
We can assume that $\E_\tau$ is not empty, otherwise the statement is
trivially true.
Let $\rho : G_{\Q_p} \to \GL_2(\bar\Q_p)$ be an element of $\E_{\tau}$. Then
$\WD(\rho)$, the Weil-Deligne representation attached to $\rho$, is
actually a Weil representation as $\rho$ is potentially crystalline. By
definition, $\WD(\rho)$ is isomorphic to $\tau\otimes_{E_0}\bar\Q_p$ as a
representation of $W_{\Q_p}$. We fix such an isomorphism $u$, it is
unique up to a scalar by the irreduciblity of $\tau$. Then $u$ gives us
an isomorphism between 
$\Dcr{F}(\rho)$ and $\DcrO\otimes_{E_0}\bar\Q_p$ 
as $\phi$-modules with an action of $\GF$, by Proposition \ref{Weil}.
This also gives us an isomorphism, that we still call $u$, between
$\Ddr{F}(\rho)$ and $\DdRO\otimes_{E_0}\bar\Q_p$. 

The isomorphism class of $\rho$ is entirely determined by the filtration
on $\Ddr{F}(\rho)$. As the Hodge-Tate weights of $\rho$ are known, the
only necessary information is the $F\otimes_{\Q_p}\bar\Q_p$-line corresponding
to the non-trivial steps of the filtration. This line is invariant by the
action of $\GF$. By the isomorphism $u$, this gives rise to a
$\GF$-invariant $F\otimes_{\Q_p}\bar\Q_p$-line in
$\DdRO\otimes_{E_0}\bar\Q_p$. This line is generated by an element of
$\DdRO\otimes_{E_0}\bar\Q_p$ that is invariant by $\GF$ by (1) of Proposition
\ref{h90fam}, hence by an element of $\DdRO^{\GF}\otimes_{E_0}\bar\Q_p$.

We define $\ltau(\rho)\in \PP(\DdRO^{\GF}\otimes_{E_0}\bar\Q_p)$ 
to be the line generated by this element in 
$\DdRO^{\GF}\otimes_{E_0}\bar\Q_p$.
This does not depend on the choices made, as $u$ is unique up to
multiplication by a scalar, and the invariant element generating the line
is well-defined up to multiplication by a scalar.
\end{proof}

\subsubsection{Making it into an analytic function}

Let $\X$ be the rigid analytic space corresponding to the deformation
ring $R^{\psi}(k,\tau,\bar\rho)$ for some representation $\bar\rho$ with
trivial endomorphisms and some supercuspidal extended type $\tau$.
Let $E = E(k,\tau,\bar\rho,\psi)$ be the field $E_0$ defined above.

\begin{prop}
\label{Lgenan}
There exists a rigid analytic map $\ltau : \X \to \PP(V_{\tau})$, defined over
$E$, such that for all $x$, $\ltau(x)$ is the $\ltau$-invariant of $\rho_x$
as defined in Theorem \ref{Lgen}.
\end{prop}

By fixing a basis of the $2$-dimensional $E$-vector space $V_{\tau}$, we
then get a map $\ltau: \X \to \PP^1_E$, which plays the role of $\lambda$
in Theorem \ref{parameter}.

\begin{proof}
It is enough to do this on an admissible covering of $\X$ by affinoid
subspaces. So we can assume that $\X = \Max(A)$ for some affinoid algebra
$A$, and replace $\X$ by an admissible covering by affinoid subspaces as
needed.

Let $\Dcr{F}(A)$ be the $(\phi,\GF)$-module
corresponding to the representation $\rho$. We can assume that $\Dcr{}(A)$ is a
free $A$-module of rank $2$. Using the correspondence between
$(\phi,\GF)$-modules and representations of the Weil
group as in Section \ref{Weil}, and Theorem \ref{familyrepr}, we can assume
that $\Dcr{F}(A) = \DcrO^F\otimes_{E}A$ as a $(\phi,\GF)$-module over
$F_0\otimes_{\Q_p}A$.

Consider now $\Ddr{F}(A)$. It is isomorphic to $F\otimes_{F_0}\Dcr{F}(A)$,
so to $\DdRO^F\otimes_{E}A$ as a $\phi$-module with action of $\GF$.
In particular, it is trivial as an $F\otimes_{\Q_p}A$-module with an
action of $\GF$. Also, it has a basis as an $A$-module given by the
chosen basis of $\DdRO^F$.
The module
$\Ddr{F}(A)$ contains a locally free sub-$F\otimes_{Q_p}A$-module $\f$ of
rank $1$, such that $\Ddr{F}(A)/\f$ is also locally free of rank $1$, that
gives at each point $x$ the filtration on $\Ddr{F}(\rho_x)$. We can assume
that $\f$ and $\Ddr{F}(A)$ are free of rank $1$ over $F\otimes_{\Q_p}A$. 
Moreover, this submodule is invariant by the action of $\GF$. Consider
a basis $f$ of $\f$. Then the action of $\GF$ on $f$ gives rise to an
element $c \in H^1(\GF,(F\otimes_{\Q_p}A)^\times)$. Using Theorem
\ref{h90fam} and replacing $\Max(A)$ by an admissible covering if
necessary, we can assume that $f$ itself is fixed by the action of
$\GF$. 

So we get that $f$ is in $\Ddr{F}(A)^{\GF}$, which is canonically
isomorphic to $\DdRO^{\Q_p}\otimes_EA$. So $f$ defines an analytic
map over $\Max(A)$ with values in $\PP(\DdRO^{\Q_p}) = \PP(V_\tau)$, which is what we wanted.

\end{proof}

\end{document}